\documentclass[11pt,reqno]{amsart}
\usepackage{amscd}
\usepackage{amsfonts}
\usepackage{amsmath}
\usepackage{amssymb}
\usepackage{amsthm}
\usepackage{fancyhdr}
\usepackage{latexsym}
\usepackage[colorlinks=true,pdfstartview=FitV,linkcolor=blue,citecolor=blue, urlcolor=blue]{hyperref}
\usepackage{enumerate}          
\usepackage{enumitem}            
            
\usepackage{upgreek}
\usepackage{times}

\DeclareSymbolFont{newfont}{OML}{cmm}{m}{it}
\DeclareMathSymbol{\Varepsilon}{3}{newfont}{15}
\DeclareMathSymbol{\Varrho}{3}{newfont}{37}

\newtheorem*{theorem*}{Theorem}
\newtheorem*{lemma*}{Lemma}

\usepackage{caption} 
\usepackage{subcaption} 
\usepackage{cancel}

\input epsf
\input texdraw
\input txdtools.tex
\input xy
\xyoption{all}

\advance\voffset-8mm

\def\paragraph#1{\par\medskip\noindent\textbf{#1}}
\def\wr{\mathop{\rm wr}\nolimits}
\let\eref=\eqref
\def\bse{\begin{subequations}}
\def\ese{\end{subequations}}
\def\Complex{{\mathbb C}}
\def\Real{{\mathbb R}}
\def\Re{{\mathrm{Re}}}
\def\Im{{\mathrm{Im}}}
\def\re{{\mathrm{re}}}
\def\im{{\mathrm{im}}}
\def\asymp{{\mathrm{asymp}}}
\def\err{{\mathrm{err}}}

\makeatletter
\def\overl@ss#1#2{\vcenter{\offinterlineskip
        \ialign{$\m@th#1\hfil##\hfil$\crcr#2\crcr<\crcr } }}
\def\gl{\mathrel{\mathpalette\overl@ss>}}
\makeatother

\def\@#1{{\mathbf{#1}}}

\usepackage{color,soul}

\usepackage{xcolor}

\usepackage{color}

\definecolor{Red}{rgb}{1.00, 0.00, 0.00}
\definecolor{DarkGreen}{rgb}{0.00, 1.00, 0.00}
\definecolor{Blue}{rgb}{0.00, 0.00, 1.00}
\definecolor{Cyan}{rgb}{0.00, 1.00, 1.00}
\definecolor{Magenta}{rgb}{1.00, 0.00, 1.00}
\definecolor{DeepSkyBlue}{rgb}{0.00, 0.75, 1.00}
\definecolor{DarkGreen}{rgb}{0.00, 0.39, 0.00}
\definecolor{SpringGreen}{rgb}{0.00, 1.00, 0.50}
\definecolor{DarkOrange}{rgb}{1.00, 0.55, 0.00}
\definecolor{OrangeRed}{rgb}{1.00, 0.27, 0.00}
\definecolor{DeepPink}{rgb}{1.00, 0.08, 0.57}
\definecolor{DarkViolet}{rgb}{0.58, 0.00, 0.82}
\definecolor{SaddleBrown}{rgb}{0.54, 0.27, 0.07}
\definecolor{Black}{rgb}{0.00, 0.00, 0.00}
\definecolor{dark-magenta}{rgb}{.5,0,.5}
\definecolor{myblack}{rgb}{0,0,0}
\definecolor{darkgray}{gray}{0.5}
\definecolor{lightgray}{gray}{0.75}

\usepackage[pdftex]{graphicx}
\usepackage{epstopdf}

\newcommand{\p}{\partial}

\DeclareMathSizes{12}{12}{7}{5}

\def\refer #1\par{\noindent\hangindent=\parindent\hangafter=1 #1\par}

\renewcommand{\theequation}{\thesection.\arabic{equation}}

\theoremstyle{plain}  
\newtheorem{theorem}{Theorem}[]

\newtheorem{lemma}{Lemma}[section]
\newtheorem{corollary}{Corollary}[section]
\theoremstyle{definition}

\newtheorem{rhp}{Riemann-Hilbert Problem}[section]
\newtheorem{remark}{Remark}[section]
\numberwithin{figure}{section}
\numberwithin{equation}{section}

\begin{document}

\title{Long-time asymptotics for the focusing\\
nonlinear Schr\"odinger equation with nonzero \\
boundary conditions at infinity and\\
asymptotic stage of modulational instability}

\author{Gino Biondini \& Dionyssios Mantzavinos}



\begin{abstract}
The long-time asymptotic behavior of the focusing nonlinear Schr\"odinger (NLS) equation on the line with symmetric nonzero  boundary conditions   at infinity is characterized by using the recently developed inverse scattering transform (IST)  for such problems \cite{bk2014} and by employing the nonlinear steepest descent method of Deift and Zhou for oscillatory Riemann-Hilbert problems \cite{dz1993}.
First, the  IST is formulated over a single sheet of the complex plane without introducing the uniformization variable that was used in \cite{bk2014}.
The solution of the focusing NLS equation with nonzero  boundary conditions is thus associated with a suitable matrix Riemann-Hilbert problem   whose jumps grow exponentially with time for certain portions of the continuous spectrum. 
This growth is the signature of the well-known modulational instability within the context of the IST.
This growth is then removed by suitable deformations of the Riemann-Hilbert problem in the complex spectral plane. 
Asymptotically in time, the $xt$-plane is found to decompose into two types of regions: 
a left far-field region and a right far-field region, where the solution  equals the condition at infinity to leading order up to a phase shift, and a central region in which the asymptotic behavior is described by slowly modulated periodic oscillations.  In the latter region, it is also shown that the modulus of the leading order solution, which is initially obtained in the form of a ratio of Jacobi theta functions, eventually reduces to the well-known elliptic solution of the focusing NLS equation.
These results provide the first characterization of the long-time behavior of generic perturbations of a constant background in a modulationally unstable medium.
\end{abstract}

\date{December 17, 2015}

\maketitle

\markboth{Long-time asymptotics for focusing NLS with nonzero  boundary conditions and asymptotic stage of modulational instability}
                    {G. Biondini \& D. Mantzavinos}



\section{Introduction}

The present work is devoted to the study of the long-time asymptotic behavior of the one-dimensional focusing nonlinear Schr\"odinger (NLS) equation formulated on the line with symmetric, nonzero   boundary conditions at infinity:
%
\begin{gather}
i q_t + q_{xx} + 2\left(|q|^2-q_{o} ^2\right)q=0,
\label{e:NLS}
\\
\noalign{\noindent with $q=q(x,t)$ complex-valued, $(x, t)\in\mathbb R\times \mathbb R^+$,
$q(x, 0)$ given, and with}
 \lim_{x\to\pm\infty} q(x,0) = q_\pm\,.
\label{e:NZBC}
\end{gather}
%
Herafter, $q_\pm$ are complex constants with $\left|q_\pm\right|=q_{o} >0$, and
\begin{equation}
q(x,0)-q_\pm\in  L^{1, 1}(\mathbb R^\pm),
\label{e:q0L1}
\end{equation}
where
\begin{equation}
L^{1, 1}(\mathbb R^\pm)
=
\Big\{ f: \mathbb R\to \mathbb C\, \Big| \int_{\mathbb R^\pm}\left(1+ |x| \right)\left|f(x)\right|dx<\infty\Big\}.
\end{equation}
The NLS equation is an important model in applied mathematics and theoretical physics due to both its surprisingly rich mathematical structure and to its physical significance and broad applicability to a number of different areas ranging from nonlinear optics to water waves and from  plasmas to Bose-Einstein condensates.
In particular, recall that the NLS equation is a universal model for the evolution of almost monochromatic waves in a weakly nonlinear dispersive medium
\cite{bn1967,ce1987}.
%
Specifically, 
the equation was derived in the context of nonlinear optics \cite{cgt1964,t1964},
in the framework of fluid mechanics for water waves of small amplitude over infinite depth 
\cite{z1968} 
and finite depth \cite{br1969,ho1972}.
A rigorous justification of the model in the latter case was also given in \cite{css1992}.
The equation has also been suggested as a model for rogue waves \cite{per1983,cha2011}.
Further references on the physical aspect of the NLS equation can be found in \cite{ss-book,l2013}. 
Moreover,
the NLS equation is also one of the prototypical infinite-dimensional completely integrable systems,
\cite{zs1972} where a Lax pair was derived, 
and it was shown that the equation admits an infinite number of conservation laws as well as exact $N$-soliton solutions for arbitrary $N$. 

The initial value problem for the NLS equation has been studied extensively during the past five decades, 
and a plethora of results is available.
Assuming sufficient smoothness of the initial data, in 1972 Zakharov and Shabat \cite{zs1972} developed 
the inverse scattering transform (IST) for the solution of the initial value problem on the line for initial conditions with sufficiently rapid decay at infinity.
In 1974 Ablowitz, Kaup, Newell and Segur \cite{akns1973} generalized those results by introducing the so-called AKNS system, 
which contains the NLS equation as a special case. 
The IST for first order systems was also developed rigorously by Beals and Coifman in 1984 \cite{bc1984}.
The periodic problem for NLS was studied by Its and Kotlyarov in 1976 \cite{ik1976},
while the half-line and  the finite interval problems were analysed via an appropriate extension of IST for bounded domains by Fokas and others \cite{fi2004,fis2005}.
Problems with \textit{nonzero } boundary conditions at infinity of the kind~\eqref{e:NZBC} have also been studied.
The IST for the \textit{defocusing} NLS equation on the line with nonzero  boundary conditions at infinity was developed soon after the case of zero boundary conditions \cite{zs1973,ft1987}.
The IST for  the \textit{focusing} NLS equation  on the line with nonzero  boundary conditions  at infinity was recently developed by Kova\v{c}i\v{c} and the first author \cite{bk2014}.
(Partial results were contained in \cite{k1977, m1979, gk2012}.) 
%
From a different point of view,
sharp  well-posedness of  the NLS equation on the line with initial data in Sobolev spaces $H^s$ for any $s\ge0$ was proved by  Bourgain  \cite{Bourgain1993} (see also \cite{Bourgain1999}). 
Well-posedness of the NLS equation on the half-line with data in Sobolev spaces was established independently and via different approaches by Holmer  \cite{h2005},
Bona, Sun and Zhang  \cite{bsz2015} and Fokas, Himonas and the second author \cite{fhm2015}.
Further functional analysis results can be found in 
Craig,  Kappeler and Strauss  
\cite{cks1995},
Cazenave
\cite{Cazenave2003}, 
Cazenave and Weissler
\cite{cw2003}, 
Kenig, Ponce and Vega \cite{kpv1991},
Ghidaglia and Saut
\cite{gs1993},
Linares and Ponce
\cite{LP2009},
Carroll and Bu 
\cite{cb1991},
Kato
\cite{Kato1995},
Ginibre and Velo
\cite{gv1979},
Tsutsumi
\cite{Tsutsumi1987},
and the references therein.

The motivation for the present work stems from its intriguing connection with the physical phenomenon of modulational instability,
(also known as Benjamin-Feir instability in the context of deep water waves \cite{bf1967}),
namely, the instability of a constant background with respect to long wavelength perturbations.
Modulational instability is one of the most ubiquitous phenomena in nonlinear science (e.g., see \cite{zo2009} and the references therein).
In many cases the dynamics of systems affected by modulational instability is governed by the one-dimensional focusing NLS equation.
Hence, the initial (i.e., linear) stage of  modulational instability can be studied by linearizing  the focusing NLS equation   around the constant background. It is then easily seen that all Fourier modes below a certain threshold are unstable and the corresponding perturbations grow exponentially.
However, the linearization ceases to be valid as soon as perturbations  become comparable with the background.
A natural question  is  therefore what happens beyond this linear stage. 
Surprisingly, this question, which is referred to as \textit{the nonlinear stage of modulational instability}, 
has remained essentially open for about fifty years.
  
Since the focusing NLS equation is a completely integrable system,
a natural conjecture 
(by analogy with the case of localized initial conditions)
is that modulational instability is mediated by solitons~\cite{{zg2013},{gz2014}}.
On the other hand, 
Fagerstrom and the first author \cite{bf2015} 
 employed the recently developed IST for the focusing NLS equation with nonzero  boundary conditions at infinity 
in order to study modulational instability, 
by computing the spectrum of the scattering problem for simple classes of perturbations of a constant background.
In particular, it was shown in \cite{bf2015} that there exist classes of perturbations for which no solitons are present.
Therefore, since all generic perturbations of the constant background are linearly unstable, 
\textit{solitons cannot be the mechanism that mediates modulational instability},
contradicting the conjecture of \cite{zg2013}.
Instead, in \cite{bf2015} the instability mechanism was identified within the context of IST and it was shown that the instability emerges 
from certain portions of the continuous spectrum of the scattering problem associated with the focusing NLS equation.
At the same time, \cite{bf2015} did not offer any insight about the actual nonlinear dynamics of the solutions of the focusing 
NLS equation past the linear stage of modulational instability.

The aim of the present work is to address this challenge and characterize the nonlinear stage of modulational instability. 
We do so by computing the long-time asymptotic behavior  
of the solution of the focusing NLS equation. 
Recall that formal but ingenious results for the long-time asymptotics of the NLS equation with zero boundary conditions at infinity were obtained
in 1976 by Segur and Ablowitz~\cite{sa1976} and Zakharov and Manakov \cite{zm1976}.
In 1993, Deift and Zhou \cite{dz1993} introduced a method for the \textit{rigorous} asymptotic analysis of oscillatory Riemann-Hilbert problems
such as those arising in the inverse problem for the solution of integrable evolution equations via IST.
The Deift-Zhou method can be regarded as the nonlinear analogue of the classical steepest descent method used in the analysis of the long-time asymptotics for linear evolution equations.
The method relies on appropriate factorizations of the jump matrices of the Riemann-Hilbert problem 
and suitable deformations of the associated jump contours 
in order to extract the leading order asymptotic behavior as well as obtain rigorous estimates for the corrections.

The Deift-Zhou method was first applied to compute the long-time asymptotics of the modified KdV equation \cite{dz1993}, 
and was subsequently extended and applied in numerous works, 
including the study of the long-time asymptotics of the KdV equation \cite{dvz1994}, 
of the defocusing NLS equation \cite{dz1995} and the Toda lattice \cite{k1993}, 
all with decaying data at infinity, 
as well as the analysis of the zero dispersion (semiclassical) limit of the KdV equation \cite{dvz1997} and the focusing NLS equation 
\cite{{mk1998},{tvz2004}}.
We also note that the Deift-Zhou method has found useful applications in the theory of orthogonal polynomials \cite{{d1999},{dkmvz1999}}.
Recent results concerning the NLS equation were also presented in the works of Buckingham and Venakides \cite{bv2007}, Boutet-de Monvel, Kotlyarov and Shepelsky \cite{bks2011}, Jenkins and McLaughlin \cite{jm2013} and Jenkins \cite{j2014}.
On the other hand, none of those works studied the problem considered in the present work, namely, 
the long-time asymptotics of the solution of the focusing NLS equation~\eqref{e:NLS}
with \textit{generic} initial conditions satisfying~\eqref{e:NZBC} and~\eqref{e:q0L1}.
Our results are summarized by the following theorems:
\begin{theorem}[Plane wave region]\label{pw-t}
Let $q(x,0)$ satisfy~\eqref{e:NZBC}, ~\eqref{e:q0L1} and~\eqref{an-cond} and be such that 
no discrete spectrum is present.
For $x<-4\sqrt 2 q_{o}  \, t$, the solution $q(x,t)$ of the focusing NLS equation~\eref{e:NLS} 
equals
\begin{equation}\label{qsol-pw-t}
q(x, t)
=
e^{2i g_\infty}q_-
+O\big(t^{-\frac 12}\big),
\quad t\to \infty,
\end{equation}
where the real quantity $g_\infty $ is defined by equation \eqref{ginfdef} and depends only on the initial datum $q(x, 0)$ and  the ratio $x/t$ through the stationary point $k_1$ defined by equation \eqref{kroots}.

For $x>4\sqrt 2 q_{o} \, t$, the leading order asymptotic behavior of the solution $q$ is given by a formula analogous to formula \eqref{qsol-pw-t}, with $q_-$ replaced by $q_+$ and with $k_1$ in the definition of $g_\infty $  replaced by the stationary point $k_2$, also defined by equation \eqref{kroots}. 
\end{theorem}

\begin{theorem}[Modulated elliptic wave region]\label{mew-t}
Let $q(x,0)$ satisfy the hypotheses of Theorem~\ref{pw-t}.
For $-4\sqrt 2 q_{o} \, t<x<0$, the solution of the focusing NLS equation~\eref{e:NLS} 
equals 
\begin{align}\label{qsol-mew-t}
q(x, t)
&=
\frac{q_{o} \left(q_{o} +\alpha_\im\right) }{\bar q_-}
\
\frac{\Theta\left(\frac 12 \right)
\Theta\Big(\frac{\sqrt{\alpha_\re^2+\left(q_{o} +\alpha_\im\right)^2}}{2K(m)}
\left(x-2\alpha_\re t\right)
-
X
+2v_\infty -\frac 12 \Big)}
{
\Theta\left(2v_\infty -\frac 12\right)
\Theta\Big(\frac{\sqrt{\alpha_\re^2+\left(q_{o} +\alpha_\im\right)^2}}{2K(m)}
\left(x-2\alpha_\re t\right)
-X-\frac 12 \Big)}
\,
e^{2i\left(g_\infty -G_\infty  t\right)}
\nonumber\\
&\hskip 8.65cm
+
O\big(t^{-\frac 12}\big),
\quad  t\to \infty,
\end{align}
where $\Theta$ is defined in terms of the third Jacobi theta function by equation \eqref{thetjdef},  the real quantities $\alpha_\re,$ $\alpha_\im$ are given by equations \eqref{ak0system}, the real quantities $G_\infty $, $g_\infty $ are defined by equations  \eqref{Ginf} and~\eqref{ginfr2} respectively, the complex quantity $v_\infty $ is defined by equation \eqref{vdefr}, $K(m)$ is the complete elliptic integral of the first kind defined by equation \eqref{km-def} with elliptic modulus $m$ given by equation~\eqref{ell-mod}, and the real quantity $X$ is defined by
\begin{equation}\label{Xdef}
X = \frac{1}{2\pi}\left[\omega - i\ln\Big(\frac{q_-}{q_{o} }\Big)\right] + \frac14
\end{equation}
with the real quantity $\omega$  given by equation \eqref{omr}.
Moreover, except for the explicit dependence on $x$ and $t$ in~\eref{qsol-mew-t},
all of the above quantities depend only on the initial datum $q(x, 0)$ and  the ratio $x/t$.

For $0<x<4\sqrt 2 q_{o} \, t$, the leading order asymptotic behavior of the solution $q$ is given by a formula analogous to formula \eqref{qsol-mew-t}, with $\bar q_-$ replaced by $\bar q_+$ and  the remaining quantities modified accordingly.
\end{theorem}

\begin{theorem}[Elliptic representation]\label{mew-ell-t}
The modulus of the leading order asymptotic solution \eqref{qsol-mew-t} in the modulated elliptic wave region can be expressed in terms of the Jabobi elliptic function $\textnormal{sn}$ via the formula 
\begin{equation}\label{sqmod-finn-t}
\left|q_\mathrm{asymp}(x, t) \right|^2
=
\left(q_{o} +\alpha_\im\right)^2 
-
4q_{o}  \alpha_\im 
 \,
 \textnormal{sn}^2
\Big(
\sqrt{\alpha_\re^2+\left(q_{o} +\alpha_\im\right)^2}
\left(x-2\alpha_\re t\right)
-
2K(m) X, 
 m \Big),
\end{equation}
where  the real quantities $\alpha_\re$, $\alpha_\im$, $K(m)$, $m$ and $X$ are  defined as in Theorem \ref{mew-t}.
\end{theorem}

Importantly, equation
\eqref{sqmod-finn-t} rigorously establishes the validity of a similar expression formally obtained by Kamchatnov (see~\cite{Kam-book}, p.~234, equation (5.29))
using Whitham's modulation theory.
(The motion of the branch points for said solution was also formally derived in \cite{el1993} using similar methods.)
Moreover, equation \eqref{sqmod-finn-t} generalizes the results of \cite{Kam-book} and \cite{el1993}
to the case of \textit{generic} initial conditions satisfying the hypotheses of Theorems~\ref{pw-t} and~\ref{mew-t}.

Plots of $|q(x,t)|$ against $x/t$ and $t$ in the modulated elliptic wave region, 
obtained using the elliptic representation \eqref{sqmod-finn-t} are provided in Figures \ref{qmod_3d_plot} and \ref{qmod_plot}.

This paper is organized as follows.
The IST for the initial value problem for the focusing NLS equation~\eqref{e:NLS} is formulated in Section~\ref{ist-sec}. 
An outline of the derivation of the long-time asymptotics of this problem is then provided in Section \ref{outline-sec}, where it is shown that the $xt$-plane needs to be decomposed into two regions, namely a plane wave region and a modulated elliptic wave region.
Theorems \ref{pw-t} and \ref{mew-t} are subsequently established in Sections \ref{pw-sec} and \ref{mew-sec}, respectively, while Theorem \ref{mew-ell-t} is proved in Section \ref{ell-sec}.
Concludings remarks and future directions are discussed in Section \ref{conc-sec}.
Finally, a rigorous estimation of the leading order error is provided in the Appendix.

\begin{figure}[t!]
\begin{center}
\hskip -2.3cm
\begin{subfigure}{0.4\textwidth}
\includegraphics[height=4cm, width=6.8cm]{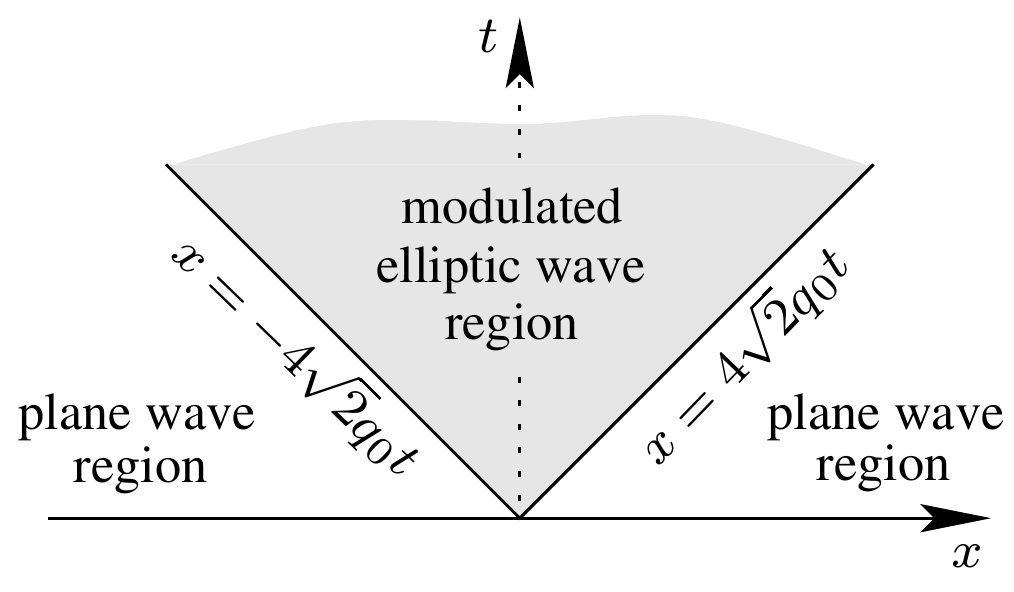}
\end{subfigure}
\hskip .7cm
\begin{subfigure}{0.4\textwidth}
\includegraphics[height=5.2cm, width=7.9cm]{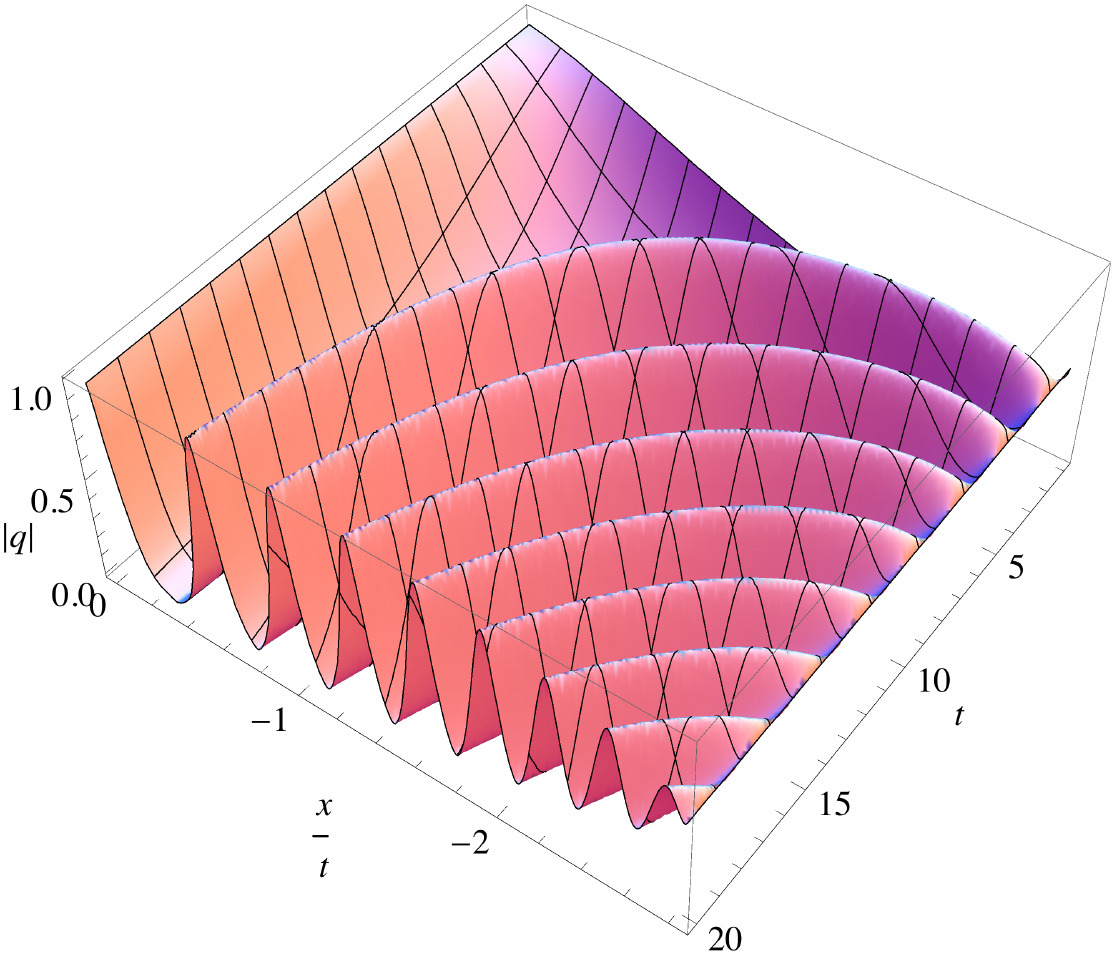}
\end{subfigure}
\caption{
Left: The plane wave and modulated elliptic wave regions of the $xt$-plane.
Right: The leading order modulus $|q(x,t)|$ in the modulated elliptic wave region,
as given by equation~\eqref{sqmod-finn-t}, for $q_{o} =\frac 12$ and  
$(x/t, t)\in \left[-2\sqrt 2, 0\right]\times \left[1, 20\right]$.}
\label{qmod_3d_plot}
\end{center}
\vglue1ex
\begin{center}
\hskip -2.5cm
\begin{subfigure}{0.3\textwidth}
\includegraphics[height=4.2cm, width=6.9cm]{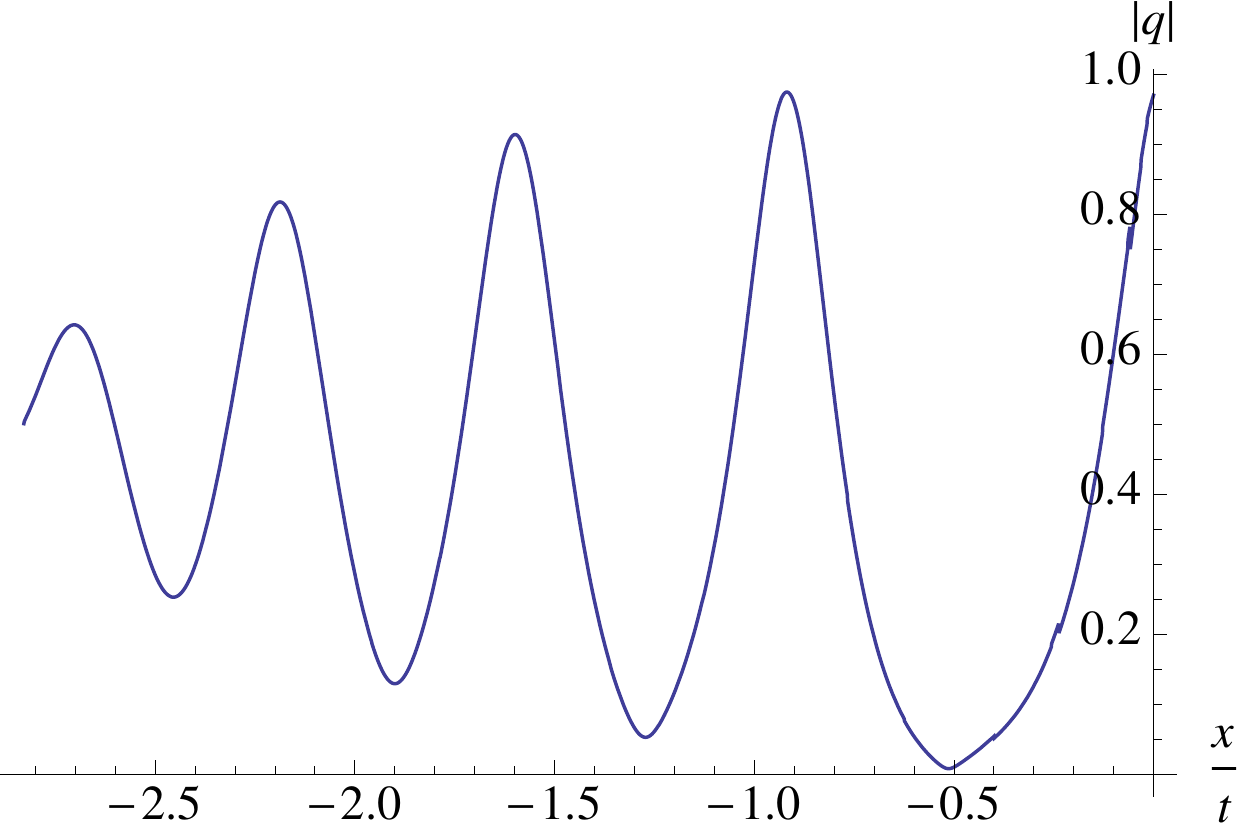}
\end{subfigure}
\hskip 3cm
\begin{subfigure}{0.3\textwidth}
\includegraphics[height=4.2cm, width=6.9cm]{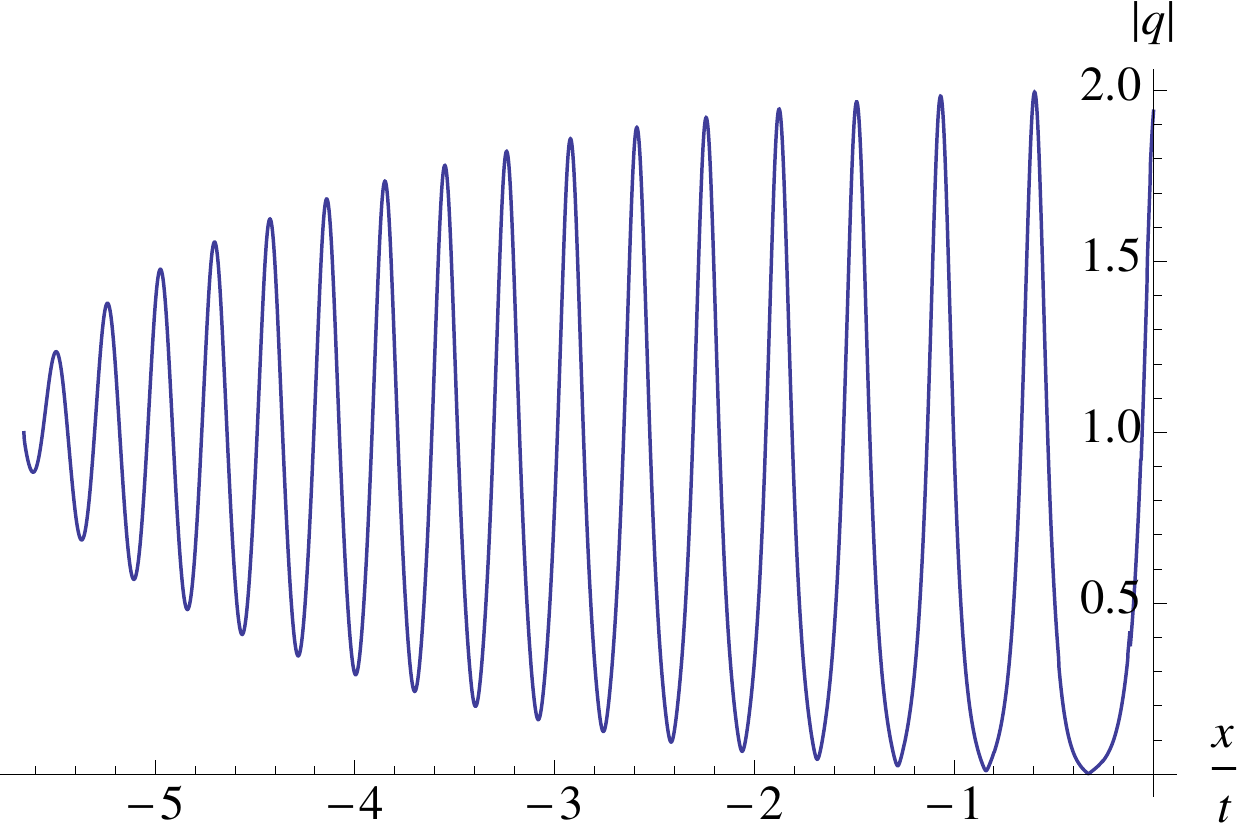}
\end{subfigure}
\caption{The leading order modulus $|q(x,t)|$ in the modulated elliptic wave region, 
as given by equation~\eqref{sqmod-finn-t}, as a function of $ x/t \in [-4\sqrt 2 q_{o}, 0]$.
Left: For $q_{o} =\frac12$ and $t=10$.
Right: For $q_{o} =1$ and $t=10$.}
\label{qmod_plot}
\end{center}
\end{figure}
%

%
\section{Inverse scattering transform with nonzero  boundary conditions}
\label{ist-sec}

In this section we formulate the IST for the initial value probelm for the NLS equation~\eref{e:NLS}
with nonzero  boundary conditions~\eref{e:NLS}, 
which we then use in the subsequent sections to compute the long-time asymptotic behavior of the solutions.
In \cite{bk2014}, the IST relied on introducing a two-sheeted Riemann surface and a suitable uniformization variable,
following the approach of~\cite{ft1987}.
Here, instead, we work directly in the spectral plane, which is more advantageous for our purposes. 

The focusing NLS equation is typically written in the form
\begin{equation}
i u_t + u_{xx} + 2|u|^2u=0\,,
\label{e:NLSu}
\end{equation}
in which case the corresponding nonzero  boundary conditions are
\begin{equation}
\lim_{x\to \pm \infty} u(x, t)
= q_\pm e^{2iq_{o} ^2t}.
\end{equation}
Equation~\eref{e:NLSu} is trivially converted into~\eref{e:NLS} via the transformation
\begin{equation}
u(x,t) = q(x, t) \,e^{2iq_{o} ^2t}\,.
\label{e:NZBCu}
\end{equation}
On the other hand, the advantage of equation~\eqref{e:NLS} over equation~\eqref{e:NLSu} for our purposes
is that, for the former case, the boundary conditions~\eref{e:NZBC} at infinity are constant, 
while in the case of~\eref{e:NZBCu} they depend on $t$.
That is, \eref{e:NLS}, \eref{e:NZBC} and~\eref{e:q0L1} imply that $\lim_{x\to\pm\infty}q(x,t) = q_\pm$ for all~$t>0$.

Recall that equation~\eqref{e:NLS} admits the Lax pair representation
\begin{subequations}\label{lp}
\begin{align}
\Psi_x &= X\Psi, 
\hskip .5cm
X=ik\sigma_3+Q, \quad k\in\mathbb C,
\label{lpx}
\\
\Psi_t &= T\Psi, 
\hskip .64cm
T = -2ik^2\sigma_3 + i\sigma_3\left(Q_x-Q^2-q_{o} ^2 I \right)-2kQ
\label{lpt}
\end{align}
\end{subequations}
(namely, equation~\eref{e:NLS} is the compatibility condition $X_t - T_x + [X,T]=0$ of equation~\eref{lp}),
where $\Psi(x, t, k)$ is a $2\times 2$ matrix-valued function and
\begin{equation}
\sigma_3
=
\begin{pmatrix}
1 &0
\\
0 &-1
\end{pmatrix},
\quad
Q(x,t) 
=
\begin{pmatrix}
0 &q(x,t)
\\
-\bar q(x, t) &0
\end{pmatrix},
\end{equation}
with the overbar denoting complex conjugation as usual.

\paragraph{Direct problem.}
Let $X_\pm = \lim_{x\to\pm\infty}X(x,t,k)$
and $T_\pm = \lim_{x\to\pm\infty}T(x,t,k)$.
It is straightforward to see that the eigenvector matrix of $X_\pm$ can be written as
\begin{equation}\label{epm}
E_\pm(k) 
=
\begin{pmatrix}
1 & \frac{i(\lambda - k)}{\bar q_\pm}
\\
\frac{i(\lambda -k)}{q_\pm} & 1
\end{pmatrix},
\end{equation}
while the associated eigenvalues $\pm i\lambda$ are defined by the complex square root
\begin{equation}
\label{ldef}
\lambda(k) = \left(k^2+q_{o} ^2\right)^{\frac 12}\,.
\end{equation}
Specifically,  introducing the branch cut  
\begin{equation}
B = i\left[-q_{o}, q_{o} \right],
\end{equation}
we take $\lambda(k)$ to be the unique, single-valued function defined by equation~\eref{ldef}
with a jump discontinuity across $B$,
identified with its limiting value from the right along $B$ so that 
\begin{equation}\label{bound}
\lambda(k) 
= 
\begin{cases}
\sqrt{k^2+q_{o} ^2}, 
&k\in\mathbb R^+\cup B,
\\
 -\sqrt{k^2+q_{o} ^2},
&k\in\mathbb R^-,
\end{cases}
\end{equation} 
 where the square root sign denotes the principal branch of the real square root.

Motivated by the above remarks, and observing that $T_\pm = -2k X_\pm$,
we seek simultaneous solutions $\Psi_\pm$ of the Lax pair \eqref{lp} such that
\begin{equation}\label{infcond}
\Psi_\pm(x, t, k)
=
E_\pm(k)\,e^{i\vartheta(x, t, k) \sigma_3}(1+o(1)),
\quad
x\to \pm\infty,
\end{equation}
where we define
\begin{equation}\label{thetv}
\vartheta(x, t, k)
=
\lambda \left(x-2kt\right).
\end{equation}
It is straightforward to see from equations \eqref{bound} and \eqref{thetv} that the exponentials in the normalization \eqref{infcond} of $\Psi_\pm$ remain bounded as $x\to\pm\infty$ provided that $k\in\Sigma$, with 
$\Sigma= \mathbb R\cup B$.
It is convenient to also introduce the matrix-valued functions $\mu_\pm$  defined by
\begin{equation}\label{mupsi}
\mu_\pm(x, t, k) 
=
 \Psi_\pm(x, t, k) e^{-i \vartheta(x, t, k) \sigma_3}.
\end{equation}
The normalization \eqref{infcond} of $\Psi_\pm$ implies 
\begin{equation}\label{minf}
\mu_\pm(x, t, k) = E_\pm(k)+o(1), \quad x\to \pm \infty, \quad k \in \Sigma .
\end{equation}
Using standard methods, we then obtain linear integral equations of Volterra type for $\mu_\pm$:
\begin{subequations}\label{inteqmu}
\begin{align}
\mu_{-}(x, t, k)
&=
E_-(k)
+
\int_{-\infty}^x 
E_-(k)e^{i\lambda (x - y) \sigma_3}E_-^{-1}(k) \Delta Q_-(y, t) \mu_{-} (y, t, k) e^{-i\lambda  (x-y) \sigma_3} dy,
\\
\mu_{+}(x, t, k)
&=
E_+(k)
-
\int_x^{\infty}
E_+(k)e^{i\lambda (x-y) \sigma_3}E_+^{-1}(k) \Delta Q_+(y, t) 
\mu_{+} (y, t, k) e^{-i\lambda  (x-y) \sigma_3} dy,
\end{align}
\end{subequations}
where 
$\Delta Q_\pm(x, t) = Q(x, t) - Q_\pm$
and $Q_\pm = \lim_{x\to\pm\infty}Q(x,t)$.

The analysis of the Neumann series for the integral equations~\eqref{inteqmu} (see \cite{bk2014} for more details)
allows one to prove existence and uniqueness of the eigenfunctions $\mu_\pm$ for all $k\in\Sigma$ 
provided that $(q(x,t) - q_\pm)\in L^1(\mathbb R_x^\pm)$.
Moreover, denoting by $\mu_{\pm 1}$ and $\mu_{\pm 2}$ the first and second column of $\mu_\pm$ respectively,  
one can conclude that $\mu_{+1}$ and $\mu_{-2}$ 
(and, respectively, $\mu_{-1}$ and $\mu_{+2}$)
can be analytically continued as functions of $k$ in 
$\mathbb C^+\setminus B^+$
(and, respectively, $\mathbb C^-\setminus B^-$), where 
\begin{equation}
\mathbb C^\pm= \left\{ k\in\mathbb C:\Im(k)\gl0\right\},
 \quad
 B^\pm = B\cap \mathbb C^\pm.
\end{equation}  
Consequently, definition \eqref{mupsi} implies that $\Psi_{-1}$ and $\Psi_{+2}$ are analytic for $k\in\mathbb C^-\setminus B^-$, while $\Psi_{+1}$ and $\Psi_{-2}$ are analytic for $k\in\mathbb C^+\setminus B^+$.

\paragraph{Scattering matrix.}
Since $X$ and $T$ are traceless, by Abel's theorem the determinants of $\Psi_\pm$ are independent of $x$ and $t$. 
Thus, using the asymptotic conditions \eqref{infcond} we have
\begin{equation}
\det \Psi_\pm (x, t, k)
= 
\lim_{x\to \pm \infty} \det \Psi_\pm (x, t, k)
=
d(k),
\end{equation}
where
\begin{equation}\label{ddef}
d(k)= \frac{2\lambda}{\lambda+k}.
\end{equation}
The definition \eqref{ldef} of $\lambda$ implies that $d$ is nonzero  and non-singular for all $k\in\Sigma^*$, where
\begin{equation}
\Sigma^*= \Sigma \setminus \left\{ \pm iq_{o} \right\}\,.
\end{equation}
Therefore, 
the matrices $\Psi_\pm$  are fundamental matrix solutions of both parts of the system \eqref{lp}
for all $k\in\Sigma^*$.
It then follows that there exists a matrix $S(k)$ independent of $x$ and $t$, hereafter called the scattering matrix, such that 
\begin{equation}\label{sdef}
\Psi_-(x, t, k)
=
\Psi_+(x, t, k) S(k), \quad k\in\Sigma^*.
\end{equation}
Note that $S$ is unimodular as
\begin{equation}\label{sdet}
\det S(k)
=
\frac{\det \Psi_-(x, t, k)}{\det \Psi_+(x, t, k)}
\equiv 1.
\end{equation}
Moreover, 
we will see that the diagonal entries of $S(k)$ can be analytically continued off $\Sigma^*$.

\paragraph{Symmetries.}
The identity
\begin{equation}\label{ids}
\overline{X(q, \bar q, \bar k)}
=
-\sigma_*
X(q, \bar q, k)
\sigma_*,
\quad
\sigma_*= 
\begin{pmatrix}
0 &1
\\
-1 & 0
\end{pmatrix},
\end{equation}
implies 
\begin{equation}\label{sym2b}
\overline{ \Psi_\pm(x, t, \bar k)}
=
-\sigma_* \Psi_\pm(x, t, k) \sigma_*,\quad k\in\Sigma^*,
\end{equation}
which in turns yields the symmetry condition
\begin{equation}
\overline{S(\bar k)}
=
-\sigma_* S(k) \sigma_*, \quad k\in\Sigma^*.\label{syms}
\end{equation}
In fact, letting $s_{11}(k)=a(k)$ and $s_{21}(k)=b(k)$ the scattering matrix $S$ takes the form
\begin{equation}\label{sform}
S(k)
=
\begin{pmatrix}
a(k) & -\bar b(k)
\\
b(k) & \bar a(k)
\end{pmatrix}, \quad k\in\Sigma^*,
\end{equation}
where $\bar a(k)=\overline{a(\bar k)}$ and $\bar b(k)=\overline{b(\bar k)}$ denote the Schwartz conjugates. 
The determinant condition \eqref{sdet} then becomes
\begin{equation}\label{abdetgen}
a(k)\bar a(k)+b(k)\bar b(k)=1, \quad k\in\Sigma^*.
\end{equation}
Finally, equations \eqref{sdef} and \eqref{sform} imply
\begin{subequations}\label{abdef}
\begin{align}
a(k)
&=
(1/d(k))\,{\wr\left(\Psi_{-1}(x, t, k), \Psi_{+2}(x, t, k)\right)}, \quad k\in\Sigma^*,
\label{adef}
\\
\bar a(k)
&=
(1/d(k))\,{\wr\left( \Psi_{+1}(x, t, k), \Psi_{-2}(x, t, k) \right)},
\quad k\in\Sigma^*,\label{abardef}
\\
b(k)
&=
(1/d(k))\,{\wr\left(\Psi_{+1}(x, t, k), \Psi_{-1}(x, t, k)\right)}, \quad k\in\Sigma^*,\label{bdef}
\\
\bar b(k)
&=
(1/d(k))\,{\wr\left(\Psi_{+2}(x, t, k), \Psi_{-2}(x, t, k)\right)}, \quad k\in\Sigma^*,\label{bbardef}
\end{align}
\end{subequations}
where $\wr$ denotes the Wronskian determinant.
Thus,
recalling that $d(k)$ in equation \eqref{ddef} is analytic for $k\notin B$,  
we infer that the function $a(k)$ is analytic in $\mathbb C^-\setminus B^-$.  Similarly, 
the Schwartz conjugate $\bar a(k)$ is analytic in $\mathbb C^+\setminus B^+$.
On the other hand, the function $b(k)$ cannot be analytically continued away from $\Sigma^*$ in general.

\paragraph{Jumps of the eigenfunctions and scattering data across the branch cut.}
The jump discontinuity of $\lambda$ across the branch cut $B$ induces a corresponding jump for the eigenfunctions and scattering data. (This is one of the points in which the present formulation of the IST differs most significantly from that in~\cite{bk2014}.) 
Since $\lambda$ has been taken to be continuous from the right of $B$, the same is true for the analytic eigenfunctions and scattering data.
That is, taking $B$ to be oriented upwards,
\begin{gather*}
\mu_{-1}^-(k) 
= 
\lim_{\varepsilon\to 0^+}
 \mu_{-1}(k+\varepsilon) = \mu_{-1}(k),
\quad
\mu_{+2}^-(k) 
= 
\lim_{\varepsilon\to 0^+}
 \mu_{+2}(k+\varepsilon) = \mu_{+2}(k),
\quad k\in B^-,
\end{gather*}
with similar relations for the eigenfunctions analytic in $\Complex^+\setminus B^+$, together with
\begin{gather*}
a^-(k) 
= 
\lim_{\varepsilon\to 0^+}
a(k+\varepsilon) = a(k), 
\quad k \in  B^-. \nonumber
\end{gather*}
The jumps of $\mu_{\pm}$ across $B$ are then expressed by the following lemma.
\begin{lemma}\label{bcm-lem}
The eigenfunctions defined via the integral equations \eqref{inteqmu} satisfy the following jump conditions across the branch cut $B\!:$
\begin{subequations}\label{bcm}
\begin{gather}	
\mu_{-1}^+(x, t, k)
=
\frac{\lambda+k}{iq_-}\,
\mu_{-2}(x, t, k), 
\quad 
\label{bcm-1}
\mu_{+2}^+(x, t, k)
=
\frac{\lambda+k}{i\bar q_+}\,
\mu_{+1}(x, t, k), \quad k\in  B^-,
\\ 
\mu_{-2}^+(x, t, k)
=
\frac{\lambda+k}{i\bar q_-}\,
\mu_{-1}(x, t, k), 
\quad 
\mu_{+1}^+(x, t, k)
=
\frac{\lambda+k}{iq_+}\,
\mu_{+2}(x, t, k), \quad k\in B^+.\label{bcm+1}
\end{gather}
\end{subequations}
\end{lemma}

\begin{proof}
Since the  pairs $\{ \Psi_{+1}, \Psi_{+2}\}$ and $\{ \Psi_{-1}, \Psi_{-2}\}$ are both fundamental sets of solutions of equation \eqref{lpx}, 
every other solution of this equation can be expressed as a linear combination of either pair. 
In particular, since the limit $\Psi_{-1}^+$ of $\Psi_{-1}$ as $k$ approaches $B^-$ from the left is also a solution of equation \eqref{lpx}, we must have
\begin{equation*}
\Psi_{-1}^+(x, t, k)
=
c_{-1}(k, t) \Psi_{-1}(x, t, k)+\tilde c_{-1}(k, t) \Psi_{-2}(x, t, k),
\quad
k\in B^-,
\end{equation*}
for some functions $c_{-1}$ and $\tilde c_{-1}$ are independent of $x$ and yet to be determined. 
Then, separating the columns of equations~\eqref{infcond} 
and taking the limit  of the above equation as $x\to -\infty$  we find that $c_{-1}(k, t)\equiv 0$ and
$\tilde c_{-1}(k,t) =
{\lambda+k}/({iq_-})$.
Definition \eqref{mupsi} then
yields  the jump condition \eqref{bcm-1}. 
The remaining jump conditions are obtained in a similar fashion.
\end{proof}


\begin{corollary}\label{acut-cor}
The function $a$ defined by  \eqref{adef} satisfies the following  condition across $B^+\!:$
\begin{equation}
\bar a^+(k)
=
({q_-}/{q_+})\, a(k), \quad k\in B^+.
\label{acut}
\end{equation}
\end{corollary}

\paragraph{Inverse problem.}
Exploiting the analyticity properties of $\mu_{\pm 1,2}$ in the spectral plane, we proceed to define 
the sectionally analytic matrix-valued function
\begin{equation}\label{mrh}
M(x, t, k)
=
\begin{cases}
\left( \dfrac{\mu_{+1}}{\bar a d}, \mu_{-2} \right)
=
\left( \dfrac{\Psi_{+1}}{\bar a d}, 
 \Psi_{-2} \right)  e^{-i \vartheta \sigma_3}, &k\in\mathbb C^+\setminus B^+,
\\
\left( \mu_{-1}, \dfrac{\mu_{+2}}{a d} \right)
=
\left(  \Psi_{-1}, \dfrac{\Psi_{+2}}{ad}  \right) 
e^{-i \vartheta \sigma_3}, &k\in\mathbb C^-\setminus B^-,
\end{cases}
\end{equation}
where the dependence on the variables $x, t, k$ on the right-hand side has been suppressed for brevity.
Note the presence of $d$ in  definition \eqref{mrh} (as compared to \cite{bk2014} and other standard formulations of the inverse problem),
which implies that $M$ is unimodular, i.e.,
\begin{equation}\label{detM}
\det M(x, t, k) \equiv 1, \quad k\in\mathbb C\setminus B.
\end{equation}
As usual, the inverse problem is formulated in terms of an appropriate Riemann-Hilbert problem.
To define this Riemann-Hilbert problem, one needs appropriate jump condition for $M$.
We first compute the jump of $M$ across $k\in\mathbb R$.
Denoting by $M^{\pm}$ the limits of $M$ as $\Im(k)\to 0^\pm$, recalling that $\vartheta$, $a$ and $d$ are continuous across $\mathbb R$, 
and using the scattering relation~\eqref{sdef} and the determinant condition \eqref{abdetgen}, we obtain the jump condition
\begin{equation*}
M^+(x, t, k)
 =
M^-(x, t,  k) 
\begin{pmatrix}
\dfrac{1}{d(k)}\left[1+r(k)\bar r(k)\right]
&
\bar r(k)e^{2i\vartheta(x, t, k)} 
\\
r(k) e^{-2i\vartheta(x, t, k)} 
&
d(k)
\end{pmatrix}, \quad k \in \mathbb R,
\end{equation*}
where the real $k$-axis is oriented from left to right as usual, 
$\vartheta$ is defined by equation~\eqref{thetv} 
and the reflection coefficient $r(k)$ with Schwartz conjugate $\bar r(k)$ is defined by
\begin{equation}\label{refdef}
r(k) = - {b(k)}/{\bar a(k)}.
\end{equation}

Next, we note that the jump of $M$ across the branch cut $B$  
is also affected by the discontinuities of
$\mu_{-1}$ and $\mu_{+2}$ across $B^-$ and 
of $\mu_{+1}$ and $\mu_{-2}$ across $B^+$.
Hence, starting from the definition \eqref{mrh} of $M$ and employing Lemma \ref{bcm-lem} and Corollary \ref{acut-cor} in combination with straightforward algebraic computations, we obtain the jump conditions
\begin{align*}
M^+(x, t,  k)
&=
M^-(x, t, k) 
\begin{pmatrix}
- \dfrac{\lambda-k}{iq_-}\,  \bar r(k)  e^{2i\vartheta(x, t, k)}
&
 \dfrac{2\lambda}{i\bar q_-} 
\\
\dfrac{\bar q_-}{2i \lambda}  \left[1+r(k) \bar r(k) \right] 
 &
 -\dfrac{\lambda+k}{i\bar q_-} \, r(k)   e^{-2i\vartheta(x, t, k)}
\end{pmatrix}, 
&&k\in B^+,
\\
M^+(x, t, k)
&=
M^-(x,  t, k)  
\begin{pmatrix}
\dfrac{\lambda+k}{iq_-}\, \bar r(k)  e^{2i\vartheta(x, t, k)}
& 
\dfrac{q_-}{2i\lambda}\left[1+r(k)\bar r(k)\right]
\\
\dfrac{2\lambda}{iq_-} 
&
\dfrac{\lambda-k}{i\bar q_-}\, r(k)  e^{-2i\vartheta(x, t, k)}
\end{pmatrix}, 
&&k\in B^-,
\end{align*}
with $B$ oriented upwards as before,
and with $\vartheta$ and $r$ as above.

Finally, to complete the formulation of the RHP one must specify a normalization condition and,
if a discrete spectrum is present, appropriate residue conditions.
The latter will not be necessary in our case, since in what follows we will assume that no discrete spectrum is present. 
Using 
the integral equations~\eqref{inteqmu} for $\mu_\pm$ and the relationship \eqref{mrh} between $M$ and $\mu_\pm$,
it can be shown (see \cite{bk2014}) that the function $M$ admits the large-$k$ asymptotic expansion
\begin{equation}\label{masy}
M(x, t, k) = I + \frac{M_1(x, t)}{k}+O \Big(\frac{1}{k^2}\Big),\quad k \to \infty.
\end{equation}
Combining equations \eqref{lpx}, \eqref{mrh} and \eqref{masy}, one can also recover 
the solution  of the focusing NLS equation \eqref{e:NLS} in the form
\begin{equation}\label{qsol}
q(x, t)
=
-2i \left(M_1(x, t)\right)_{12}.
\end{equation}
Based on the above discussion, the function $M$ satisfies the following Riemann-Hilbert problem:
\begin{rhp}\label{rhp0}
Suppose that $a(k)\neq 0$ for all $k\in\mathbb C^-\cup \Sigma$ so that no discrete spectrum is present. 
Determine a sectionally analytic matrix-valued function $M(k)=M(x, t, k)$ 
in $\mathbb C\setminus \Sigma$ satisfying the jump conditions
\begin{subequations}\label{rhp00}
\begin{align}
&M^+( k) = M^-(k) V_1,\quad k\in \mathbb R,
\\
&M^+(k) = M^-( k)V_2,\quad k\in B^+,
\\
&M^+(k) = M^-(k)V_3,\quad k\in B^-,
\end{align}
and the normalization condition
\begin{equation}
M(k) =I +O (1/k),\quad k \to \infty,
\end{equation}
\end{subequations}
where 
\begin{subequations}\label{rhp0-jumps}
\begin{equation}
V_1(x, t, k)
=
\begin{pmatrix}
\dfrac{1}{d(k)}\left[1+r(k)\bar r(k)\right]
&
\bar r(k)e^{2i \theta(\xi, k) t} 
\\
r(k) e^{-2i \theta(\xi, k) t} 
&
d(k)
\end{pmatrix},
\end{equation}
\begin{equation}
V_2(x, t, k)
=
\begin{pmatrix}
- \dfrac{\lambda-k}{iq_-}\,  \bar r(k) \, e^{2i \theta(\xi, k) t}
&
 \dfrac{2\lambda}{i\bar q_-} 
\\
 \dfrac{\bar q_-}{2i \lambda}  \left[1+r(k)\bar r(k)\right] 
 &
 -\dfrac{\lambda+k}{i\bar q_-} \, r(k) \,  e^{-2i \theta(\xi, k) t}
\end{pmatrix},
\end{equation}
\begin{equation}
V_3(x, t, k)
=
\begin{pmatrix}
\dfrac{\lambda+k}{iq_-}\, \bar r(k) \, e^{2i \theta(\xi, k) t}
& 
\dfrac{q_-}{2i\lambda}\left[1+r(k)\bar r(k)\right]
\\
\dfrac{2\lambda}{iq_-} 
&
\dfrac{\lambda-k}{i\bar q_-}\, r(k)\,  e^{-2i \theta(\xi, k) t}
\end{pmatrix},
\end{equation}
\end{subequations}
the reflection coefficient $r$ is defined by equation \eqref{refdef}, the similarity variable is
\begin{equation}
\xi  =  x/t,
\end{equation}
and the function $\theta(\xi, k) = \vartheta(x,t,k)/t$ is given by
\begin{equation}\label{thet}
\theta(\xi, k) = \lambda \left(\xi-2k \right).
\end{equation}
\end{rhp}


\section{Long-time asymptotics: preliminaries}
\label{outline-sec}

We now compute the long-time asymptotic behavior of of the solution $q$ of the focusing NLS equation,
as obtained by equation~\eqref{qsol}, 
by analyzing  the Riemann-Hilbert problem \ref{rhp0} via the Deift-Zhou nonlinear steepest descent for oscillatory Riemann-Hilbert problems \cite{dz1993}. 
Recall that, in general, the Deift-Zhou method is based on deforming the jump contours of the Riemann-Hilbert problem to contours in the complex $k$-plane across which the relevant jumps have a well-defined limit as $t\to \infty$.
In our case, it turns out that,
after the appropriate deformations have been performed, 
the majority of the jumps across the deformed contours tend to the identity matrix, 
while those that yield the leading order contribution to the asymptotics tend to constant matrices.

\paragraph{Analyticity in a neighborhood of the continuous spectrum.}
Hereafter, in order to be able to deform the contours away from the continuous spectrum $\Sigma$, we place an additional restriction on the initial datum $q(x, 0)$ which ensures that the reflection coefficient $r$ can be analytically extended off the continuous spectrum.

\begin{lemma}[Analyticity of the reflection coefficient]\label{an-lem}
Suppose that there exist a constant  $\epsilon>0$ such that 
\begin{equation}\label{an-cond}
e^{\pm \epsilon x}\left(q(x, 0) - q_{\pm}\right) \in L^1(\mathbb R^\pm),
\end{equation}
where $q(x,0)$ and $q_\pm$ are the initial and boundary conditions, respectively, of the focusing NLS equation \eqref{e:NLS}.
Then, all the eigenfunctions and the spectral functions $a$ and $b$ are analytic in the region
$$
\Sigma_{\epsilon} = \left\{ k\in\mathbb C\setminus B: \left|\textnormal{Im}(\lambda)\right|<\epsilon \right\},
$$
where $\lambda$ is defined by equation \eqref{ldef}. The same conclusion follows for the reflection coefficient $r$ defined by equation \eqref{refdef}
as long as $a(k)\neq 0$ for all $k\in \Sigma_{\epsilon}$.
\end{lemma}

The region $\Sigma_\epsilon$ in Lemma~\ref{an-lem} is the analogue of what is known as the Bargmann strip in the case of zero conditions at infinity \cite{g1991}.
Lemma~\ref{an-lem} can be established by employing a Neumann series for the integral equations \eqref{inteqmu} that define $\mu_\pm$, using similar arguments as in~\cite{bk2014}. 
It is important to note that the contour lines $\Im(\lambda)=0$ do not intersect the branch cut $B$, 
e.g., see Figure \ref{iml-cont}. 
Hence, the existence of any $\epsilon>0$ ensures that $r$ is analytic in a neighborhood of $B$ (except possibly at the branch points), 
which is the domain of analyticity required for the deformations in the Deift-Zhou method.
\begin{figure}[b!]
\begin{center}
\includegraphics[scale=.435]{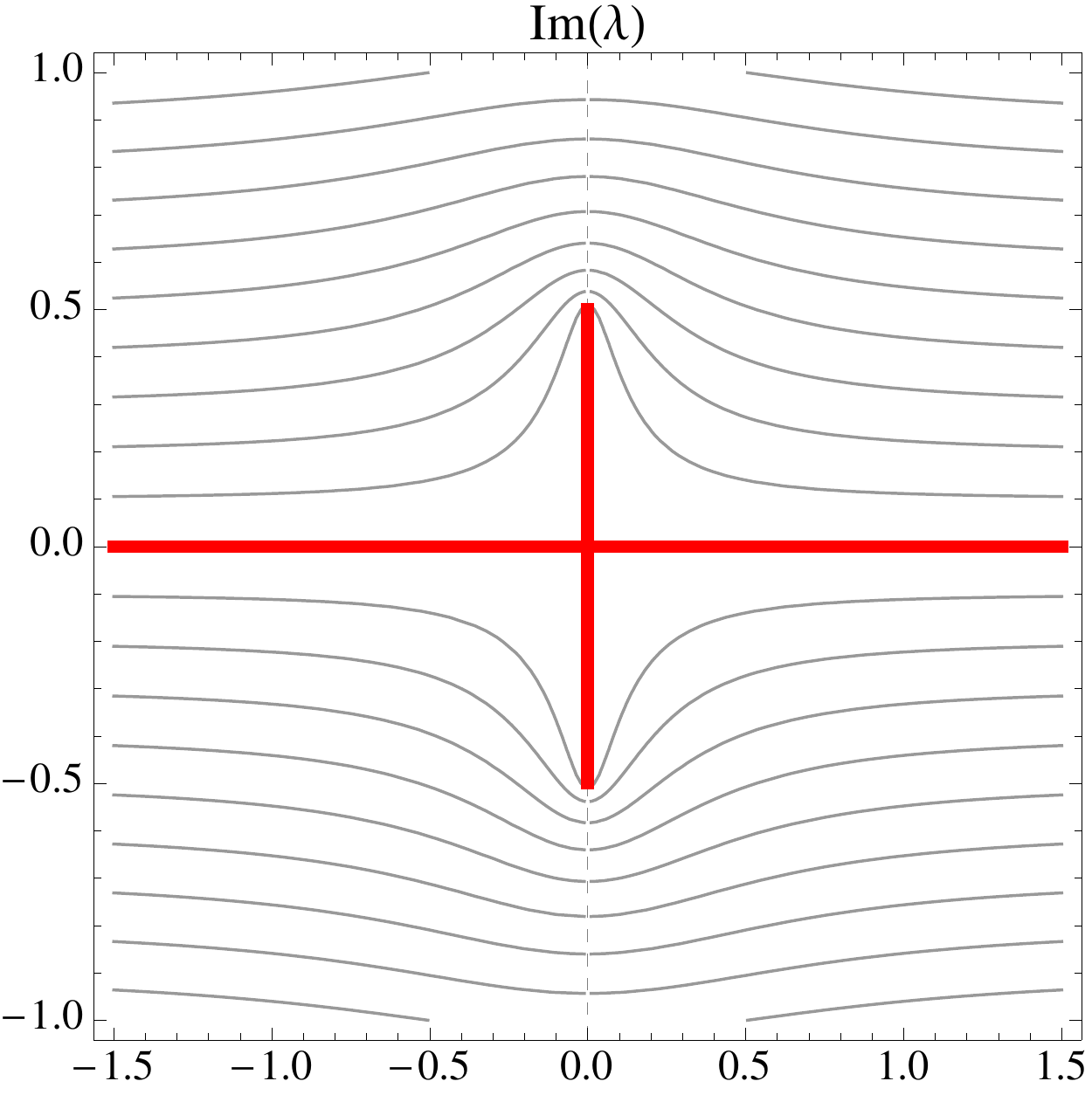}
\caption{Contour plot of $\Im(\lambda)$ for $q_{o} =\frac 12$.}
\label{iml-cont}
\end{center}
\end{figure}

When the hypotheses of Lemma \ref{an-lem}  are satisfied, the jump condition~\eqref{acut} also holds across~$B^-$.
In addition, \eqref{abdef} and \eqref{bcm} yield a similar jump condition for $b$, namely,
$b^+(k) =
-({\bar q_+}/{q_-})\,\bar b(k)$ for all $k\in B$.
Hence, we obtain the following jump condition for the reflection coefficient $r$ across~$B$, 
namely
\begin{equation}\label{rp0}
r^+(k) 
= -({\bar q_-}/{q_-})\, \bar r(k), \quad k\in B.
\end{equation}

An example of an initial condition satisfying the hypothesis of Lemma \ref{an-lem} is  the following box-like initial datum
considered in \cite{bf2015}:
\begin{equation}
q(x, 0)
=
\begin{cases}
q_{o}, &|x|>L,
\\
\beta e^{i \chi }, &|x|<L, 
\end{cases}
\end{equation}
with $\beta>0$, $L>0$ and $\chi\in\Real$,
which gives rise to the  reflection coefficient
\begin{equation}
r(k)
=
\frac{e^{2i \lambda L}
\big[\left(\beta \cos\chi -q_{o} \right)k-i \beta \lambda \sin \chi \big]}
{\lambda \sqrt{k^2+\beta^2} \cot\big(2L\sqrt{k^2+\beta^2}\, \big) - i  \left(k^2+q_{o}  \beta \cos\chi \right)}.
\end{equation}

\paragraph{The sign structure of $\Re(i\theta)$.}
The choice of deformations of the Deift-Zhou method  depends crucially on
the sign structure of the quantity $\Re(i\theta)$, which is involved  in all three jump matrices \eqref{rhp0-jumps} of the Riemann-Hilbert problem
\ref{rhp0}.
From the definition \eqref{thet} of $\theta$ we have
\begin{equation}\label{rethet}
\Re(i\theta)
=
-\lambda_\im \left(\xi-2k_\re\right)+2\lambda_\re k_\im.
\end{equation}
The above expression simplifies significantly in the
special cases  $|k_\im|\ll 1$ and  $|k_\im|\gg 1$.
When $|k_\im|\gg 1$, recalling that $\lambda \sim k$
as $k\to \infty$ we have 
\begin{equation}
\Re(i\theta)
=
 \left(4k_\re-\xi\right)k_\im +O(1/{k_\im}),
 \quad k_\im \to \pm\infty.\label{rethet3}
\end{equation}
Thus, as $k_\im\to\pm\infty$, the sign of $\Re(i\theta)$ is determined by that of $k_\re-\xi/4$.
%
When $0<k_\im\ll 1$, on the other hand, the situation is more complicated.  
The definition of $\lambda$ implies
\begin{equation}\label{las}
\lambda
=
\textrm{sign}(k_\re)
\sqrt{k_\re^2+q_{o} ^2}
\left[
1-\dfrac{q_{o} ^2}{2(k_\re^2+q_{o} ^2)^2}\, k_\im^2
+
\dfrac{ik_\re}{k_\re^2+q_{o} ^2}\, k_\im
+O(k_\im^3)\right].
\end{equation}
Hence, 
\begin{equation}\label{rethet2}
\Re(i\theta)
=
\textrm{sign}(k_\re)\,
\dfrac{4}{\sqrt{k_\re^2+q_{o} ^2}}
\left(
k_\re^2-\dfrac{\xi}{4}\, k_\re+\dfrac{q_{o} ^2}{2}
\right)
k_\im 
+O(k_\im^3),\quad  k_\im\to 0^+.
\end{equation}
For $|\xi|<4\sqrt 2 q_{o} $ the quadratic expression in the leading order term of \eqref{rethet2} is always positive, while for $|\xi|\geqslant 4\sqrt 2 q_{o} $ it has real roots $k_1, k_2$ equal to 
\begin{equation}\label{kroots}
k_{1, 2}
= \tfrac18 \bigg(
\xi\pm\sqrt{\big(\xi-4\sqrt 2q_{o} \big)\big(\xi+4\sqrt 2q_{o} \big)}\, \bigg),
\end{equation}
where we take $k_1<k_2$.
A similar expansion is obtained when $k_\im\to0^-$, leading to the same roots as in equation~\eref{kroots}.
The overall sign structure of $\Re(i\theta)$ in the complex $k$-plane is illustrated in Figure~\ref{f:imtheta}.

\begin{remark}
Importantly, the points $k_1, k_2$ are the stationary points of $\theta$.
In the following sections we will show that the sectors of the $xt$-plane where
$|\xi|> 4\sqrt 2 q_{o} $ correspond to \textit{plane wave regions},
whereas the sectors where $|\xi|<4\sqrt 2 q_{o} $ correspond to \textit{modulated elliptic wave regions}. 
\end{remark}

\begin{figure}[t!]
\includegraphics[width=0.24\textwidth]{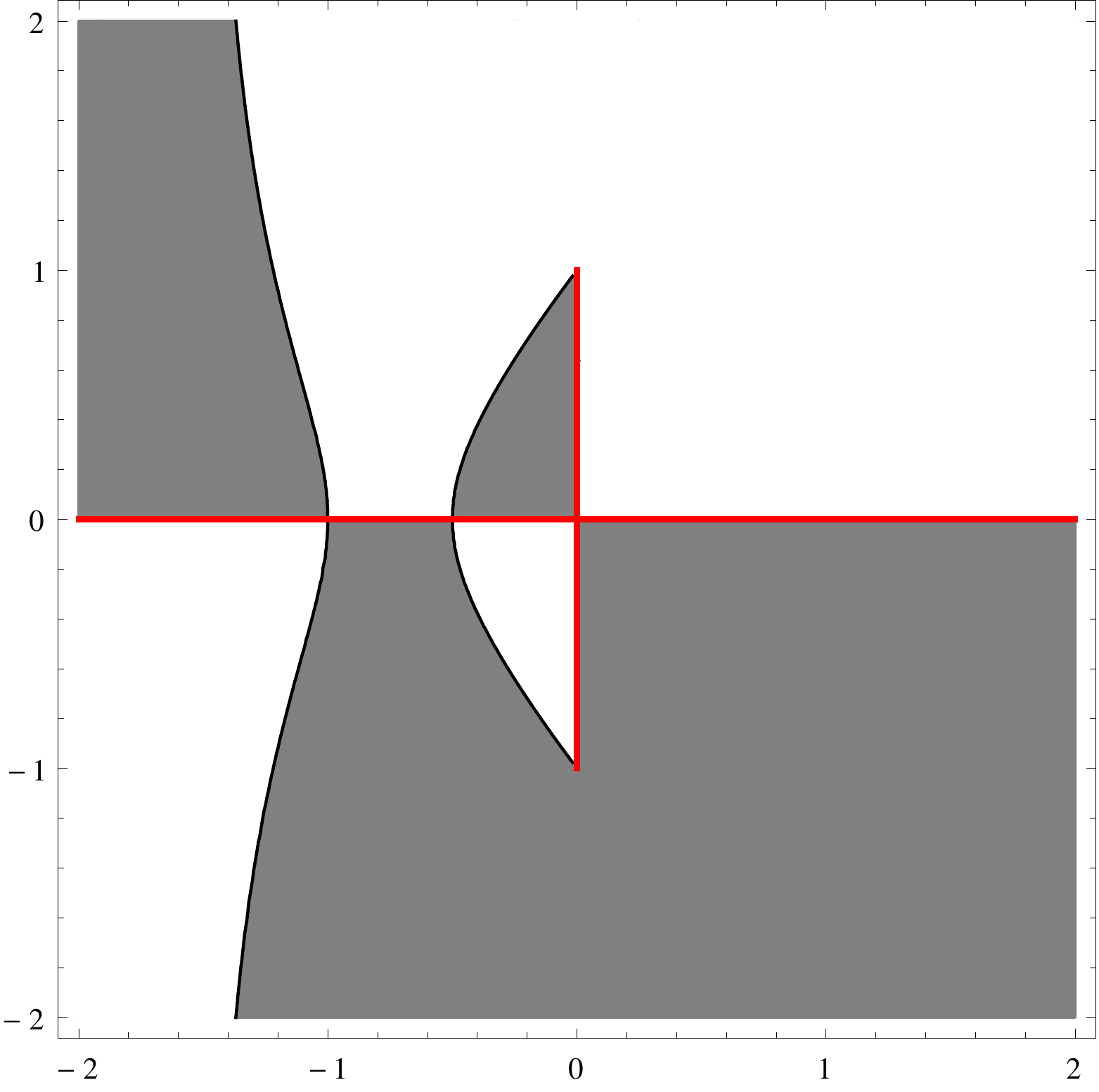}~%
\includegraphics[width=0.24\textwidth]{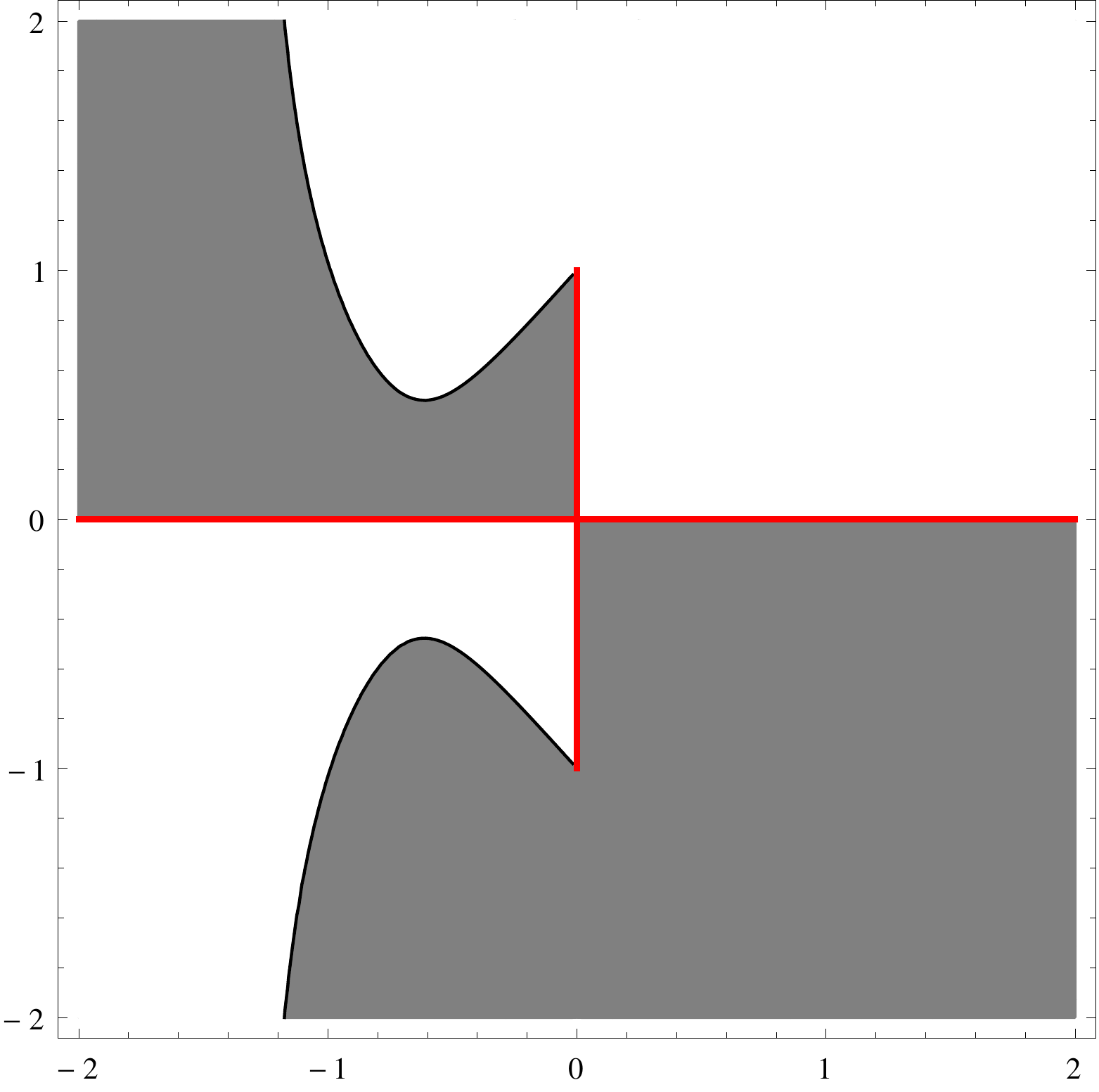}~%
\includegraphics[width=0.24\textwidth]{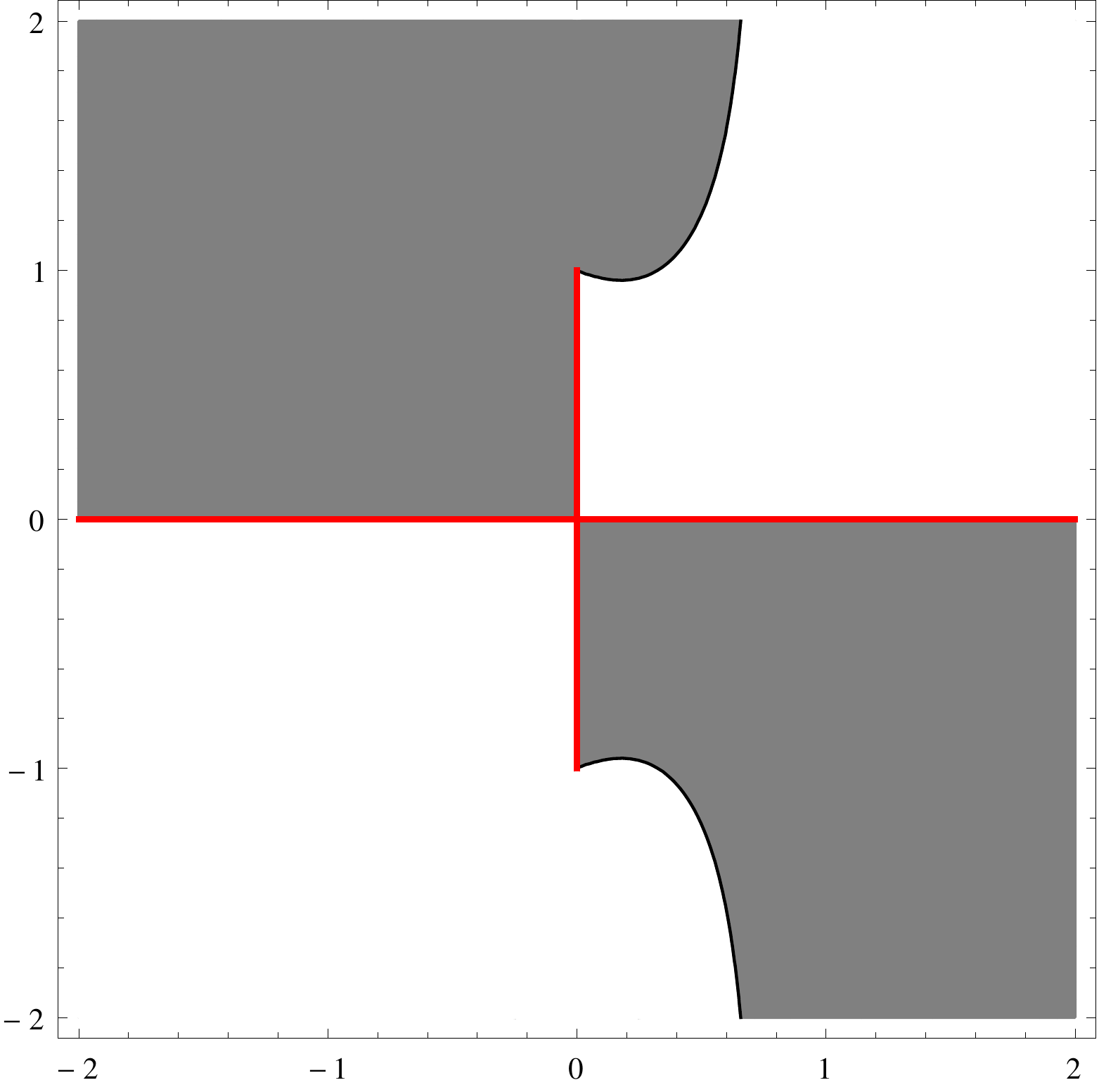}~%
\includegraphics[width=0.24\textwidth]{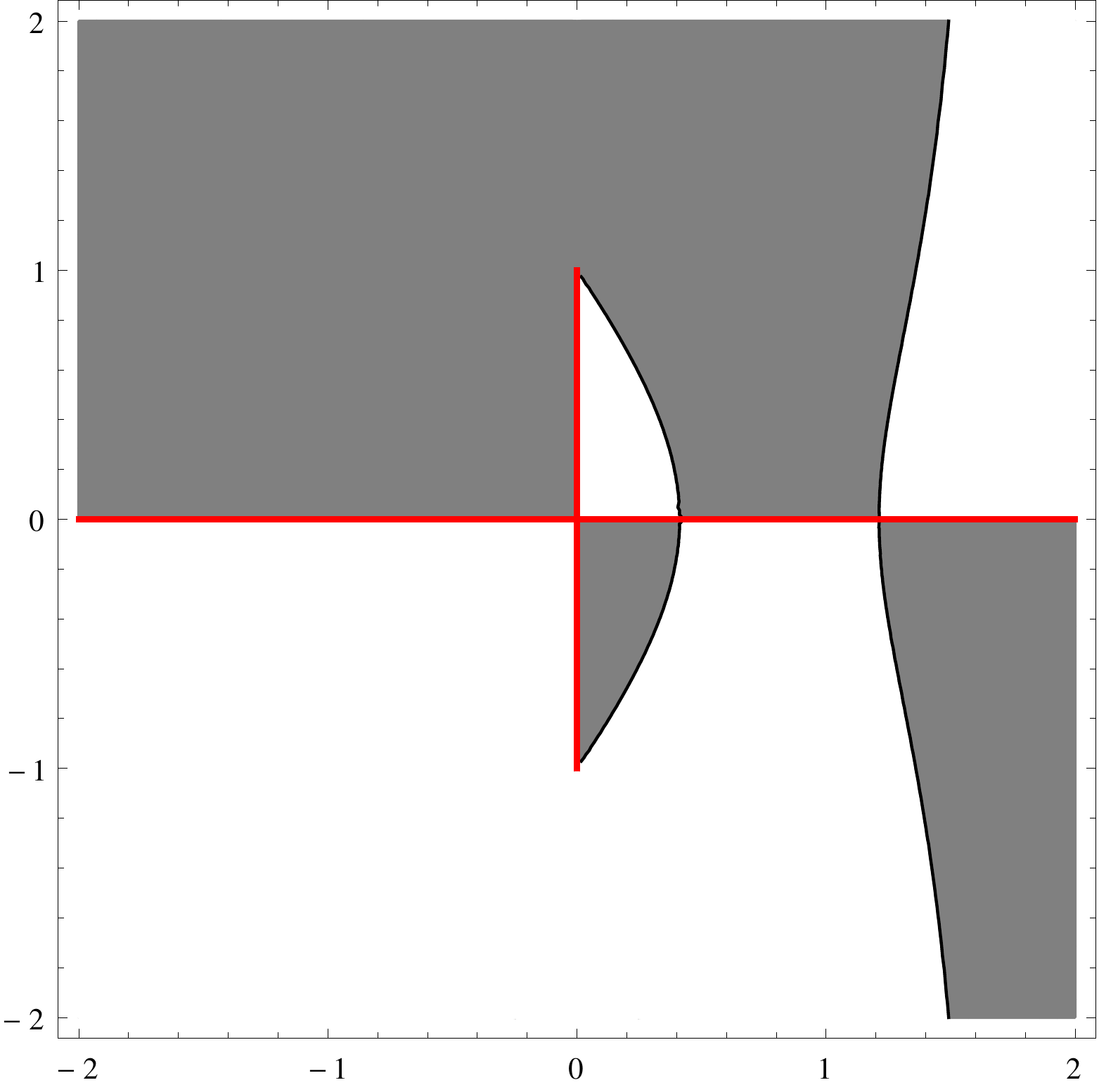}%
\begin{caption}
{The sign structure of $\Re(i\theta)$ in the complex $k$-plane for various values of $\xi$ and $q_{o} =1$:
(a)~$\xi = -6$, corresponding to $x< - 4\sqrt2q_{o} t$;
(b)~$\xi = -5.2$, corresponding to $-4\sqrt2q_{o} t<x<0$;
(c)~$\xi = 3$, corresponding to $0<x< 4\sqrt2q_{o} t$;
(d)~$\xi = 6.5$, corresponding to $x > 4\sqrt2q_{o} t$.
In the gray regions $\Re(i\theta)<0$, whereas in the white regions $\Re(i\theta)>0$.
Cases~(a) and~(d) correspond to plane wave regions, 
while cases~(b) and~(c) correspond to modulated elliptic wave regions.
The points $k_{1,2}$ in cases~(a) and~(d) correspond with the sign changes of $\Re(i\theta)$ along the real axis away from $k=0$. 
}
\label{f:imtheta}
\end{caption}
\end{figure}


\section{ The plane wave region: proof of theorem \ref{pw-t}}\label{pw-reg}
\label{pw-sec}

In this section we prove Theorem \ref{pw-t}, i.e., we compute the leading order long-time asymptotic behavior of the  solution of the focusing NLS equation \eqref{e:NLS} in the plane wave region $|x|>4\sqrt 2 q_{o}  t$.
As mentioned earlier, we do so by performing appropriate deformations of the Riemann-Hilbert problem~\ref{rhp0}.

All three jump matrices $V_1^{(0)}=V_1$, $V_2^{(0)}=V_2$ and $V_3^{(0)}=V_3$ of the Riemann-Hilbert problem \ref{rhp0} contain both $e^{i \theta t}$ and $e^{-i \theta t }$. 
Recall also that $\theta(\xi,k)\in\Real$ for all $k\in\Sigma$.
The first step required in order to take the limit $t\to \infty$ is to express these jumps as products of matrices that 
involve only one of the two aforementioned exponentials,
and with the matrices ordered in such a way that they remain bounded when the contour is deformed away from $\Sigma$.

\paragraph{Preliminary factorizations: $x<-4\sqrt2q_{o} t$.}
We will show that, in this case, the factorizations of $V_1^{(0)}$ convenient for our purposes are
\begin{equation}
\label{v10fac}
V_1^{(0)} = \begin{cases}
  V_2^{(1)}V_0^{(1)}V_1^{(1)}, &k_\re<k_1,\\
  V_4^{(1)}V_3^{(1)}, &k_\re>k_1,
\end{cases}
\end{equation}
%
where 
\begin{subequations}
\label{v112345}
\begin{gather}
V_0^{(1)}
=
\begin{pmatrix}
1+r\bar r & 0 \\
0 & \dfrac{1}{1+r\bar r}
\end{pmatrix},
\quad
V_1^{(1)}
=
\begin{pmatrix}
d^{-\frac 12} & \dfrac{d^{\frac 12}\, \bar r e^{2i \theta t } }{1+r\bar r} \\
0 & d^{\frac 12}
\end{pmatrix},
\\
V_2^{(1)}
=
\begin{pmatrix}
d^{-\frac 12} &  0 \\
\dfrac{d^{\frac 12} r e^{-2i \theta t } }{1+r\bar r} & d^{\frac 12}
\end{pmatrix},
\quad
V_3^{(1)}
=
\begin{pmatrix}
d^{-\frac 12} & 0 \\
d^{-\frac 12}r e^{-2i \theta t } & d^{\frac 12}
\end{pmatrix},
\quad
V_4^{(1)}
=
\begin{pmatrix}
d^{-\frac 12}
&
d^{-\frac 12}\bar r e^{2i \theta t } 
\\
0
&
d^{\frac 12}
\end{pmatrix}.
\end{gather}
\end{subequations}
Moreover, since $k_{1,2}<0$ in this region (cf.\ Figure~\ref{f:imtheta}\textcolor{blue}{a}), the branch cut $B$ lies to the right of the point $k_1$.
Thus, the appropriate factorizations for the jumps $V_2^{(0)}$ and $V_3^{(0)}$ will turn out to be 
\begin{subequations}\label{cut1m}
\begin{align}
V_2^{(0)}
&=
\big(V_{3-}^{(1)}\big)^{-1}
V_B
V_{3+}^{(1)},\quad k\in B^+,
\\
V_3^{(0)}
&=
V_{4-}^{(1)}
V_B
\big(V_{4+}^{(1)}\big)^{-1},\quad k\in B^-,
\end{align}
\end{subequations}
where $V_{3,4\pm}^{(1)}$ denote the left and right-sided limits of $V_{3,4}^{(1)}$ and $V_B$ is the \textit{constant} matrix
\begin{equation}
V_B
=
\begin{pmatrix}
0
&
{q_-}/{iq_{o} }
\\
{\bar q_-}/{iq_{o} }
&
0
\end{pmatrix}.\label{v18}
\end{equation}
%

\begin{remark}[$k_1$ versus $k_2$]
In the case of $x<-4\sqrt2q_{o} t$, the point $k_2$ is not significant concerning the choice of factorization of the jump $V_1^{(0)}$. 
This is because, 
as can be seen from Figure \ref{f:imtheta}\textcolor{blue}{a}, 
the change of sign at $k_2$ affects only a finite region of the $k$-plane, 
which can therefore be bypassed by suitable deformations of the jump contours.
The opposite is of course true for $x>4\sqrt2q_{o} t$.
\end{remark}

\begin{remark}
[$x<-4\sqrt2q_{o} t$ versus $x>4\sqrt2q_{o} t$]
\label{xpm-rem}
For $x<-4\sqrt2q_{o} t$, using the factorizations~\eqref{v10fac}-\eqref{cut1m} and performing suitable deformations in the complex $k$-plane, 
we will eventually be able to reduce the jump across the branch cut $B$ to the \textit{constant} matrix $V_B$ defined above. 
This reduction is crucial, since it implies that the Riemann-Hilbert problem that yields the leading order asymptotic behavior of the  focusing NLS solution can be solved explicitly.
For $x>4\sqrt2q_{o} t$, however, the sign structure of $\Re(i\theta)$ is such that the desired reduction cannot be accomplished using the above factorizations. 
In that case, it is then necessary to first perform a rescaling of the original Riemann-Hilbert problem, 
as discussed next.
\end{remark}

\paragraph{Preliminary factorizations: $x>4\sqrt2q_{o} t$.} 
In this case we first rescale the Riemann-Hilbert problem \ref{rhp0}  as follows. Let
\begin{equation}\label{mrhtil}
\tilde M(x, t, k)
=
\begin{cases}
M(x, t, k)\overline{A}(k), &k\in\mathbb C^+\setminus \Sigma^+,
\\
M(x, t, k)A^{-1}(k), &k\in\mathbb C^-\setminus \Sigma^-,
\end{cases}
\quad
A(k) 
=
\begin{pmatrix}
a(k) & 0
\\
0 & a(k)^{-1}
\end{pmatrix}.
\end{equation}
Then, $\tilde M(k)= \tilde M(x, t, k)$ is analytic in $\Complex\setminus \Sigma$
and  satisfies the jump conditions
\begin{subequations}\label{rhp0til}
\begin{align}
&\tilde M^+( k) = \tilde M^-(k) \tilde V_1^{(0)},\quad k\in \mathbb R,
\\
&\tilde M^+(k) = M^-( k)\tilde V_2^{(0)},\quad k\in B^+,
\\
&\tilde M^+(k) = \tilde M^-(k)\tilde V_3^{(0)},\quad k\in B^-,
\end{align}
and the normalization condition
\begin{equation}
\tilde M(k) =I +O ({1}/{k}),\quad k \to \infty,
\end{equation}
\end{subequations}
where 
\begin{equation}\label{jump0til}
\tilde V_1^{(0)} = A V_1^{(0)} \overline{A}, 
\quad \
\tilde V_2^{(0)} = \big(\overline{A^-}\big)^{-1} V_2^{(0)} \overline{A^+},
\quad \
\tilde V_3^{(0)} = A^- V_3^{(0)} \left(A^+\right)^{-1}.
\end{equation}
The advantage of considering $\tilde M$ instead of $M$ for $x>4\sqrt2q_{o} t$ is that, in contrast to the jumps $V_1^{(0)}$, $V_2^{(0)}$ and $V_3^{(0)}$, 
the  jumps $\tilde V_1^{(0)}$, $\tilde V_2^{(0)}$ and $\tilde V_3^{(0)}$
can be factorized in a way that eventually leads to a Riemann-Hilbert problem 
with a \textit{constant} jump across $B$, just like in the case $x<-4\sqrt2q_{o} t$.
%
Indeed, we may now use the factorizations 
\begin{equation}\label{v10factil}
\widetilde 
V_1^{(0)} =
\begin{cases}
\tilde V_2^{(1)}\tilde V_0^{(1)}\tilde V_1^{(1)},  &k_\re>k_2,
\\
\tilde V_4^{(1)}\tilde V_3^{(1)},  &k_\re<k_2,
\end{cases}
\end{equation}
where for  $\rho(k) = -b(k)/a(k)$ we define
\begin{subequations}\label{v112345til}
\begin{equation}
\tilde V_0^{(1)} =
\begin{pmatrix}
\dfrac{1}{1+\rho\bar \rho}
&
0
\\
0
&
1+\rho\bar \rho
\end{pmatrix},
\quad
\tilde V_1^{(1)}
=
\begin{pmatrix}
d^{-\frac 12}
&
0
\\
\dfrac{d^{-\frac 12} \rho e^{-2i \theta t } }{1+\rho\bar \rho}
&
d^{\frac 12}
\end{pmatrix},
\end{equation}
\begin{equation}
\tilde V_2^{(1)} =
\begin{pmatrix}
d^{-\frac 12} & \dfrac{d^{-\frac 12}\, \bar \rho e^{2i \theta t } }{1+\rho\bar \rho} \\
0 & d^{\frac 12}
\end{pmatrix},
\quad
\tilde V_3^{(1)}
=
\begin{pmatrix}
d^{-\frac 12}
&
d^{\frac 12}\bar \rho e^{2i \theta t } 
\\
0
&
d^{\frac 12}
\end{pmatrix},
\quad
\tilde V_4^{(1)}
=
\begin{pmatrix}
d^{-\frac 12}
&
0
\\
d^{\frac 12}\rho e^{-2i \theta t } 
&
d^{\frac 12}
\end{pmatrix}.
\end{equation}
\end{subequations}
Also, since the branch cut $B$ lies to the left of $k_2$ (see Figure \ref{f:imtheta}\textcolor{blue}{d}), the appropriate factorizations for $\tilde V_2^{(0)}$ and $\tilde V_3^{(0)}$ are
\begin{subequations}\label{cut1til}
\begin{align}
\tilde V_2^{(0)}
&=
\big(\tilde V_{3-}^{(1)}\big)^{-1}
\tilde V_B
\tilde V_{3+}^{(1)},\quad k\in B^+,
\\
\tilde V_3^{(0)}
&=
\tilde V_{4-}^{(1)}
\tilde V_B
\big(\tilde V_{4+}^{(1)}\big)^{-1},\quad k\in B^-,
\end{align}
where the \textit{constant}  matrix $\tilde V_B$ is defined by
\begin{equation}
\tilde V_B
=
\begin{pmatrix}
0
&
{q_+}/{iq_{o} }
\\
{\bar q_+}/{iq_{o} }
&
0
\end{pmatrix}.
\end{equation}
\end{subequations}

\begin{remark}
The factorizations \eqref{v10factil} and~\eqref{cut1til} are completely analogous to the factorizations~\eqref{v10fac} and~\eqref{cut1m}. 
Moreover, it is straightforward to see that $(M_1(x,t))_{12} = (\tilde M_1(x,t))_{12}$
since $a(k) = 1 + O(1/k^2)$ as $k\to\infty$ (e.g., see \cite{bk2014}).
Hence, the rescaling does not affect the reconstruction formula~\eref{qsol} of the potential.
As a result, once the rescaling~\eref{mrhtil} has been performed,
the steps required for the implementation of the Deift-Zhou method in the two cases $x<-4\sqrt2q_{o} t$ and $x>4\sqrt2q_{o} t$ are identical.
Therefore, we hereafter limit ourselves to discussing the case $x<-4\sqrt2q_{o} t$ for brevity.
\end{remark}

%
%
%
%
%
\paragraph{First deformation.}
Define $M^{(1)} = M^{(1)}(x, t, k)$  in terms of the solution $M^{(0)}= M$ of 
the Riemann-Hilbert problem \ref{rhp0} 
according to Figure \ref{figdef1},
with the matrices $V_0^{(1)},\dots,V_B$ given by equations \eqref{v112345} and \eqref{v18}.
Then $M^{(1)}$ is analytic in $ \mathbb C\setminus ( B\cup L_0\cup L_1\cup L_2\cup_{j=1}^4 L_{3, j}\cup L_{4, j})$ and satisfies the jump conditions
\begin{subequations}\label{rhp1}
\begin{align}
&M^{(1)+}( k) =M^{(1)-}( k)\, V_B,\hskip 1.36cm k\in B,
\\
&M^{(1)+}( k) =M^{(1)-}( k)\, V_j^{(1)},\hskip 1.15cm k\in L_j,\ j=0,1,2,
\\
&M^{(1)+}( k) =M^{(1)-}( k)\, V_3^{(1)},\hskip 1.15cm k\in L_{3, 1}\cup L_{3, 2},
\\
&M^{(1)+}( k) =M^{(1)-}( k)\, \big(V_3^{(1)}\big)^{-1},\quad k\in L_{3, 3}\cup L_{3, 4},
\\
&M^{(1)+}( k) =M^{(1)-}( k)\, V_4^{(1)},\hskip 1.15cm k\in L_{4, 1}\cup L_{4,3},
\\
&M^{(1)+}( k) =M^{(1)-}( k)\, \big(V_4^{(1)}\big)^{-1},\quad k\in L_{4, 2}\cup L_{4,4},
\end{align}
with the jump contours $L_j$ as in Figure \ref{figdef1},
and the normalization condition
\begin{equation}
M^{(1)}(k) =I +O(1/k), \quad k \to \infty\,.
\end{equation}
\end{subequations}
Note that the jump across $(k_1, \infty)$ has been eliminated as a result of the transformation.
%
%
%
%
%
Furthermore, note that all the jump contours of Figure \ref{figdef1}
apart from $(-\infty, k_1)$ and the branch cut $B$
can  be deformed to a single contour
which intersects with  the continuous spectrum $\Sigma$ only at $k_1$,
as shown in Figure~\ref{figdef2}.
Near the branch point $iq_{o} $, in particular, we first deform the contours $L_{3,2}$ and $L_{3, 3}$ as shown in Figure \ref{figdefcut} 
and then further deform the contours $\cup_{j=1}^4 L_{3, j}$ of Figure~\ref{figdef1} to the contour $L_3$ of Figure \ref{figdef2}, which does not go through the origin. 
The point $-iq_{o} $ is handled in an analogous way, so that the contours $\cup_{j=1}^4 L_{4, j}$ of Figure~\ref{figdef1} are deformed to the contour $L_4$ of Figure \ref{figdef2}.
\begin{figure}[ht!]
\begin{center}
\includegraphics[scale=1]{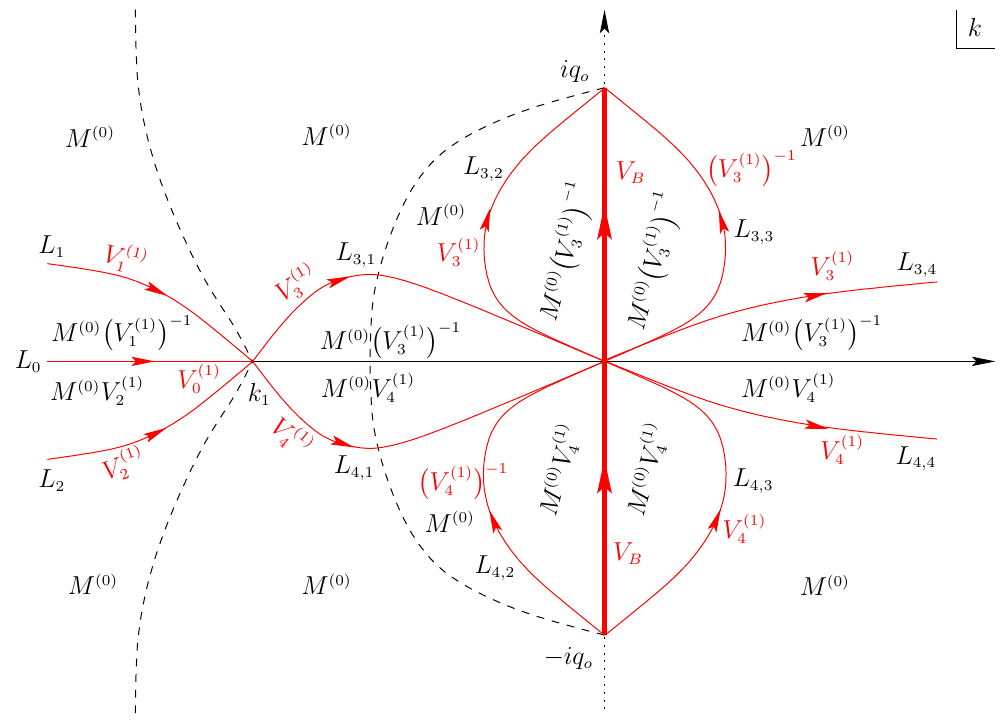}
\caption{The first deformation in the plane wave region.}
\label{figdef1}
\end{center}
\end{figure}
\begin{figure}[ht!]
\begin{center}
\includegraphics[scale=1]{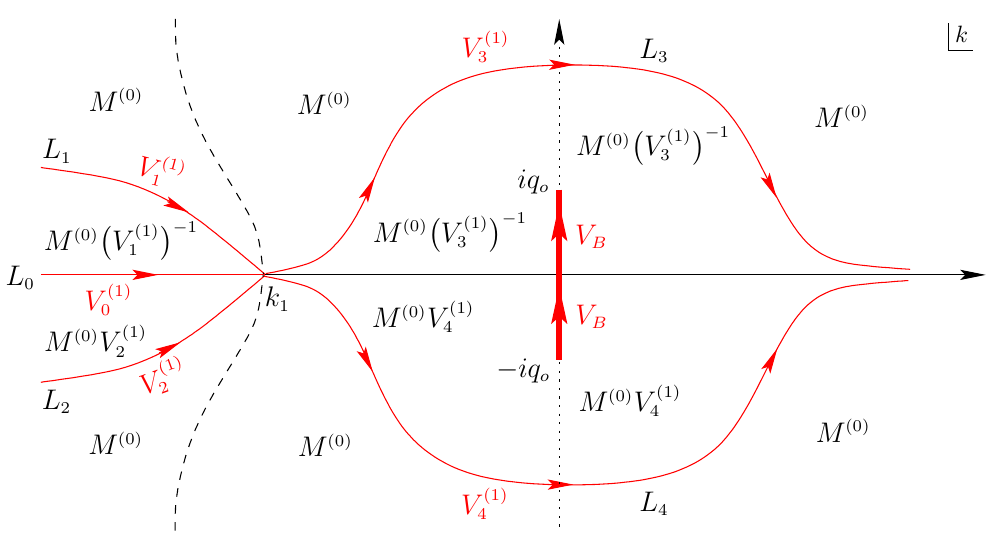}
\caption{Final form of the Riemann-Hilbert problem for $M^{(1)}$ in the plane wave region.}
\label{figdef2}
\end{center}
\end{figure}
\begin{figure}[ht]
\begin{center}
\includegraphics[scale=0.72]{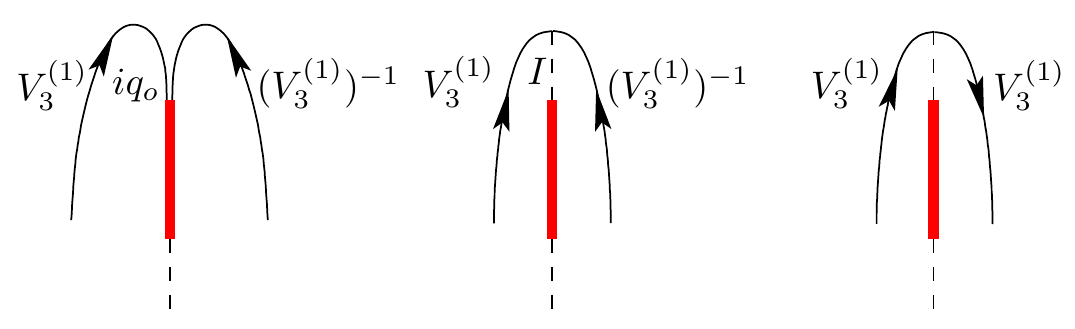}
\caption{The sequence of deformations around the branch point $iq_{o} $.}
\label{figdefcut}
\end{center}
\end{figure}

%
%
%
%
%
\paragraph{Second deformation.}
The purpose of this deformation is to eliminate the jump across $(-\infty, k_1)$.
This goal is accomplished
by introducing an auxiliary scalar function $\delta=\delta(k)$ which is  
analytic in $\mathbb C\setminus (-\infty, k_1)$ and satisfies the jump condition
\begin{subequations}\label{rhpdelta}
\begin{equation}
\delta^+(k) = \delta^-(k)\left[1+r(k)\bar r(k)\right],\quad k\in (-\infty, k_1),\label{deljump}
\end{equation}
and the normalization condition
\begin{equation}
\delta(k) =1+O(1/k),\quad k \to \infty.
\end{equation}
\end{subequations}
The solution of this scalar Riemann-Hilbert problem  is obtained in explicit form via the Plemelj formulae 
\begin{equation}\label{deldef}
\delta(k)
=
\exp
\left\{
\frac{1}{2i\pi}
\int_{-\infty}^{k_1}
\frac{\ln\!\big[1+r(\nu )\bar r(\nu )\big]}{\nu -k}\, d\nu 
\right\},\quad k\notin(-\infty, k_1).
\end{equation} 
Then, the function $M^{(2)}$ defined by
\begin{equation}\label{m2def}
M^{(2)}(x, t, k)
=
M^{(1)}(x, t, k) \delta(k)^{-\sigma_3}
\end{equation}
does not have a discontinuity across $(-\infty, k_1)$ since
$M^{(2)+}(k)
=
M^{(2)-}(k)
\left(\delta^{-}\right)^{\sigma_3}
V_0^{(1)}
\left(\delta^{+}\right)^{-\sigma_3}
=
M^{(2)-}(k)$ 
for all $k\in (-\infty, k_1)$.
Moreover, the jumps off the real $k$-axis become
\begin{subequations}\label{jump2b}
\begin{equation}
V_B^{(2)}
=
\begin{pmatrix}
0
&
\dfrac{q_-}{iq_{o} }\,\delta^{2}
\\
\dfrac{\bar q_-}{iq_{o} }\,\delta^{-2}
&
0
\end{pmatrix},
\quad
V_1^{(2)}
=
\begin{pmatrix}
d^{-\frac 12}
&
\dfrac{d^{\frac 12}\, \bar r e^{2i \theta t } }{1+r\bar r}\,\delta^{2}
\\
0
&
d^{\frac 12}
\end{pmatrix},
\quad
V_2^{(2)}
=
\begin{pmatrix}
d^{-\frac 12}
&
0
\\
\dfrac{d^{\frac 12} r e^{-2i \theta t } }{1+r\bar r}\,\delta^{-2}
&
d^{\frac 12}
\end{pmatrix},
\end{equation}
\begin{equation}
V_3^{(2)}
=
\begin{pmatrix}
d^{-\frac 12}
&
0
\\
d^{-\frac 12}\,r e^{-2i \theta t } \,\delta^{-2}
&
d^{\frac 12}
\end{pmatrix},
\quad
V_4^{(2)}
=
\begin{pmatrix}
d^{-\frac 12}
&
d^{-\frac 12}\,\bar r e^{2i \theta t } \,\delta^{2}
\\
0
&
d^{\frac 12}
\end{pmatrix},
\end{equation}
\end{subequations}
and, since 
$\delta$ is analytic away from $(-\infty, k_1)$,
no additional jumps are introduced.
Overall, the function $M^{(2)}$ is analytic in $ \mathbb C\setminus ( \cup_{j=1}^4 L_j \cup B )$ and satisfies the jump conditions
\begin{subequations}\label{rhp2}
\begin{align}
&M^{(2)+}( k) = M^{(2)-}( k) V_B^{(2)},\quad k\in  B,
\\
&M^{(2)+}( k) = M^{(2)-}( k) V_j^{(2)},\quad k\in L_j,\ j=1,2,3,4,
\end{align}
with the jump contours $L_j$ as shown in Figure \ref{figdef30}, and the normalization condition
\begin{equation}
M^{(2)}(k) =I +O(1/k),\quad k \to \infty.
\end{equation}
\end{subequations}

\begin{figure}
\begin{center}
\includegraphics[scale=1]{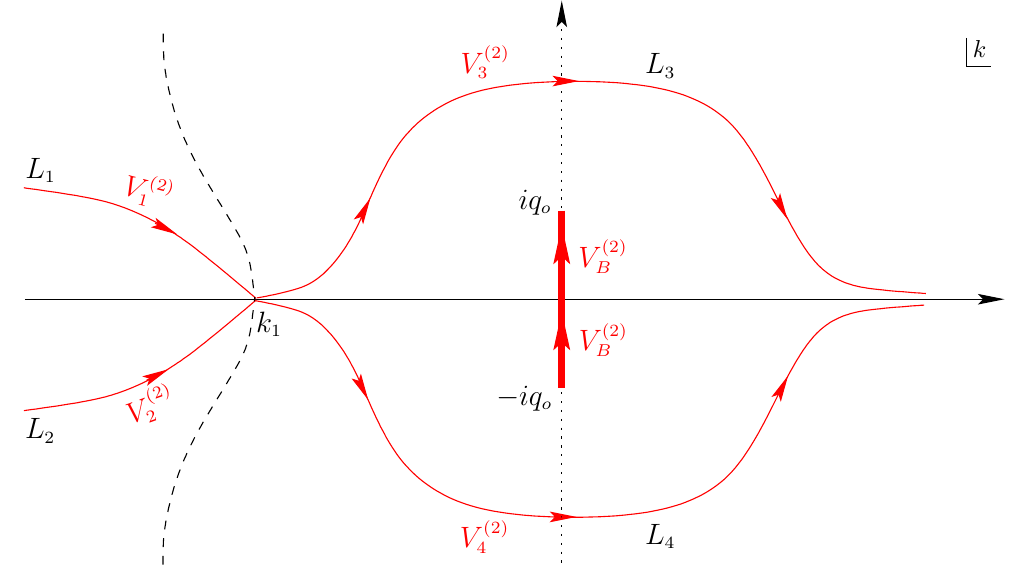}
\caption{The jumps of $M^{(2)}$ in the plane wave region.}
\label{figdef30}
\end{center}
\end{figure}

\begin{figure}
\begin{center}
\includegraphics[scale=1]{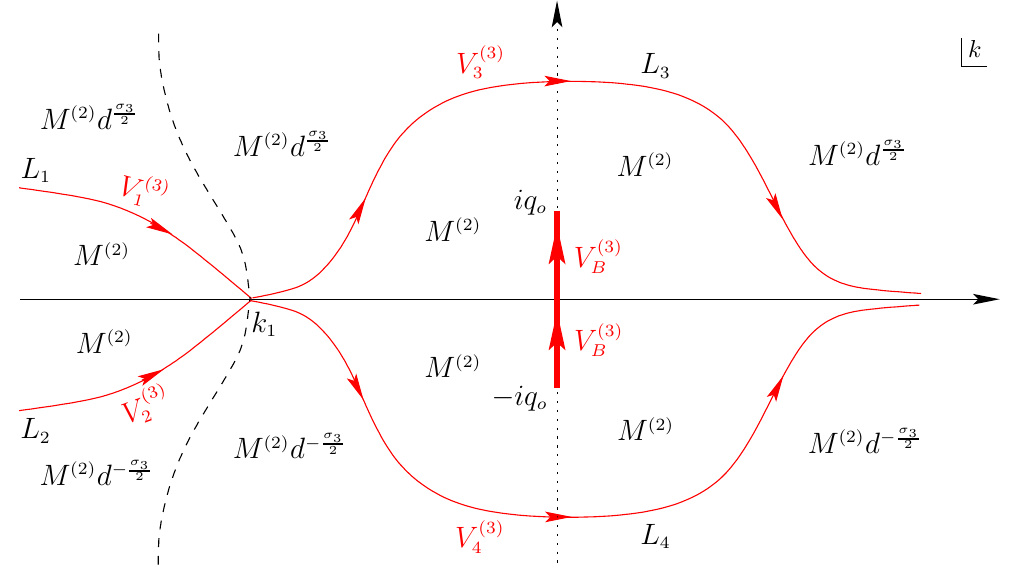}
\caption{The third deformation in the plane wave region.}
\label{figdef3}
\end{center}
\end{figure}

%
%

\paragraph{Third deformation.}
The next step consists in removing the function $d$ from the jump matrices \eqref{jump2b} so that the jumps along the contours $L_j$ eventually tend to the identity as $t\to\infty$.
This is accomplished by switching from $M^{(2)}$ to $M^{(3)}$ according to Figure \ref{figdef3}.
It then follows that
\begin{align*}
&M^{(3)+}(k)
=
M^{(3)-}(k)V_j^{(2)}
d^{\frac{\sigma_3}{2}}, \quad k\in L_j,\ j=1,3,
\\
&M^{(3)+}(k)
=
M^{(3)-}(k) d^{\frac{\sigma_3}{2}}V_j^{(2)}, \quad k\in L_j,\ j=2,4\,.
\end{align*}
Furthermore, the jump $V_B^{(2)}$ of $M^{(2)}$ across $B$, which  does not involve $d$, remains the same for $M^{(3)}$. Overall, $M^{(3)}$ is analytic in 
$\mathbb C\setminus ( \cup_{j=1}^4 L_j \cup B )$ and satisfies the jump conditions
\begin{subequations}\label{rhp3}
\begin{align}
&M^{(3)+}( k) = M^{(3)-}( k)V_B^{(3)},\quad k\in B,
\\
&M^{(3)+}( k) = M^{(3)-}( k)V_j^{(3)},\quad k\in L_j,\ j=1,2,3,4,
\end{align}
and the normalization condition
\begin{equation}
M^{(3)}( k) =I +O(1/k),\quad k \to \infty,
\end{equation}
\end{subequations}
where the jump contours $L_1, L_2, L_3, L_4$ are shown in Figure \ref{figdef3} and
\begin{subequations}\label{jump3}
\begin{equation}
V_B^{(3)} =V_B^{(2)},
\quad
V_1^{(3)}
=
\begin{pmatrix}
1
&
\dfrac{\bar r e^{2i \theta t } }{1+r\bar r}\,\delta^{2}
\\
0
&
1
\end{pmatrix},
\quad
V_2^{(3)}
=
\begin{pmatrix}
1
&
0
\\
\dfrac{r e^{-2i \theta t } }{1+r\bar r}\,\delta^{-2}
&
1
\end{pmatrix},
\end{equation}
\begin{equation}
V_3^{(3)}
=
\begin{pmatrix}
1
&
0
\\
r e^{-2i \theta t } \,\delta^{-2}
&
1
\end{pmatrix},
\quad
V_4^{(3)}
=
\begin{pmatrix}
1
&
\bar r e^{2i \theta t } \,\delta^{2}
\\
0
&
1
\end{pmatrix}.
\end{equation}
\end{subequations}
%
%

%
%
%
\paragraph{Fourth deformation.}
Next, we eliminate the auxiliary function $\delta$  from the jump $V_B^{(3)}$  across the branch cut $B$ 
--- and hence turn this jump into a constant ---
by employing the so-called $g$-function mechanism \cite{{dvz1994}, {dvz1997}}.
More precisely,  we let
\begin{equation}\label{m4def}
M^{(4)}(x,t,k)
=
M^{(3)}(x,t,k)
e^{ig(k)\sigma_3},
\end{equation}
where the yet to be determined scalar function $g$ is analytic in $\mathbb C\setminus B$
and satisfies the discontinuity condition
\begin{equation}\label{jump-g}
 e^{i\left(g^+(k)+g^-(k)\right)}=\delta(k)^2,\quad k\in B.
\end{equation}
It is straightforward to check that if condition \eqref{jump-g} holds then the jump of $M^{(4)}$ across $B$ is constant and equal to
\begin{subequations}\label{m4pw-jumps}
\begin{equation}\label{jump4a}
V_B^{(4)} = V_B =
\begin{pmatrix}
0 & {q_-}/{iq_{o} } \\
{\bar q_-}/{iq_{o} } & 0
\end{pmatrix},\quad k\in B.
\end{equation}
Moreover, since $g$ is analytic away from $B$ no new jumps are introduced. 
The remaining jump matrices in equation~\eqref{jump3} are changed into
\begin{equation}\label{jump4b}
V_1^{(4)}
=
\begin{pmatrix}
1 & \dfrac{\bar r e^{2i(\theta t -g)} }{1+r\bar r}\,\delta^{2}
\\
0 & 1
\end{pmatrix},
\quad
V_2^{(4)}
=
\begin{pmatrix}
1 & 0 \\
\dfrac{r e^{-2i(\theta t -g)} }{1+r\bar r}\,\delta^{-2} & 1
\end{pmatrix},
\end{equation}
\begin{equation}\label{jump4c}
V_3^{(4)}
=
\begin{pmatrix}
1
&
0
\\
r e^{-2i(\theta t -g)} \delta^{-2}
&
1
\end{pmatrix},
\quad
V_4^{(4)}
=
\begin{pmatrix}
1
&
\bar r e^{2i(\theta t -g)}  \delta^{2}
\\
0
&
1
\end{pmatrix},
\end{equation}
\end{subequations}

The  function $g$  can be determined explicitly as follows. 
Dividing  condition \eqref{jump-g} by $\lambda$  and using  equation \eqref{deldef} for $\delta$,  we deduce that the scalar function $g/\lambda=(g/\lambda)(k)$ is analytic in $\mathbb C\setminus B$ and satisfies the jump condition
\begin{subequations}\label{rhpg}
\begin{equation}
\left(\frac{g}{\lambda}\right)^+(k)  -
\left(\frac{g}{\lambda}\right)^-(k)
=
\frac{1}{\pi \lambda}
\int_{-\infty}^{k_1}
\frac{\ln\!\big[1+r(\nu )\bar r(\nu )\big]}{\nu -k}\, d\nu , \quad k\in B,
\end{equation}
and the normalization condition
\begin{equation}
\frac{g}{\lambda}(k)
=
O(1/k), \quad k \to\infty.
\end{equation}
\end{subequations}
Plemelj's formulae then yield $g$ in the explicit form
\begin{equation}\label{gdef}
g(k)
=
\frac{\lambda(k)}{2i\pi^2}
\int_{\zeta\in B}
\frac{1}{\lambda(\zeta)(\zeta-k)}
\int_{-\infty}^{k_1}
\frac{\ln\!\big[1+r(\nu )\bar r(\nu )\big]}{\nu -\zeta}\, d\nu 
d\zeta,\quad k\notin B.
\end{equation}
In summary, the function $M^{(4)}$  defined by equation \eqref{m4def}, with $g$ given by equation \eqref{gdef}, satisfies the following
Riemann-Hilbert problem:
\begin{rhp}[Final problem in plane wave region]\label{rhp4}
Determine a sectionally analytic matrix-valued function $M^{(4)}(k)=M^{(4)}(x, t, k)$ in $\mathbb C\setminus ( \cup_{j=1}^4 L_j \cup B)$ satisfying the jump conditions
\begin{subequations}\label{rhp44}
\begin{align}
&M^{(4)+}(k) = M^{(4)-}( k) V_B,\hskip .6cm k\in  B,
\\
&M^{(4)+}(k) = M^{(4)-}( k) V_j^{(4)},\quad k\in L_j,\ j=1,2,3,4,
\end{align}
and the normalization condition
\begin{equation}
M^{(4)}( k) =\left[I +O(1/k)\right]e^{ig_\infty \sigma_3},\quad k \to \infty,
\end{equation}
\end{subequations}
where  the jump contours $L_j$ are shown in Figure \ref{figdef3}, the jump matrices are given by equations \eqref{m4pw-jumps},  and the real constant $g_\infty$ is given by
\begin{equation}\label{ginfdef}
g_\infty 
=
-\frac{1}{2i\pi^2}
\int_{\zeta\in B}
\frac{1}{\lambda(\zeta)}
\int_{-\infty}^{k_1}
\frac{\ln\!\big[1+r(\nu )\bar r(\nu )\big]}{\nu -\zeta}\, d\nu 
d\zeta.
\end{equation}
\end{rhp}

%
%
%
%
%
\paragraph{Decomposition of the final Riemann-Hilbert problem and limit $t\to \infty$ .}
Starting from formula \eqref{qsol} for the solution $q$ of the  focusing NLS equation~\eqref{e:NLS} in terms of $M^{(0)}$ and applying the four successive deformations that lead to $M^{(4)}$, we obtain 
\begin{equation}\label{qsolm4}
q(x, t)
=
-2i \big(M_1^{(4)}(x, t)\big)_{12}e^{ig_\infty },
\end{equation}
where $M_1^{(4)}$ is the $O(1/k )$ coefficient of the large-$k$ expansion of $M^{(4)}$:
\begin{equation}
M^{(4)}(x, t, k) = e^{ig_\infty \sigma_3} + \frac{M_1^{(4)}(x, t)}{k}+O \Big(\frac{1}{k^2}\Big),\quad k \to \infty.
\end{equation}
In order to be able to take the limit $t\to \infty$ in   formula \eqref{qsolm4}, a suitable decomposition of $M^{(4)}$ is first required
into an asymptotic problem, which will yield the leading-order contribution to the solution of the NLS equation, 
and an error problem, which will yield the leading-order error.

%
%
%
%
%
%
It is evident from the structure of the jump matrices $V_1^{(4)}$, $V_2^{(4)}$, $V_3^{(4)}$ and $V_4^{(4)}$ in equations~\eqref{jump4b} and~\eqref{jump4c} 
and the sign structure of $\Re(i\theta)$ (see Figure \ref{f:imtheta}\textcolor{blue}{a}) that $k_1$ is the only point at which 
the above matrices fail to tend uniformly to $I$ as $t\to \infty$.
Hence, a neighborhood of the point $k_1$  
--- in addition, of course, to the branch cut $B$ ---
is the only region expected to yield the leading-order contribution to the long-time asymptotics of $M^{(4)}$, 
while the remaining contours are expected to contribute only in the error. 
This reasoning motivates the following 
decomposition of the Riemann-Hilbert problem~\ref{rhp4}:

Let 
$D_{k_1}^{\varepsilon}$ be a disk of radius $\varepsilon $ centered at $k_1$, with $\varepsilon $ sufficiently small so that $D_{k_1}^\varepsilon  \cap B =\varnothing$. Then,  write
\begin{equation}\label{m4ea}
M^{(4)}
=
M^\err M^\asymp,
\quad 
M^\asymp
=
\begin{cases}
M^B, \quad k\in \mathbb C\setminus D_{k_1}^{\varepsilon },
\\
M^D, \quad k\in D_{k_1}^{\varepsilon },
\end{cases}
\end{equation}
where:
\begin{itemize}
\item
the function $M^B(k)=M^B(x, t, k)$ is analytic in $\mathbb C\setminus B$ and satisfies the jump condition
\begin{subequations}\label{rhpB}
\begin{equation}
M^{B+}(k) = M^{B-}(k) \,
V_B, \quad k\in B,
\label{rhpBb}
\end{equation}
and the normalization condition
\begin{equation}
M^B(k) =  \left[I +O(1/k)\right]e^{ig_\infty \sigma_3},\quad k \to \infty,
\end{equation}
\end{subequations}
\item
the function $M^D$ is analytic  in $D_{k_1}^{\varepsilon }\setminus L_j$, $j=1,2,3,4,$ with jumps
\begin{equation}\label{rhpP}
M^{D+}(k)
=
M^{D-}(k)
V_j^{(4)},
\quad
k\in \hat L_j = L_j\cap D_{k_1}^{\varepsilon },\ j=1, 2, 3, 4, 
\end{equation}
\item
 the function $M^\err(k)=M^\err(x, t, k)$ is analytic in $\mathbb C
\setminus ( \cup_{j=1}^4 \check L_j \cup  \p D_{k_1}^{\varepsilon})$ and satisfies the jump condition
\begin{subequations}\label{rhpE}
\begin{equation}
M^{\err+}(k) = M^{\err-}(k) \,
V^\err,\quad  k\in \cup_{j=1}^4 \check L_j \cup  \p D_{k_1}^{\varepsilon},
\quad
\check L_j= L_j \setminus \overline{D_{k_1}^\varepsilon},
\label{rhpEb}
\end{equation}
and the normalization condition
\begin{equation}
M^\err(k) = I +O(1/k),\quad k \to \infty,
\end{equation}
\end{subequations}
with  the jump $V^\err$ defined by
\begin{equation}\label{VEdef}
V^\err
=
\begin{cases}
M^B V_j^{(4)} (M^B)^{-1}, &k\in \check L_j,
\\
M^{\asymp-}(V_D^\asymp)^{-1}(M^{\asymp-})^{-1}, &k\in \p D_{k_1}^\varepsilon,
\end{cases}
\end{equation}
where  $V_D^\asymp$ is the  jump of $M^\asymp$ across the circle $\p D_{k_1}^\varepsilon$, which is yet unknown. 
\end{itemize}

Under the decomposition \eqref{m4ea} of $M^{(4)}$, formula \eqref{qsolm4} becomes
\begin{equation}\label{qsolmea}
q(x, t)
=
-2i \left(M_1^{B}(x, t)e^{ig_\infty }+M_1^\err(x, t)\right)_{12}.
\end{equation}
Moreover, thanks to the fact that the jump matrix $V_B$ is constant, 
$M^B$ is  actually  given by the explicit formula
\begin{equation}\label{mlsol}
M^B(k)
=
e^{ig_\infty \sigma_3}
\begin{pmatrix}
\dfrac{1}{2}\left(\Lambda+\Lambda^{-1}\right)
&
-\dfrac{q_{o} }{2\bar q_-}\left(\Lambda-\Lambda^{-1}\right)
\\
\dfrac{q_{o} }{2q_-}\,\overline{\left( \Lambda- \Lambda^{-1}\right)}
&
\dfrac{1}{2}\, \overline{\left(\Lambda+ \Lambda^{-1}\right)}
\end{pmatrix},
\end{equation}
where the function $\Lambda$ is defined by
\begin{equation}\label{Lamdef}
\Lambda(k)  = \left(\frac{k-iq_{o} }{k+iq_{o} }\right)^{\frac 14}.
\end{equation}
In addition, in the Appendix we show that $M_1^\err$ admits the following estimate:
\begin{equation}\label{m1e-est1}
\left| M_1^\err(x, t)\right|
=
O\big(t^{-\frac 12}\big),
\quad t\to \infty.
\end{equation}
Thus, returning to equation \eqref{qsolmea} we conclude that
 the  long-time asymptotic behavior of the solution of the  focusing NLS equation \eqref{e:NLS} in the plane wave region is  given  to leading order by
\begin{equation*}
q(x, t)
=
-2i \left(M_1^{B}(x, t)\right)_{12}
e^{ig_\infty }+O\big(t^{-\frac 12}\big),
\quad  t\to \infty,
\end{equation*}
or, explicitly,
\begin{equation}\label{qsol-genus0}
q(x,t) =
e^{2i g_\infty }q_-
+O\big(t^{-\frac 12}\big),
\quad t\to \infty.
\end{equation}
Note that  the definitions \eqref{kroots} and \eqref{ginfdef} of $k_1$ and $g_\infty $ imply that $\lim_{\xi\to -\infty}g_\infty =0$. Hence, in consistency 
with the infinity condition \eqref{e:NZBC}, in the plane wave region  
as $t\to \infty$ we have
$q(x, t) 
\xrightarrow{\hspace*{.3cm}}  q_-$ 
in the limit $x\to -\infty$, $\xi\to -\infty$. 

The proof of Theorem \ref{pw-t} is complete.


\section{The modulated elliptic wave region: proof of theorem \ref{mew-t}}\label{mew-reg}
\label{mew-sec}

In this section we prove Theorem \ref{mew-t}, i.e., we compute the leading order long-time asymptotic behavior of the solution of the  focusing NLS equation \eqref{e:NLS} in the modulated elliptic wave region $|x|<4\sqrt 2 q_{o}  t$. 

Recall that the sign structure of $\Re(i\theta)$ in this region  was discussed in Section \ref{pw-reg} and is depicted in Figures~\ref{f:imtheta}\textcolor{blue}{b} and \ref{f:imtheta}\textcolor{blue}{c}.
As in the plane wave region, we consider only the case $-4\sqrt2q_{o} t<x<0$ since the case $0<x<4\sqrt2q_{o} t$ is entirely analogous after suitably reformulating the Riemann-Hilbert problem~\ref{rhp0} (see discussion below Remark \ref{xpm-rem}).

\begin{remark}
The main difference between the plane wave and the modulated elliptic wave regions is the absence of real stationary points in the latter case. 
As a result, the curves identifying the sign changes of $\Re(i\theta)$ in Figure~\ref{f:imtheta}\textcolor{blue}{b} as $k\to\infty$ connects directly to the branch cut $B$. 
This implies that it is not possible anymore to use the previous factorizations and deformations to lift the contours off the real $k$-axis in such a way that the corresponding jump matrices remain bounded as $t\to\infty$.
In other words, in order to connect the negative real $k$-axis to the white region in the upper half-plane, one cannot avoid passing through the gray region,
in which $\Re(i\theta)$ has the ``wrong'' sign.
To circumvent this problem, 
similarly to \cite{bv2007}
we will introduce an appropriate, artificial change-of-factorization point $k_{o} \in\Real^-$, 
which will be determined as part of the problem.
\end{remark}

%
%
%
%
%
\paragraph{First, second and third deformations.}
The first deformation, which is defined in Figure \ref{defr1}, is the same as the first deformation in the plane wave region, except that the change of factorization now occurs  at the point $k_{o} $ instead of the point $k_1$.
The second and third deformations are also the same as the corresponding deformations in the plane wave region and lead to the function $M^{(3)}(k)=M^{(3)}(x, t, k)$ which is analytic in 
$\mathbb C\setminus( \cup_{j=1}^4 L_j \cup  B)$ and satisfies the jump conditions
\begin{subequations}\label{rhp3r}
\begin{align}
&M^{(3)+}(k) = M^{(3)-}(k) V_B^{(3)},\quad k\in B,
\\
&M^{(3)+}(k) = M^{(3)-}(k) V_j^{(3)},\quad k\in L_j,\ j=1,2,3,4,
\label{m3jumpr}
\end{align}
and the normalization condition
\begin{equation}
M^{(3)}(k) =I +O(1/k),\quad k \to \infty,
\end{equation}
\end{subequations}
with the jump contours $L_j$  shown in Figure \ref{defr2} and the   jump matrices $V_j^{(3)}$  given by equations \eqref{jump3} after changing the definition of the function $\delta$ to
\begin{equation}\label{deldefr}
\delta(k)
=
\exp
\left\{
\frac{1}{2i\pi}
\int_{-\infty}^{k_{o} }
\frac{\ln\!\big[1+r(\nu )\bar r(\nu )\big]}{\nu -k}\, d\nu 
\right\},\quad k\notin(-\infty, k_{o} ).
\end{equation}

\begin{figure}[b!]
\begin{center}
\includegraphics[scale=1]{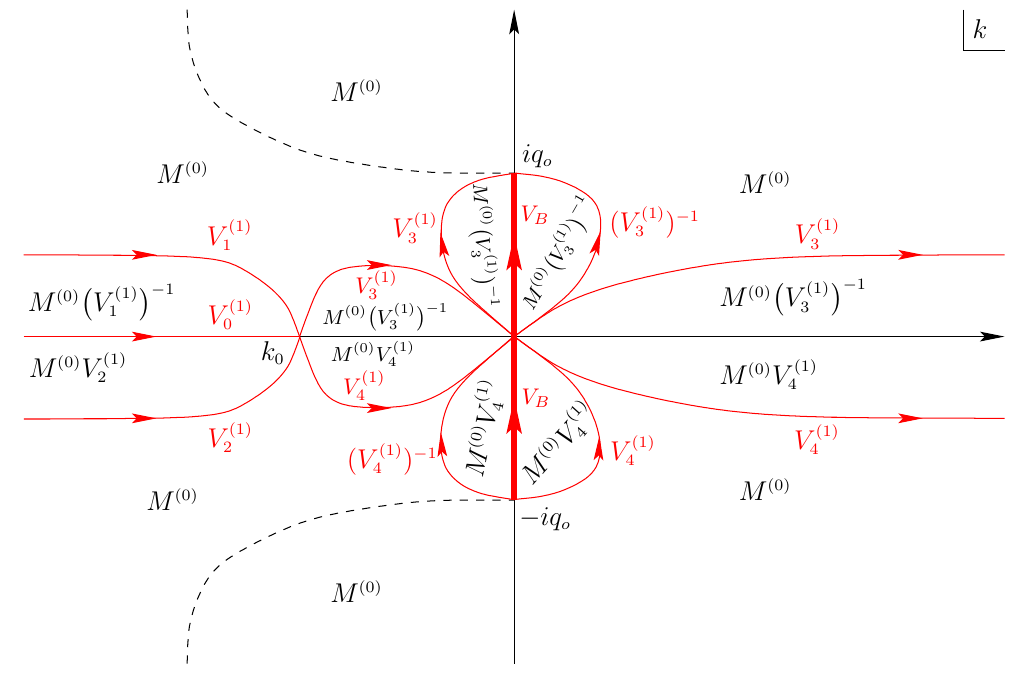}
\caption{The first deformation in the modulated elliptic wave region.}
\label{defr1}
\end{center}
\end{figure}
\begin{figure}[t!]
\begin{center}
\includegraphics[scale=1]{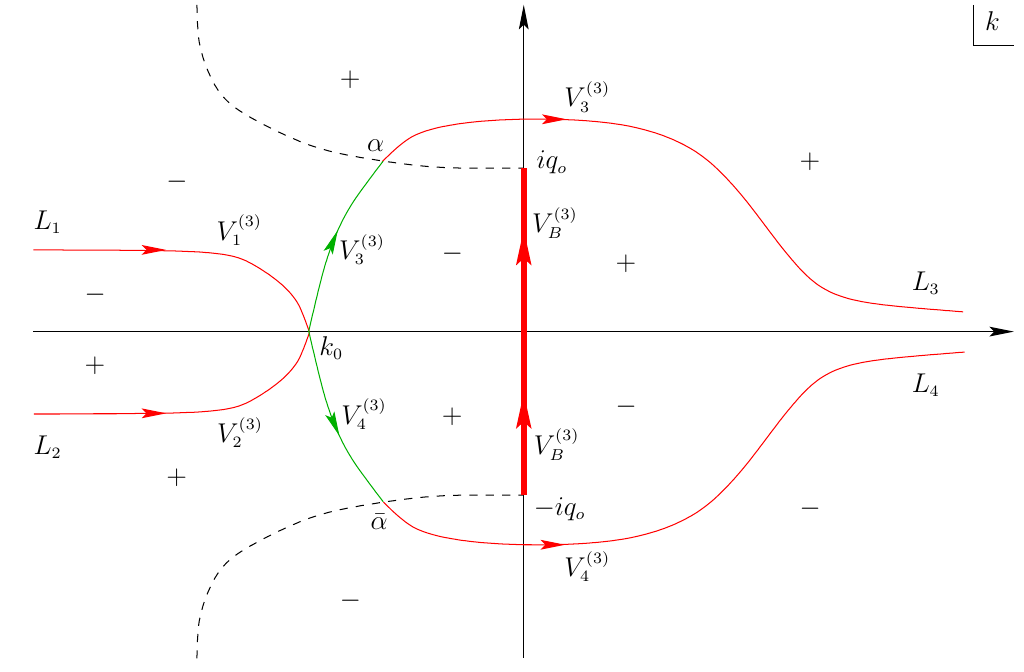}
\caption{The jumps of $M^{(3)}$ in the modulated elliptic wave region. 
The jump matrices $V_3^{(3)}$ and $V_4^{(3)}$  grow exponentially as a function of $t$ on the green-colored contours.}
\label{defr2}
\end{center}
\end{figure}

\begin{remark}
Importantly, the contour $L_3$ now intersects the curve $\Re(i\theta)=0$ in the second quadrant at a point $\alpha$,
as depicted in Figure \ref{defr2},
unlike what happens in the plane wave region.
This intersection point will be determined in terms of $k_{o} $ in due course.
\end{remark}

%
%
%
%
%
\paragraph{Fourth deformation: Elimination of the exponential growth.}
So far we have managed to deform the jump contour across $\mathbb R$ to contours in the complex $k$-plane as in the plane wave region. 
We now encounter a new phenomenon which is not present in the plane wave region, however:
the jumps $V_3^{(3)}$ and $V_4^{(3)}$ grow exponentially with $t$ along the green-colored segments of the deformed contours  shown 
in Figure~\ref{defr2}.
To address this issue, we employ the following new factorizations for these jumps:
\begin{equation}\label{gr-facs}
V_3^{(3)}
=
V_5^{(3)} V_7^{(3)} V_5^{(3)},
\qquad
V_4^{(3)}
=
V_6^{(3)} V_8^{(3)} V_6^{(3)},
\end{equation}
where
\begin{subequations}\label{v4r}
\begin{equation}
V_5^{(3)}
=
\begin{pmatrix}
1
&
\dfrac{\delta^{2}}{r}\, e^{2i \theta t } 
\\
0
&
1
\end{pmatrix},
\quad
V_6^{(3)}
=
\begin{pmatrix}
1
&
0
\\
\dfrac{1}{ \bar r\delta^{2}}\, e^{-2i \theta t } 
&
1
\end{pmatrix},
\end{equation}
\begin{equation}
V_7^{(3)}
=
\begin{pmatrix}
0
&
-\dfrac{\delta^{2}}{r}\, e^{2i \theta t } 
\\
\dfrac{r}{\delta^{2}}\, e^{-2i \theta t } 
&
0
\end{pmatrix},
\quad
V_8^{(3)}
=
\begin{pmatrix}
0
&
\bar r\delta^{2}\, e^{2i \theta t } 
\\
-\dfrac{1}{\bar r \delta^{2}}\, e^{-2i \theta t } 
&
0
\end{pmatrix}.
\end{equation}
\end{subequations}
\begin{figure}[t!]
\begin{center}
\includegraphics[scale=1]{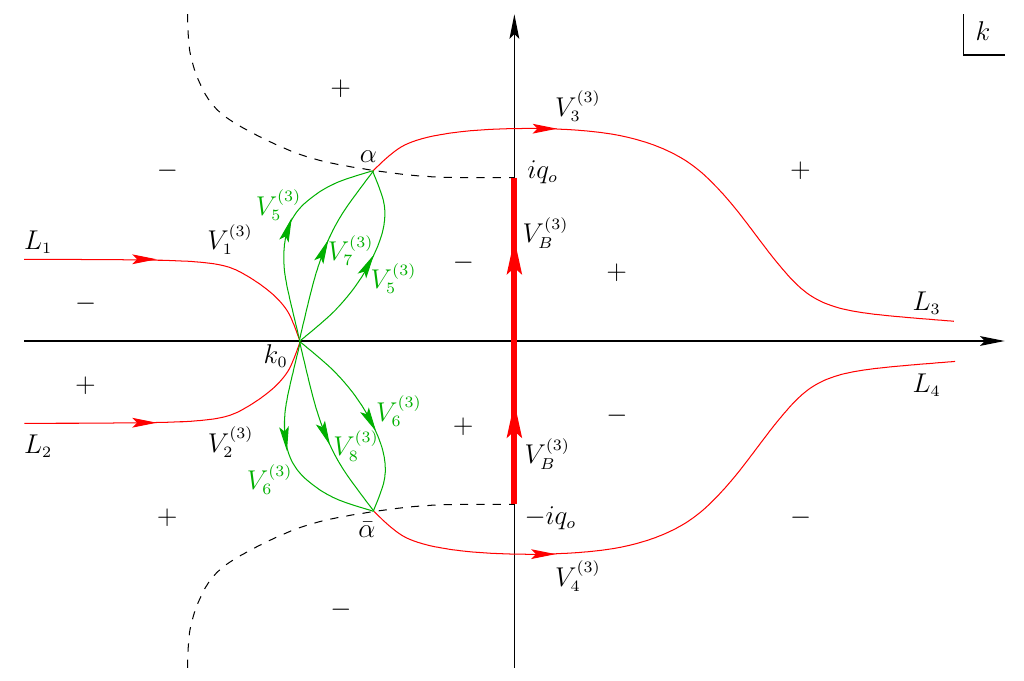}
\caption{The fourth deformation in the modulated elliptic wave region.}
\label{defr3}
\end{center}
\end{figure}
These factorizations allow for the deformation of the green  contours of Figure \ref{defr2} into the green contours shown in Figure \ref{defr3}, 
where for $j=1,\ldots,8$ we denote by $V_j^{(3)}$ the jump associated with the contour $L_j$.

Observe that the jumps $V_5^{(3)}$ and $V_6^{(3)}$ 
are now bounded along the contours $L_5$ and $L_6$ respectively.
On the other hand, the $21$-entry of the jump $V_7^{(3)}$ and the $12$-entry of the jump $V_8^{(3)}$  are still unbounded along the contours $L_7$ and $L_8$ respectively.
To resolve this issue, we employ once again the $g$-function mechanism as in the previous section (cf.\ transformation \eqref{m4def}).
This time, however, we need to do so through a $t$\textit{-dependent} exponential, 
namely, by letting
\begin{equation}\label{m4r}
M^{(4)}(x, t, k)
=
M^{(3)}(x, t, k) e^{-iG(\xi, \alpha, k_{o}, k)t \sigma_3},
\end{equation}
for a function $G(\xi, \alpha, k_{o}, k)$ which is required to be analytic in $\mathbb C\setminus (B\cup L_7\cup L_8)$.
In fact, instead of working with $G$ it will be more convenient to consider the function $h$ defined by
\begin{equation}\label{hdef}
h(\xi, \alpha, k_{o}, k)
=
\theta(\xi, k) + G(\xi, \alpha, k_{o}, k).
\end{equation}
From the above definition, we infer  that $h$ must be analytic in $\mathbb C\setminus(B\cup L_7\cup L_8)$ and has jump
discontinuities across $B$ and $\tilde B$.
%
%
Moreover, according to transformation \eqref{m4r}
the jumps of $M^{(4)}$ read  
\begin{align}\label{jumph}
&
V_B^{(4)}
=
\begin{pmatrix}
0
&
\dfrac{q_-\delta^{2}}{iq_{o} }\, e^{i(h^++h^-)t}
\\
\dfrac{\bar q_-}{iq_{o} \delta^{2}}\, e^{-i(h^++h^-)t}
&
0
\end{pmatrix},
\quad
V_1^{(4)}
=
\begin{pmatrix} 
1 & \dfrac{\bar r\,\delta^{2} }{1+r\bar r}\,e^{2i h t} \\
0 & 1
\end{pmatrix},
\nonumber\\
&
V_2^{(4)}
=
\begin{pmatrix}
1 & 0 \\
\dfrac{r }{\left(1+r\bar r\right) \delta^{2}}\,e^{-2i h t} & 1
\end{pmatrix}, 
\quad
V_3^{(4)}
=
\begin{pmatrix}
1 & 0 \\
\dfrac{r}{ \delta^2}\, e^{-2i h t} & 1
\end{pmatrix}, 
\quad
V_4^{(4)}
=
\begin{pmatrix}
1 & \bar r \delta^{2}\, e^{2i h t} \\
0 & 1
\end{pmatrix},
\nonumber\\
&
V_5^{(4)}
=
\begin{pmatrix}
1
&
\dfrac{\delta^{2}}{r}\, e^{2i h t} 
\\
0
&
1
\end{pmatrix},
\quad
V_6^{(4)}
=
\begin{pmatrix}
1
&
0
\\
\dfrac{1}{ \bar r\delta^{2}}\, e^{-2i h t} 
&
1
\end{pmatrix},
\quad
V_7^{(4)}
=
\begin{pmatrix}
0
&
-\dfrac{\delta^{2}}{r}\, e^{i(h^++h^-)t} 
\\
\dfrac{r}{\delta^{2}}\, e^{-i(h^++h^-)t} 
&
0
\end{pmatrix},
\nonumber\\
&
V_8^{(4)}
=
\begin{pmatrix}
0
&
\bar r\delta^{2}\, e^{i(h^++h^-)t} 
\\
-\dfrac{1}{\bar r \delta^{2}}\, e^{-i(h^++h^-)t} 
&
0
\end{pmatrix}.
\end{align}
The task which we turn to next
is then to determine $k_{o} $, $\alpha = \alpha_\re + i\alpha_\im$ and $h$ so that all these jump matrices remain bounded as $t\to\infty$.

\paragraph{The definition of $h$.}
We begin by introducing the upwardly oriented branch cut 
%
$\tilde B = L_7\cup \left(-L_8\right)$,
%
and we define the single-valued function $\gamma$ with branch cuts $B$ and $\tilde B$ by
\begin{equation}\label{gammadef}
\gamma(k)
=
\left[\left(k^2+q_{o} ^2\right)\left(k-\alpha\right)\left(k-\bar \alpha\right)\right]^{\frac 12},
\end{equation}
where we identify $\gamma$  with its  right-sided limit along  $B$ and $\tilde B$,  i.e., we set $\gamma(k) = \gamma^-(k) = -\gamma^+(k)$ for $k\in B\cup \tilde B $, so that  $\gamma(k) \sim k^2$ as $k \to \infty$.  
Then $\gamma$ gives rise to a genus-1 Riemann surface $\Upsigma$ with sheets $\Upsigma_1, \Upsigma_2$ and a basis $\{ \upalpha, \upbeta\}$ of cycles defined as follows: the $\upbeta$-cycle is  a closed, anticlockwise contour around the branch cut $B$ which remains entirely on the first sheet $\Upsigma_1$ of the Riemann surface; the $\upalpha$-cycle consists of an anticlockwise contour that starts on the left of $\tilde B $, then approaches $B$ from the right while on the first sheet $\Upsigma_1$, and finally returns to the starting point via the second sheet $\Upsigma_2$. 
These cycles are depicted in Figure \ref{abcycles}~(left).

\begin{figure}[t!]
\hskip 2cm
\begin{subfigure}{0.45\textwidth}
\begin{center}
\includegraphics[scale=.45]{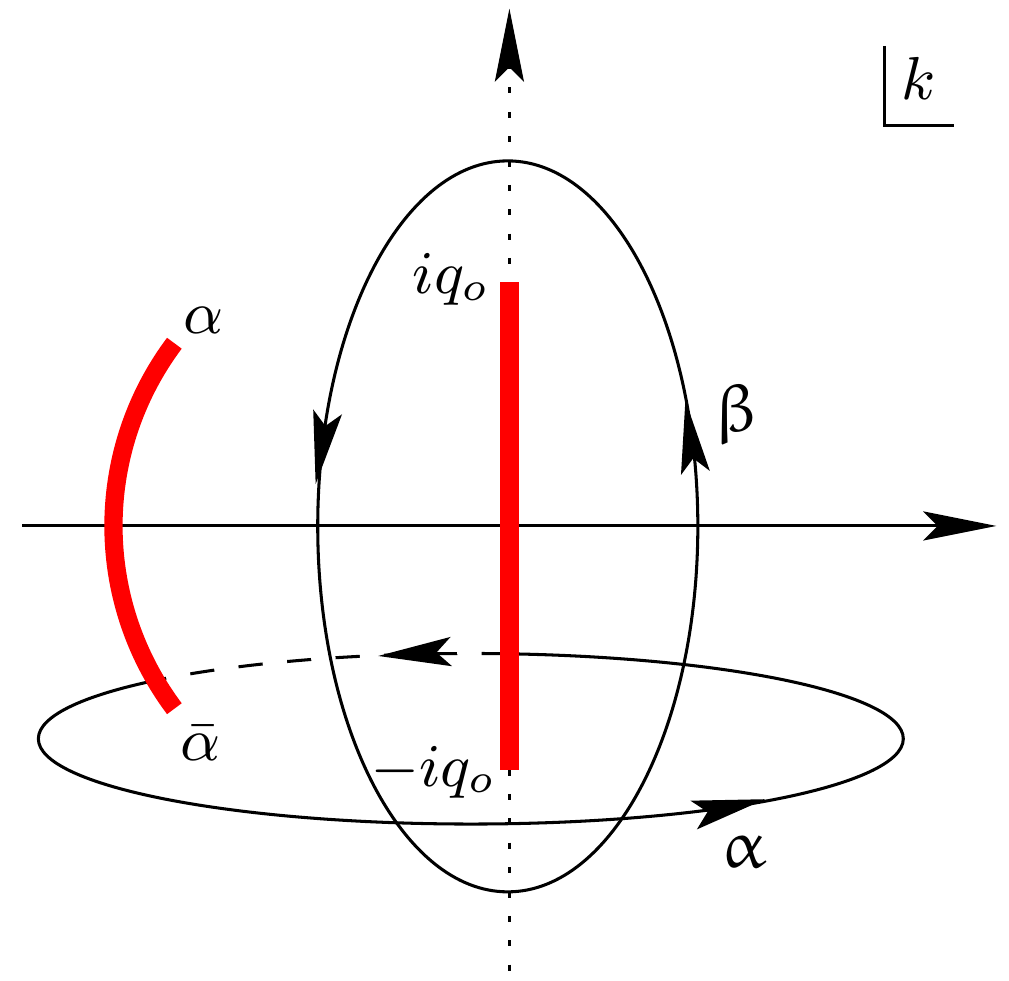}
\end{center}
\end{subfigure}
\hskip -1.2cm
\begin{subfigure}{0.45\textwidth}
\begin{center}
\includegraphics[scale=.23]{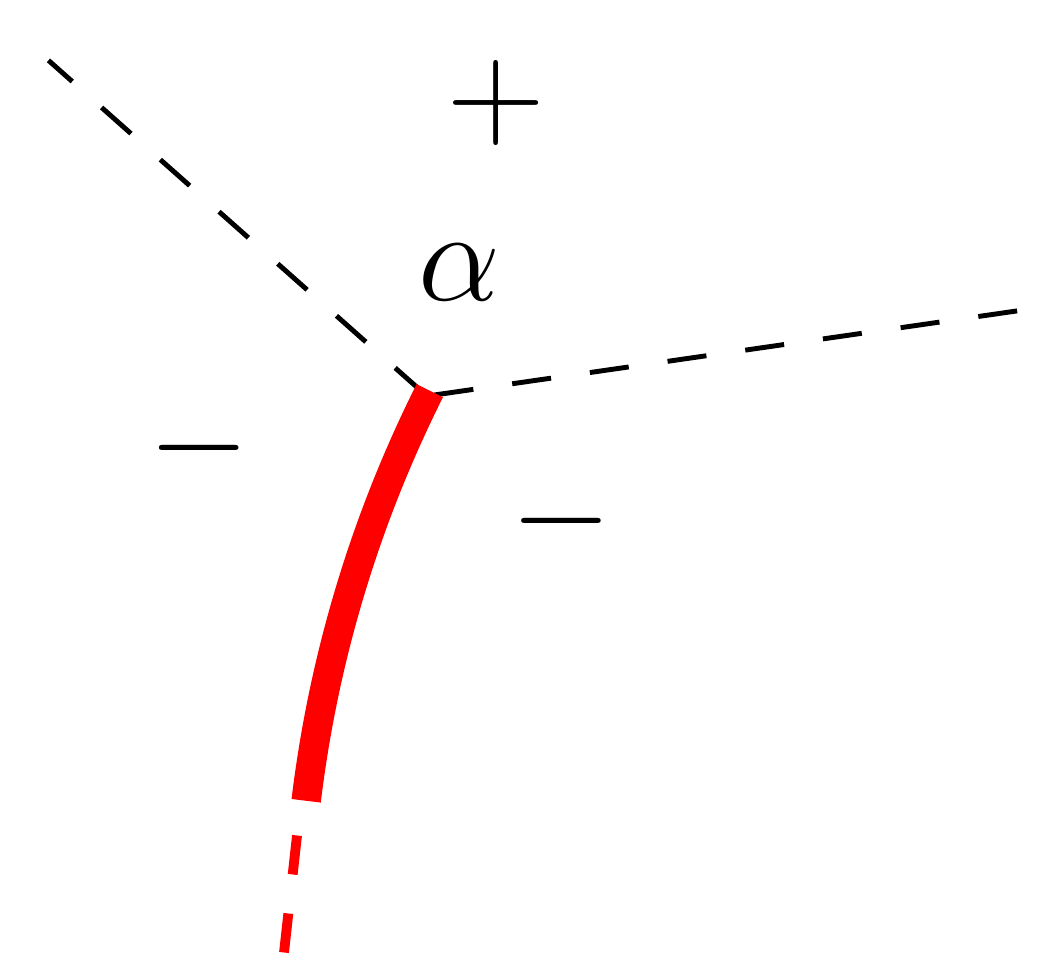}
\end{center}
\end{subfigure}
\caption{
Left: The basis $\{\upalpha, \upbeta\}$ of cycles of the genus-1 Riemann surface $\Upsigma$.
Right: The desired sign structure of $\Re(ih)$ in a neighborhood of $\alpha$
(cf.\ Fig.~\ref{defr3}).}
\label{abcycles}
\end{figure}

Next, similarly to~\cite{bks2011}, we let $h$ be given by the Abelian integral
\begin{equation}\label{habel}
h(k)
=
\frac 12 \left(\int_{iq_{o} }^k+\int_{-iq_{o} }^k \right)dh(z),
\end{equation}
where the Abelian differential $dh$ is defined by
\begin{equation}\label{dhabel}
dh(k)
=
-4\,\frac{\left(k-k_{o} \right)\left(k-\alpha\right)\left(k-\bar \alpha\right)}{\gamma(k)}\, dk.
\end{equation} 
The task is therefore to determine the triplet $\left\{k_{o}, \alpha_\re, \alpha_\im\right\}$ so that:
(a)~the function $h$ defined by equation \eqref{habel} 
satisfies a Riemann-Hilbert problem 
that removes the existing growth from the jumps $V_7^{(4)}$ and $V_8^{(4)}$ of~\eqref{jumph}, 
and 
(b)~the sign structure of $\Re(ih)$ is such that no new growth is introduced in any of the other jumps of~\eqref{jumph}.

\paragraph{The sign structure of $\Re(ih)$.}
%
It is convenient to write
\begin{equation}\label{guant}
\left(k-k_{o} \right)\left(k-\alpha\right)\left(k-\bar \alpha\right)
=
k^3+c_2k^2+c_1 k+c_0,
\end{equation}
where the real constants $c_0, c_1, c_2$ are given by
\begin{equation}\label{cdef}
c_0=-k_{o} \left(\alpha_\re^2+\alpha_\im^2\right),
\quad
c_1=\alpha_\re^2+2k_{o} \alpha_\re+\alpha_\im^2,
\quad
c_2=-\left(k_{o} +2\alpha_\re\right).
\end{equation}
We then have
\begin{equation}\label{habel2}
h(k)
=
-2 \left(\int_{iq_{o} }^k+\int_{-iq_{o} }^k \right)
\frac{z^3+c_2z^2+c_1 z+c_0}{\gamma(z)}\, dz.
\end{equation}
We next use this representation to study the sign structure of $\Re(ih)$ for all $k\in\Complex$.
Note that, since $h$ is completely specified in terms of $k_{o} $ and $\alpha$ by~\eref{habel} and~\eref{dhabel}, 
determining $k_{o} $ and $\alpha$ is equivalent to determining $c_0$, $c_1$ and $c_2$.

We start by studying the sign structure of $\Re(ih)$ as $k\to\infty$.
Recall that the deformations of the Riemann-Hilbert problem were chosen according to the sign structure of $\Re(i\theta)$.
Thus, in the transition from $\theta$ to $h$, this structure should be preserved.
We can ensure that this is the case by requiring 
\begin{equation}\label{thetahas}
\Re(ih)
=
\Re(i\theta)
+
O(1/k),\quad k \to \infty.
\end{equation}
In this regard, note that since
\begin{equation*}
\gamma(k)
=
k^2
\bigg[
1
-\frac{\alpha+\bar \alpha}{2k}
+\frac{4q_{o} ^2-\left(\alpha-\bar \alpha\right)^2}{8k^2}
+O\Big(\frac{1}{k^3}\Big)
\bigg], \quad k \to\infty,
\end{equation*}
it follows that
\begin{equation}
\frac{dh}{dk}
=
-4
\bigg[
k
+(c_2+\alpha_\re)
+
\frac{c_1+c_2 \alpha_\re-\frac 12\left(q_{o} ^2+\alpha_\im^2\right)
+\alpha_\re^2}{k}
+O\Big(\frac{1}{k^2}\Big)
\bigg], \quad k \to \infty.\label{dhas0}
\end{equation}
On the other hand, from the definition \eqref{thet} of $\theta$ we have
\begin{equation}\label{thetas}
\theta(k)
=
-2k^2+\xi k-q_{o} ^2+O\Big(\frac{1}{k}\Big),\quad k \to \infty.
\end{equation}
Hence, to satisfy relation \eqref{thetahas}   we require 
\begin{equation}\label{c1c2cond}
c_1
=
\frac 12\left(q_{o} ^2+\alpha_\im^2\right)
+\frac{\xi}{4}\alpha_\re,
\quad
c_2=-\frac{\xi}{4}-\alpha_\re.
\end{equation}
Then, integrating expansion \eqref{dhas0} we  find
\begin{equation}\label{has}
h(k)
=
-2k^2+\xi k+ H_{o} +O\Big(\frac 1k \Big),\quad k \to \infty,
\end{equation}
where the constant $H_{o} $ can be determined by observing that 
\begin{equation}\label{asobs}
2\left(\int_{iq_{o} }^k+\int_{-iq_{o} }^k\right)\left(z-\frac \xi 4\right)dz
=
2k^2-\xi k+2q_{o} ^2.
\end{equation}
Indeed, combining equations \eqref{habel2} and \eqref{asobs} we can express $h$ in the form
\begin{equation*}
h(k)
=
-2 \left(\int_{iq_{o} }^k+\int_{-iq_{o} }^k \right)
\left[
\frac{z^3+c_2z^2+c_1 z+c_0}{\gamma(z)}
-
\left(z-\frac \xi 4\right)
\right] dz
-\left(2k^2-\xi k +2q_{o} ^2\right).
\end{equation*}
Hence, using also equation \eqref{has} we deduce
\begin{equation}\label{omtil}
H_{o}   
=
-2 \left(\int_{iq_{o} }^\infty+\int_{-iq_{o} }^\infty \right)
\left[
\frac{z^3+c_2z^2+c_1 z+c_0}{\gamma(z)}
-
\left(z-\frac \xi 4\right)
\right] dz
-2q_{o} ^2
\end{equation}
with $c_1$ and $c_2$ given by equation \eqref{c1c2cond} and $c_0$ yet to be determined.
Note that $H_{o} $ is well-defined since the relation
$$
\frac{z^3+c_2z^2+c_1 z+c_0}{\gamma(z)}
-
\Big(z-\frac \xi 4\Big)
=
O\Big(\frac{1}{z^2}\Big), \quad z\to \infty
$$
ensures that the integrals in~\eref{omtil} are convergent.
Moreover,  it is evident from the contours of integration and the fact that $c_0, c_1, c_2 \in\mathbb R$ that $H_{o}   \in\mathbb R$.
Thus the desired behavior~\eqref{thetahas} of $\Re(ih)$ for large $k$ is achieved.
Also, combining~\eref{cdef} and~\eref{c1c2cond} we have 
\begin{equation}
\label{ak0system}
\alpha_\re = -k_{o} +\frac \xi 4,
\quad
\alpha_\im = \sqrt{2k_{o} ^2-\frac \xi 2 k_{o} +q_{o} ^2}.
\end{equation}

It thus remains to determine $k_{o} $.  We do so by analyzing the behavior of $h$ near $\alpha$.
In order for the jumps $V_3^{(4)}$, $V_5^{(4)}$ and $V_6^{(4)}$ to be bounded near $\alpha$,  
$\Re(ih)$ should have the sign structure shown in Figure~\ref{abcycles}~(right). 
Letting $\zeta=k-\alpha$, we have
\bse
\begin{gather}
\label{lamexp}
\left(k^2+q_{o} ^2\right)^{\frac 12}
=
\left(\alpha^2+q_{o} ^2\right)^{\frac 12}
\left[
1+\frac{\alpha \zeta}{\alpha^2+q_{o} ^2}+O\left(\zeta^2\right)
\right],
\quad \zeta\to 0,
\\
\label{gamexp}
\left[\left(k-\alpha\right)\left(k-\bar \alpha\right)\right]^{\frac 12}
=
\left(\alpha-\bar \alpha\right)^{\frac 12}
\left[\zeta^{\frac 12}+\frac{\zeta^{\frac 32}}{2\left(\alpha-\bar \alpha\right)}+O\big(\zeta^{\frac 52}\big)\right],\quad \zeta\to 0.
\end{gather}
\ese
Hence, in a neighborhood of $\alpha$, we have
\begin{equation*}
\frac{dh}{dk}
=
-\frac{4\left(\alpha-\bar \alpha\right)^{\frac 12}\left(\alpha-k_o\right)}{\left(\alpha^2+q_{o} ^2\right)^{\frac 12}}
\left\{
\zeta^{\frac 12}
+
\left[
\frac{1}{\alpha-k_{o} }
+
\frac{1}{2\left(\alpha-\bar \alpha\right)}
-\frac{\alpha}{\alpha^2+q_{o} ^2}
\right]
\zeta^{\frac 32}
+O(\zeta^{\frac 52})
\right\}, \quad \zeta\to 0,
\end{equation*}
from which, integrating, we obtain the expansion
\begin{align*}
h(k)
&=
h(\alpha)
-\frac{4\left(\alpha-\bar \alpha\right)^{\frac 12}\left(\alpha-k_{o} \right)}{\left(\alpha^2+q_{o} ^2\right)^{\frac 12}}
\bigg\{
\frac 23
\left(k-\alpha\right)^{\frac 32}
\nonumber\\
&\quad
+
\frac 25
\left[
\frac{1}{\left(\alpha-k_{o} \right)}
+
\frac{1}{2\left(\alpha-\bar \alpha\right)}
-\frac{\alpha}{\alpha^2+q_{o} ^2}
\right]
\left(k-\alpha\right)^{\frac 52}
+O(\left(k-\alpha\right)^{\frac 72})
\bigg\},\quad k\to \alpha,
\end{align*}
where 
\begin{equation}\label{hadef}
h(\alpha)
=
 -2
 \left(\int_{iq_{o} }^\alpha+\int_{-iq_{o} }^\alpha \right)
\frac{z^3+c_2z^2+c_1z+c_0}{\gamma(z)}\, dz.
\end{equation}
On the other hand, in order to obtain the sign structure in Figure \ref{abcycles}~(right), 
the leading order term of the expansion of $\Re(ih)$ near $\alpha$ should be of $O(k-\alpha)^{\frac 32}$. 
Thus,  we  must have
\begin{equation}\label{imha}
\Im \,h(\alpha) = 0.
\end{equation}
Using contour deformations, we can write $h(\alpha)$ in the form
\begin{equation}\label{hadef2}
h(\alpha)
=
 -2
 \left(\int_{iq_{o} }^\alpha+\int_{-iq_{o} }^{\bar \alpha} \right)
\frac{z^3+c_2z^2+c_1 z +c_0}{ \gamma(z)}\, dz
 -2
\int_{\bar\alpha}^\alpha 
\frac{z^3+c_2z^2+c_1 z +c_0}{ \gamma(z)}\, dz,
\end{equation}
where, as before,  $\gamma(z)$ is taken to be continuous from the right on the branch cut $\tilde B= [\bar\alpha,\alpha]$. 
We then note that the first term on the right-hand side of equation \eqref{hadef2} is real, while the second term is imaginary.
Hence \eqref{imha} is equivalent to the following condition:
\begin{equation}\label{imha3}
\int_{\bar \alpha}^{\alpha} 
\frac{z^3+c_2z^2+c_1 z +c_0}{ \gamma(z)}\, dz
=
0.
\end{equation}
It is convenient to reformulate the above condition as follows.
The integrals of $dh$ from $\bar \alpha$ to $-iq_{o} $ and from $\alpha$ to $iq_{o} $ are both equal to half of the integral of $dh$ along the $\upalpha$-cycle. 
Hence, by analyticity the contour of integration from $\bar \alpha$ to $\alpha$ on the right of  $\tilde B $  can be deformed to the contour from $-iq_{o} $ to $iq_{o} $ on the left of  $B$. 
Thus, recalling the fact that $\gamma^+=-\gamma^-=-\gamma$ across $B$ and also equation \eqref{guant},  condition~\eqref{imha3} takes the equivalent forms
\begin{equation*}
\int_{-iq_{o} }^{iq_{o} } 
\frac{z^3+c_2z^2+c_1 z +c_0}{ \gamma(z)}\, dz
=
\int_{-iq_{o} }^{iq_{o} } 
\sqrt{\frac{\left(z-\alpha_\re\right)^2+\alpha_\im^2}{z^2+q_{o} ^2}}\left(z-k_{o} \right) dz
=0.
\end{equation*}
Finally, recalling~\eref{ak0system}, the above condition yields the following integral equation for $k_{o} $:
\begin{equation}\label{k0k0}
\int_{-iq_{o} }^{iq_{o} } 
\sqrt{\frac{\left(z+k_{o} -\frac \xi 4\right)^2+2k_{o} ^2-\frac \xi 2 k_{o}  +q_{o} ^2}{z^2+q_{o} ^2}}\left(z-k_{o} \right) dz
=0.
\end{equation} 
The solution of~\eref{k0k0} uniquely determines the points $\alpha$ and $k_{o} $, and hence the function $h$, in terms of $\xi$ and $q_{o} $.  

\begin{remark}
The integral equation \eqref{k0k0} is trivially satisfied for:
\begin{enumerate}
\item[(i)]
$\xi = -4\sqrt 2 q_{o} $ (i.e, $x= -4\sqrt2 q_{o} t$).
In this case $k_{o}  = - q_{o} /\sqrt2$, and $\alpha_\re=k_{o} $ and $\alpha_\im=0$. 
Note that for $\xi = -4\sqrt 2 q_{o} $ equation \eqref{kroots} implies that $k_1=k_2 = k_{o} $. 
Thus, at the interface between the plane wave and modulated elliptic wave regions,
the two asymptotic descriptions are consistent. 
\item[(ii)]
$\xi = 0 $ (i.e, either $x=0$ or in the limit $t\to\infty$). 
In this case 
the point $k_{o} $  collapses to  the origin
and the branch points $\alpha$ and $\bar \alpha$ collapse to the branch points $\pm iq_{o} $.
\end{enumerate}
More generally, we can show the following, similarly to \cite{bks2011}:
\end{remark}
\begin{lemma} 
For all $\xi\in(-4\sqrt2q_{o},0)$, the integral equation \eqref{k0k0} has a unique solution $k_{o} =k_{o} (\xi)$. 
\end{lemma}

\begin{proof}
The result essentially follows from the implicit function theorem.  The change of variables
\begin{equation}\label{xy-def}
x=-{\xi}/{(8q_{o} )}, \quad y={k_{o} -\frac \xi 4}/{q_{o} },
\end{equation}
and the parametrization $z=iq_{o}  \zeta$ turn equation \eqref{k0k0} into
\begin{equation}\label{k0k02}
f(x,y)
=
\int_{-1}^{1} 
\sqrt{\frac{\left(i\zeta+y\right)^2+2y\left(y-2x\right)+1}{1-\zeta^2}}\left(i\zeta+2x-y\right) d\zeta
=0.
\end{equation} 
%
Since $\xi\in\left[-4\sqrt 2 q_{o}, 0\right]$ and $k_{o} \in (\frac \xi 8, 0 )$, we consider equation \eqref{k0k02} only for $(x, y)\in \mathcal D= (-\tfrac{\sqrt 2}{2}, \sqrt 2)\times (0, \tfrac{\sqrt{2}}{2})$.
Since $f(0, 0)=0$ and $f_x(0, 0)< 0$,   by the implicit function theorem there exists a neighborhood $X\times Y$ of $(0,0)$ and a unique function $F: X\mapsto Y$ such that 
$
\left\{\left(x, F(x)\right): x\in X\right\} = \left\{ (x, y)\in X\times Y: f(x, y)=0\right\}.
$ 
Hence, there exists a unique solution $y(x)$ of the integral equation \eqref{k0k02}
in a neighborhood of $(0, 0)$. Moreover,  $f_x(x, y)<0$ throughout the domain $\mathcal D$ with the exception of the point $(\frac{\sqrt{2}}{2},\frac{\sqrt{2}}{2})$. Thus, there exists a neighborhood around every point $(x, y)\in \mathcal D$ in which there exists a unique function $F(x)$ such that $y=F(x)$ represents the unique solution of equation \eqref{k0k02}, and the only point where uniqueness is violated is $(\frac{\sqrt{2}}{2},\frac{\sqrt{2}}{2})$. However, this point is itself a solution of equation \eqref{k0k02} (see case (i) above for $\xi = -4\sqrt 2 q_{o} $). Thus, for any $x\in (-\frac{\sqrt 2}{2}, \sqrt 2)$  there exists a unique $y(x)$ such that $f(x, y(x))=0$, and $y( \frac{\sqrt{2}}{2})=\frac{\sqrt{2}}{2}$.
Equivalently, for any $\xi\in \left[-4\sqrt 2 q_{o}, 0\right]$ there exists a unique solution $k_{o} (\xi)$ of the integral equation \eqref{k0k0} with $k_{o} (-4\sqrt 2 q_{o} )=-\frac{\sqrt{2}}{2} q_{o} $.
\end{proof}

\paragraph{Global sign structure of $\Re(ih)$.}
In summary, when the points $\alpha$ and $k_{o} $ are chosen according to equations~\eqref{ak0system}  and~\eqref{k0k0}, the function $h$ defined by equation \eqref{habel} eliminates the growth from the matrices $V_7^{(4)}$ and $V_8^{(4)}$ and the quantity $\Re(ih)$   has  the correct sign structure both for large $k$ and for $k$ near $\alpha$ and $\bar \alpha$.
%
%
Before formulating the Riemann-Hilbert problem for the function $M^{(4)}$, however, one must check that $\Re(ih)$ has the correct sign structure \textit{for all finite} $k\in\Complex$.
To confirm that this is indeed the case, it remains to verify that $\Re(ih)$ has the correct behavior near $k=0$.
Specifically, we want to show that for any $k_{o}, \alpha_\re<0$, there exists a neighborhood of the origin in which $\Re(ih)$ has the required sign structure.  So see this, using the expansions
\begin{gather*}
\left[\left(k-\alpha\right)\left(k-\bar \alpha \right)\right]^{\frac 12}
=
|\alpha|
\Big[1-\frac{\alpha_\re k}{|\alpha|^2}
+
 \frac{\alpha_\im^2}{2|\alpha|^4} k^2
+O(k^3)\Big],
\quad
\lambda 
=
\textrm{sign}(k_\re)
\,
q_{o} 
\Big[
1+\frac{k^2}{2q_{o} ^2}+O(k^4)\Big],
\end{gather*}
as $k\to0$ with $k_\im>0$, we find
\begin{equation*}
\frac{dh}{dk}
=
\frac{4|\alpha|k_{o} }{\textrm{sign}(k_\re)q_{o} }
\Big[
1-\Big(\frac{\alpha_\re}{|\alpha|^2}+\frac{1}{k_{o} }\Big)k
+
\Big(
\frac{\alpha_\im^2}{2|\alpha|^4}-\frac{1}{2q_{o} ^2}
+
\frac{\alpha_\re}{k_{o} |\alpha|^2}
\Big)k^2
+
O(k^3)
\Big],
\end{equation*}
which implies 
\begin{equation}
h(k)
=
\frac{4|\alpha|k_{o} }{\textrm{sign}(k_\re)q_{o} }
\Big[
k-\frac 12 \Big(\frac{\alpha_\re}{|\alpha|^2}+\frac{1}{k_{o} }\Big)k^2
+
O(k^3)
\Big]
+h(0)
\label{h0exp}
\end{equation}
in the same limit.
Since by the definition \eqref{habel2} of $h$ it follows that $h(0)\in\mathbb R$, equation \eqref{h0exp} yields
\begin{equation}
\Re(ih)
=
-\frac{4|\alpha|k_{o} k_\im}{\textrm{sign}(k_\re)q_{o} }
\Big[
1- \Big(\frac{\alpha_\re}{|\alpha|^2}+\frac{1}{k_{o} }\Big)k_\re
\Big]
+
O(k^3)
\end{equation}
as $k\to0$ with $k_\im>0$.
Recalling that $k_{o}, \alpha_\re< 0$ we thereby deduce that
for small $k$ in the first quadrant  we have
$\Re(ih)>0$, while for small $k$  in the second quadrant  
$\Re(ih)<0$ provided that
$
\left(k_{o} \alpha_\re+|\alpha|^2\right)k_\re>k_{o}  |\alpha|^2.
$
The treatment of the lower half-plane is analogous.

\begin{remark}
The analysis  of the sign structure of $\Re(ih)$ as $k\to\infty$, for $k$ near $\alpha$ and $\bar \alpha$, and as $k\to0$ implies that it is always possible to deform the contours of Figure \ref{defr3}  so that they do not go through regions in which $\Re(ih)$ has the ``wrong'' sign (i.e., a sign that causes exponential growth). 
This is because $\Re(ih)$ is a harmonic function away from the branch cuts. 
If there existed a region of ``wrong'' sign separating two regions of ``correct'' sign, then there would be further critical points of $\Re(ih)$ in addition to $\alpha, \bar \alpha, k_{o} $. 
In turn, this would imply additional critical points for $h$ besides $\alpha, \bar \alpha, k_{o} $. This is impossible, however, since the Abelian differential \eqref{habel} has exactly three zeros, namely  $\alpha, \bar \alpha, k_{o} $.
Therefore, $\Re(ih)$ has the appropriate sign structure \textit{for all} $k\in\mathbb C$, as shown in Figure \ref{hsign2}. 
\end{remark}

\begin{figure}[t!]
\begin{center}
\includegraphics[scale=1]{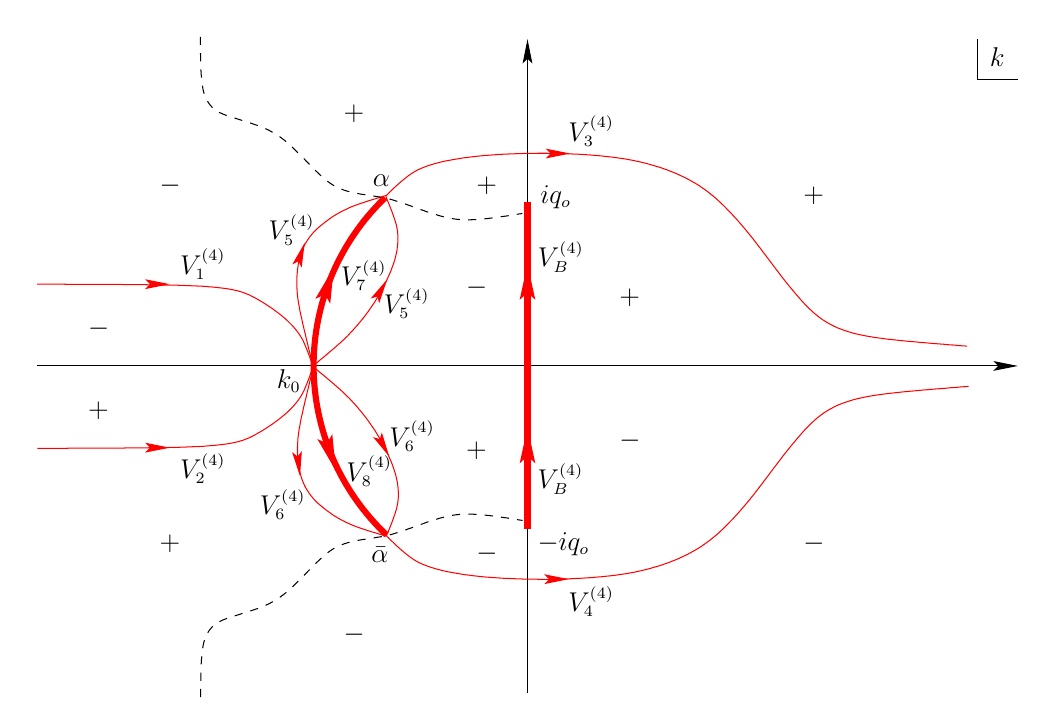}
\caption{The sign structure of $\Re(ih)$ for all $k\in\mathbb C$.}
\label{hsign2}
\end{center}
\end{figure}

\begin{figure}[t!]
\hskip -0.1cm
\begin{subfigure}{0.275\textwidth}
\includegraphics[scale=.42]{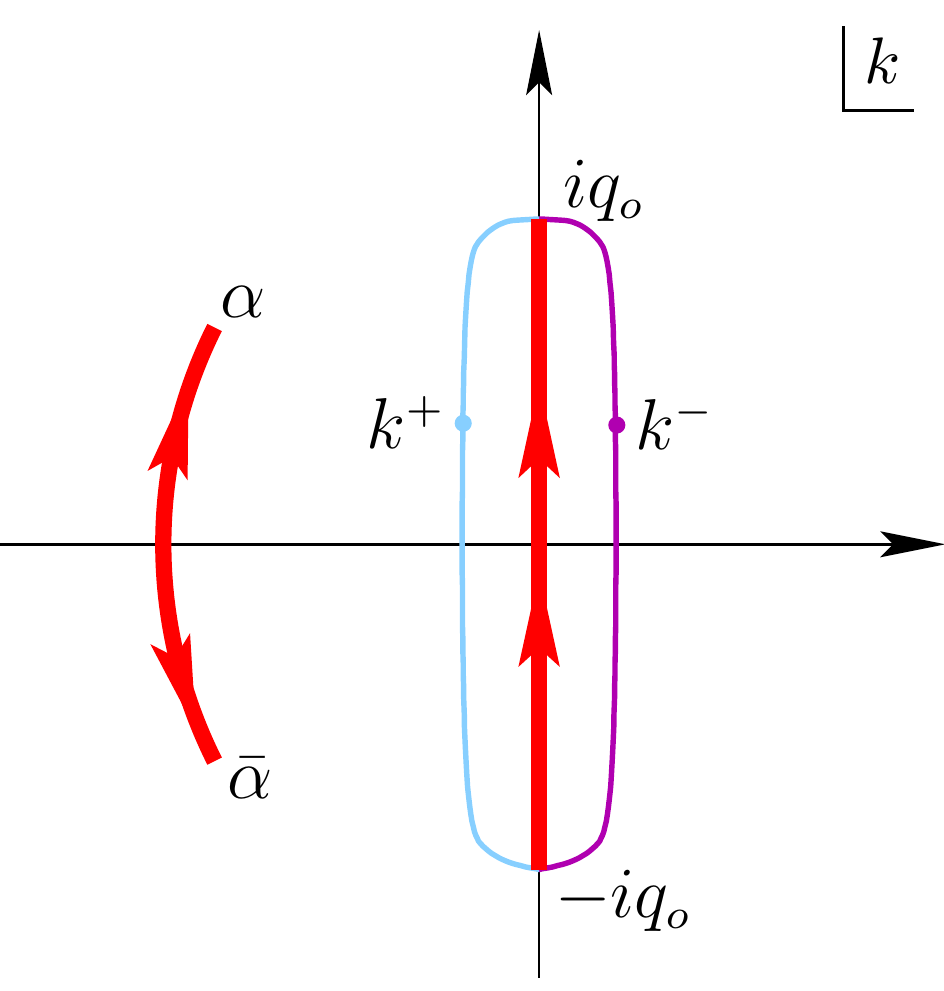}
\end{subfigure}
\hskip 1.4cm
\begin{subfigure}{0.275\textwidth}
\includegraphics[scale=.42]{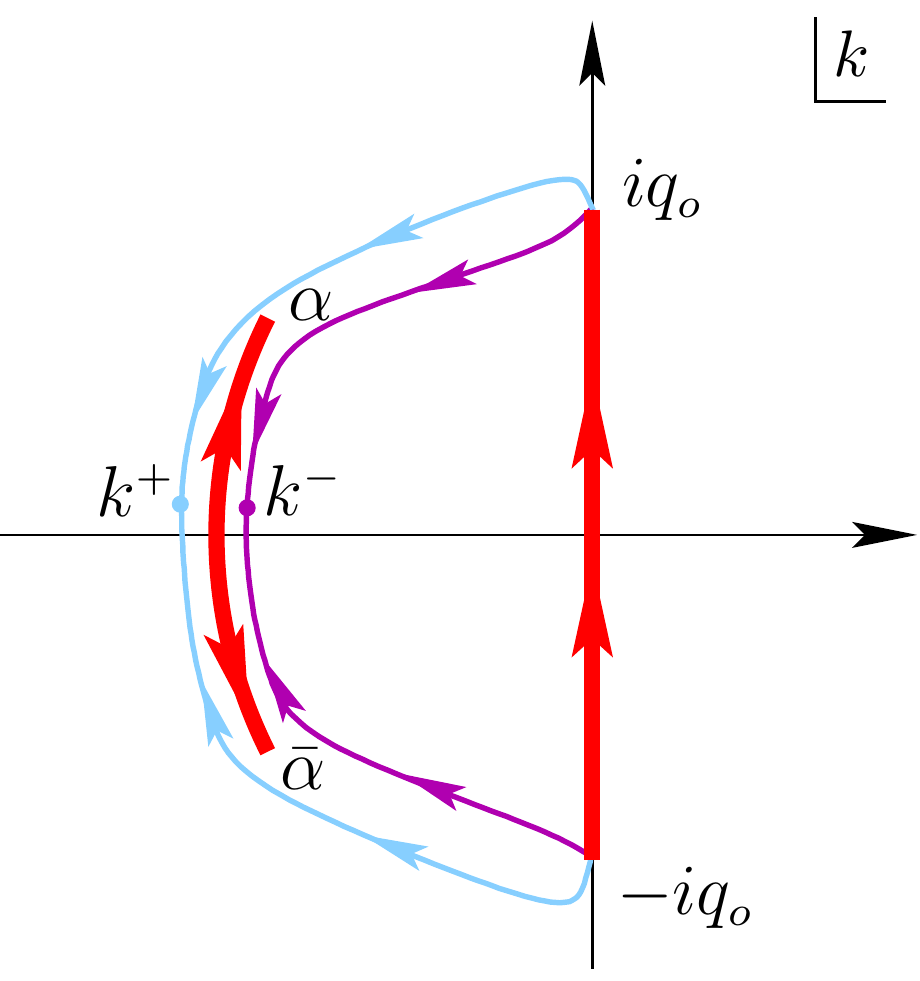}
\end{subfigure} 
\caption{Left: The computation of the jump of $h$ across $B$ is done by recalling that both Abelian integrals in~\eqref{habel} change sign due to the function $\gamma$ involved in the Abelian differential $dh$. Right: The computation of the jump of $h$ across $\tilde B$ is done by first deforming the contours of integration of the Abelian integrals in  \eqref{habel} as shown in the figure and then by noting that $dh$ changes sign across $\tilde B$ due to the function $\gamma$.}
\label{abelfig-comb}
\end{figure}

\paragraph{The jumps of $h$.}
Recall that $h(k)=h(\xi, \alpha, k_{o}, k)$ is analytic in $\mathbb C\setminus (B\cup\tilde B)$. 
We now compute the jumps of $h$ across the branch cuts $B$ and $\tilde B$, which are needed to determine the jump matrices in~\eref{jumph}.
We do so by deforming the contours involved in equation \eqref{habel2} as shown in Figures  \ref{abelfig-comb},
obtaining the jump conditions
\begin{subequations}\label{rhph}
\begin{align}
& h^+(k)+h^-(k)=0, 
\hskip .45cm
k\in B,\label{jhL}
\\
&h^+(k)+h^-(k)=\Omega,
\quad
k\in L_7\cup L_8,
\label{jhLL}
\end{align}
\end{subequations}
where the real constant $\Omega $ is defined by 
\begin{equation} \label{Om-def}
\Omega 
=
-4 \left(\int_{iq_{o} }^{\alpha}+\int_{-iq_{o} }^{\bar \alpha} \right)
\frac{\left(z-k_{o} \right)\left(z-\alpha\right)\left(z-\bar \alpha\right)}{ \gamma(z)}\, dz.
\end{equation}
Moreover, recalling the large-$k$ behavior of $h$ specified by equation \eqref{has}, we infer that $h$ satisfies the normalization condition 
\begin{equation}
h(k) = -2k^2+\xi k+H_{o}   +O(1/k),
\quad k \to \infty,
\end{equation}
where, for  $\alpha$ and $k_{o} $ given by equations \eqref{ak0system} and \eqref{k0k0}, the jump contours $L_7$ and $L_8$ are depicted in Figure \ref{defr3}, 
and the real constant $H_{o} $ is defined by equation \eqref{omtil}.

\paragraph{The Riemann-Hilbert problem for $M^{(4)}$.}
By the definition \eqref{m4r} of $M^{(4)}$ and the Riemann-Hilbert problem  \eqref{rhp3r} for $M^{(3)}$, we infer that $M^{(4)}$ is analytic in $\mathbb C \setminus  (\cup_{j=1}^8 L_j \cup  B)$ and satisfies the jump conditions
\begin{subequations}\label{rhp4r}
\begin{align}
&M^{(4)+}(k) = M^{(4)-}(k)V_B^{(4)},\quad k\in B,
\\
&M^{(4)+}(k) = M^{(4)-}(k)V_j^{(4)},\quad k\in L_j,\ j=1, \ldots, 8,
\end{align}
and the normalization condition
\begin{equation}\label{m4-norm-mew}
M^{(4)}(k) =
\left[I+O(1/k)\right] e^{-iG_\infty t  \sigma_3},\quad k \to \infty,
\end{equation}
\end{subequations}
where the jump contours $L_j$ are  shown in Figure \ref{defr3}, 
the jump matrices $V_j^{(4)}$  \eqref{jumph} simplify thanks to the jumps \eqref{rhph} satisfied by $h$ to  
\begin{align}\label{jump4r}
&
V_B^{(4)}
=
\begin{pmatrix}
0
&
\dfrac{q_-\delta^{2}}{iq_{o} }
\\
\dfrac{\bar q_-}{iq_{o} \delta^{2}}
&
0
\end{pmatrix},
\quad
V_1^{(4)}
=
\begin{pmatrix}
1
&
\dfrac{\bar r\,\delta^{2} }{1+r\bar r}\,e^{2i h t}
\\
0
&
1
\end{pmatrix},
\quad
V_2^{(4)}
=
\begin{pmatrix}
1
&
0
\\
\dfrac{r }{\left(1+r\bar r\right) \delta^{2}}\,e^{-2i h t}
&
1
\end{pmatrix},
\nonumber\\
&V_3^{(4)}
=
\begin{pmatrix}
1
&
0
\\
\dfrac{r}{\delta^2}\, e^{-2i h t}
&
1
\end{pmatrix},
\ 
V_4^{(4)}
=
\begin{pmatrix}
1
&
\bar r \delta^{2}\, e^{2i h t}
\\
0
&
1
\end{pmatrix},
\ 
V_5^{(4)}
=
\begin{pmatrix}
1
&
\dfrac{\delta^{2}}{r}\, e^{2i h t} 
\\
0
&
1
\end{pmatrix}, 
V_6^{(4)}
=
\begin{pmatrix}
1
&
0
\\
\dfrac{1}{ \bar r\delta^{2}}\, e^{-2i h t} 
&
1
\end{pmatrix},
\nonumber\\
&
V_7^{(4)}
=
\begin{pmatrix}
0
&
-\dfrac{\delta^{2}}{r}\, e^{i\Omega   t} 
\\
\dfrac{r}{\delta^{2}}\, e^{-i\Omega   t} 
&
0
\end{pmatrix},
\quad
V_8^{(4)}
=
\begin{pmatrix}
0
&
\bar r\delta^{2}\, e^{i\Omega   t} 
\\
-\dfrac{1}{\bar r \delta^{2}}\, e^{-i\Omega   t} 
&
0
\end{pmatrix},
\end{align}
and, using the expansions \eqref{thetas} and \eqref{has}, the real constant $G_\infty$ involved in the normalization condition \eqref{m4-norm-mew} is equal to
\begin{equation}\label{Ginf}
G_\infty 
=
H_{o}   +q_{o} ^2
\end{equation}
with the real constant $H_{o} $ defined by equation \eqref{omtil}.
%
%
%
%
%
%
\paragraph{Fifth deformation.}
Our final task is to eliminate the dependence on $k$ from the three jumps $V_B^{(4)}$, $V_7^{(4)}$ and $V_8^{(4)}$ across the branch cuts $B$  and $\tilde B$.
This can be achieved with the help of an additional $g$-function, this time introduced via a $t$\textit{-independent} exponential, 
exactly as in the fourth deformation \eqref{m4def} for the plane wave region (as opposed to transformation \eqref{m4r}). 
In particular, we let
\begin{equation}\label{m5r}
M^{(5)}(x, t, k)
=
M^{(4)}(x, t, k) e^{ig(k)\sigma_3},
\end{equation}
where the function $g$ is analytic in $\mathbb C\setminus (B\cup \tilde B)$ with jumps
\begin{subequations}\label{pav1}
\begin{align}
&g^+(k)+g^-(k)
=
-i\ln\left(\delta^2\right), \hskip 1.9cm
k\in B,
\\
& g^+(k)+g^-(k)
=
-i\ln\left(\delta^2/r\right) + \omega, \hskip .75cm k\in L_7,
\\
&
g^+(k)+g^-(k)
=
-i\ln\left( \delta^2  \bar r\right) + \omega, \hskip .94cm
k\in L_8,
\end{align}
\end{subequations}
with the function $\delta$  defined by equation \eqref{deldefr} and 
the real constant $\omega$  given by
\begin{equation}\label{omr}
\omega
=
 i 
 \left(
\displaystyle
\int_{B} \frac{\ln\left[\delta^2(\nu) \right]}{  \gamma(\nu )}\, d\nu
+
\int_{L_7} \frac{
\ln\left[\frac{\delta^2(\nu) }{r(\nu )}\right]}{  \gamma(\nu )}\, d\nu 
-
\int_{L_8} \frac{\ln\left[ \delta^2(\nu)  \, \bar r(\nu )\right]}{  \gamma(\nu )}\, d\nu 
\right)
\left/
\int_{\tilde B}  \frac{d\nu}{  \gamma(\nu )}\right..
\end{equation}
Proceeding as in the plane wave region, we arrive at a scalar, additive Riemann-Hilbert problem analogous to problem \eqref{rhpg}, which is solved via Plemelj's formulae to yield 
\begin{align}
g(k)
&=
\frac{\gamma(k)}{2\pi}
\bigg[
\int_{B} \frac{\ln\left[\delta^2(\nu) \right]}{ \gamma(\nu )(\nu -k)}\, d\nu 
+
\int_{L_7} \frac{
\ln\left[\frac{\delta^2(\nu) }{r(\nu )}\right]+i\omega}{ \gamma(\nu )(\nu -k)}\, d\nu 
-
\int_{L_8} \frac{\ln\left[ \delta^2(\nu)  \, \bar r(\nu )\right]+i\omega}{\gamma(\nu )(\nu -k)}\, d\nu 
\bigg].\label{gsolr}
\end{align}
The definition \eqref{omr} of $\omega$ ensures that $g(k)=O(1)$ as $k \to \infty$. In particular, we have
\begin{equation}\label{ginfr}
g(k)
=
g_\infty 
+
O(1/k),
\quad 
k \to \infty,
\end{equation}
where the real constant $g_\infty $ is given by
\begin{equation}
g_\infty 
=
\frac{1}{2\pi}
\bigg[
-
\int_{B} \frac{\ln\left[\delta^2(\nu) \right]}{\gamma(\nu )}\,\nu d\nu
-
\int_{L_7} \frac{
\ln\left[\frac{\delta^2(\nu) }{r(\nu )}\right]}{ \gamma(\nu )}\,\nu  d\nu 
+
\int_{L_8} \frac{\ln\left[ \delta^2(\nu)  \, \bar r(\nu )\right]}{\gamma(\nu )}\,\nu  d\nu 
-i \omega\,\alpha_\re\int_{\tilde B} \frac{d\nu}{\gamma(\nu )}  \bigg]. 
\label{ginfr2}
\end{equation}
Overall, the function $M^{(5)}$ defined by equation \eqref{m5r} satisfies the
following Riemann-Hilbert problem:
\begin{rhp}[Final problem in the modulated elliptic wave region]\label{rhp5r}
Determine a sectionally analytic matrix-valued function $M^{(5)}(k)=M^{(5)}(x, t, k)$ in $\mathbb C\setminus ( \cup_{j=1}^6 L_j \cup B \cup \tilde B)$ satisfying the jump conditions
\begin{subequations}\label{rhp55r}
\begin{align}
&M^{(5)+}(k) = M^{(5)-}(k)V_B, \hskip .6cm k\in B,
\\
&M^{(5)+}(k) = M^{(5)-}(k)V_{\tilde B}, \hskip .6cm k\in \tilde B,
\\
&M^{(5)+}(k) = M^{(5)-}(k)V_j^{(5)}, \quad k\in L_j,\ j=1, \dots, 6, 
\end{align}
and the normalization condition
\begin{equation}
M^{(5)}(k) =
 \left[I+O\left(1/k \right)\right]e^{i\left(g_\infty -G_\infty t  \right)\sigma_3},
\quad k \to \infty,
\end{equation}
\end{subequations}
where the jump $V_B$ across the branch cut $B$ is defined by equation \eqref{v18}, the jump $V_{\tilde B}$ across the branch cut $\tilde B$ is defined by 
\begin{equation}
V_{\tilde B}
=
\begin{pmatrix}
0
&
-e^{i(\Omega t - \omega)} 
\\
e^{-i(\Omega t - \omega)} 
&
0
\end{pmatrix},
\end{equation}
the jumps $V_j^{(5)}$ across the  contours  $L_j$ shown in Figure \ref{defr3} are given by
\begin{align}\label{jump5r}
&
V_1^{(5)}
=
\begin{pmatrix}
1
&
\dfrac{\bar r\,\delta^{2} }{1+r\bar r}\,e^{2i(ht-g)}
\\
0
&
1
\end{pmatrix}, 
\quad
V_2^{(5)}
=
\begin{pmatrix}
1
&
0
\\
\dfrac{r }{\left(1+r\bar r\right) \delta^{2}}\,e^{-2i(ht-g)}
&
1
\end{pmatrix}, 
\nonumber\\
&V_3^{(5)}
=
\begin{pmatrix}
1
&
0
\\
\dfrac{r}{\delta^{2}}\, e^{-2i(ht-g)}
&
1
\end{pmatrix}, 
\quad
V_4^{(5)}
=
\begin{pmatrix}
1
&
\bar r \delta^{2}\, e^{2i(ht-g)}
\\
0
&
1
\end{pmatrix}, 
\nonumber
\\
&V_5^{(5)}
=
\begin{pmatrix}
1
&
\dfrac{\delta^{2}}{r}\, e^{2i(ht-g)} 
\\
0
&
1
\end{pmatrix},
\quad 
V_6^{(5)}
=
\begin{pmatrix}
1
&
0
\\
\dfrac{1}{ \bar r\delta^{2}}\, e^{-2i(ht-g)} 
&
1
\end{pmatrix},
\end{align}
the function $\delta$ is defined by equation \eqref{deldefr}, the function $h$ is defined by the Abelian integral \eqref{habel}, and satisfies the jumps \eqref{rhph}, the function $g$ is defined by equation \eqref{gsolr}, the real constants $\Omega$,  $G_\infty$, $\omega$ and $g_\infty$ are given by equations \eqref{Om-def},  \eqref{Ginf}, \eqref{omr} and \eqref{ginfr2} respectively.
\end{rhp}

%
%
%
%
%
\paragraph{Decomposition of $M^{(5)}$.}
As for the final Riemann-Hilbert problem in the plane wave region, a suitable decomposition of $M^{(5)}$ is now required.
Namely, we decompose $M^{(5)}$ in a way that separates the jumps expected to yield the leading order contribution from the jumps expected to contribute only in the error. 
In particular, denoting by
$D_{k_{o} }^{\varepsilon}$, $D_{\alpha}^{\varepsilon}$ and $D_{\bar \alpha}^{\varepsilon}$ the disks of radius $\varepsilon$ centred at $k_{o} $, $\alpha$ and $\bar \alpha$ respectively, with $\varepsilon$ sufficiently small so that these disks do not intersect with each other or with $B$, we write
\begin{equation}\label{m5ear}
M^{(5)}
=
M^\err M^\asymp
\quad
\text{with}
\quad
M^\asymp
=
\begin{cases}
M^B, \quad k\in \mathbb C\setminus (D_{k_{o} }^{\varepsilon}\cup D_{\alpha}^{\varepsilon}\cup D_{\bar \alpha}^{\varepsilon}),
\\
M^D, \quad k\in D_{k_{o} }^{\varepsilon}\cup D_{\alpha}^{\varepsilon}\cup D_{\bar \alpha}^{\varepsilon},
\end{cases}
\end{equation}
where:
\begin{itemize}[leftmargin=*]
\item
the function $M^B$ is analytic in $\mathbb C\setminus(B \cup \tilde B)$  and satisfies the jump conditions 
\begin{subequations}\label{rhpBr}
\begin{align}
&M^{B+}(k) = M^{B-}(k) \,
V_B, \quad k\in B,
\label{rhpBbr}\\
&M^{B+}(k) = M^{B-}(k) \,
V_{\tilde B}, \quad k\in \tilde B
\label{rhpBcr}
\end{align}
(see Figure \ref{mbfig}), and the normalization condition
\begin{equation}
M^B(k) =  
 \left[I+O\left(1/k \right)\right]e^{i\left(g_\infty -G_\infty t  \right)\sigma_3},\quad k \to \infty,\label{mbasymrhpr}
\end{equation}
\end{subequations}

\begin{figure}[t!]
\begin{center}
\includegraphics[scale=.425]{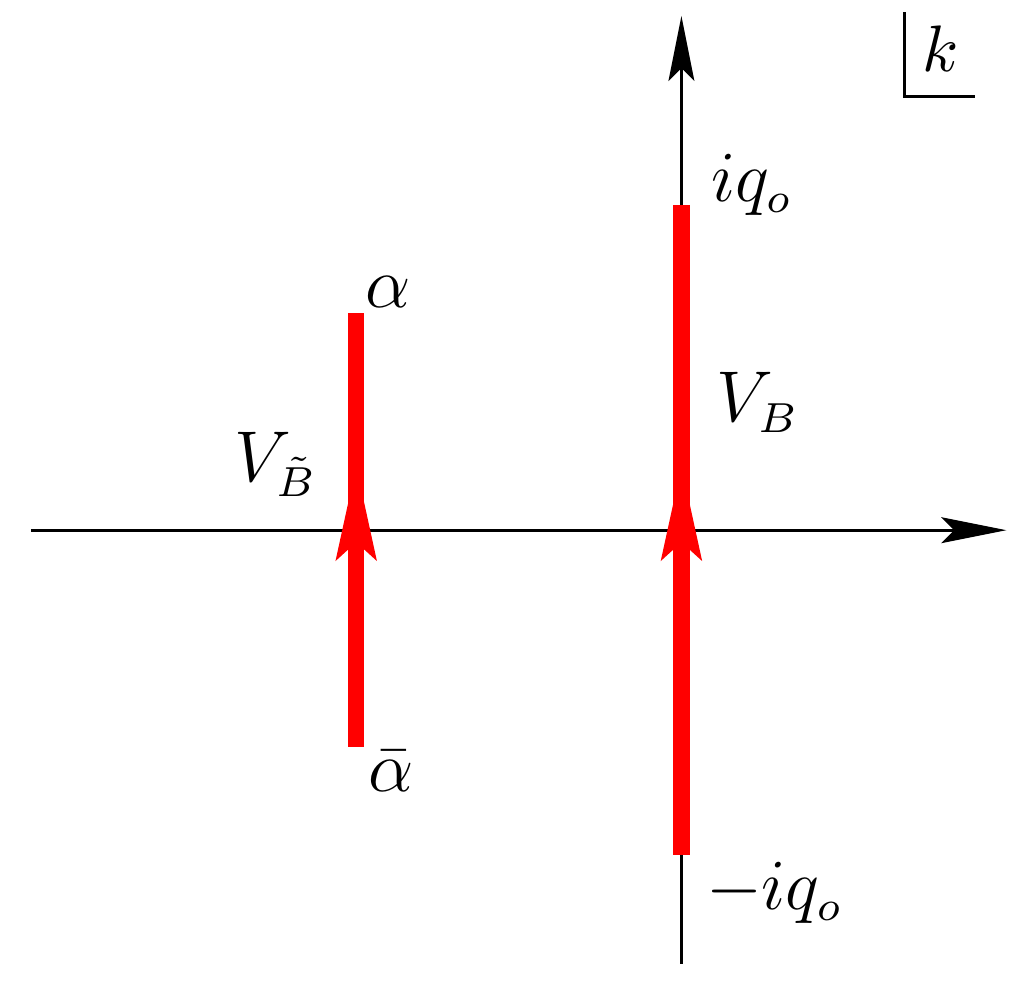}
\caption{The jumps of $M^B$ in the modulated elliptic wave region. 
Note that, since the jump $V_{\tilde B}$ is constant, the contour $\tilde B$ can be deformed 
to the straight line segment from $\bar\alpha$ to $\alpha$.}
\label{mbfig}
\end{center}
\bigskip
\medskip
\begin{center}
\includegraphics[width=0.235\textwidth]{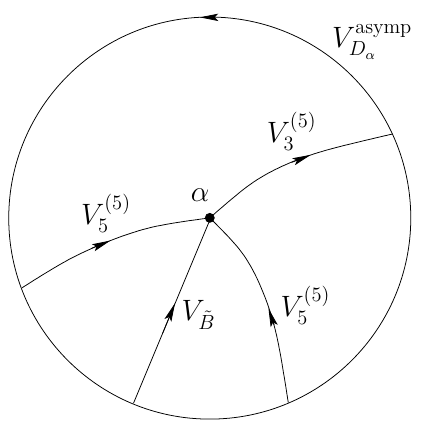}
\hskip 4mm
\includegraphics[width=0.235\textwidth]{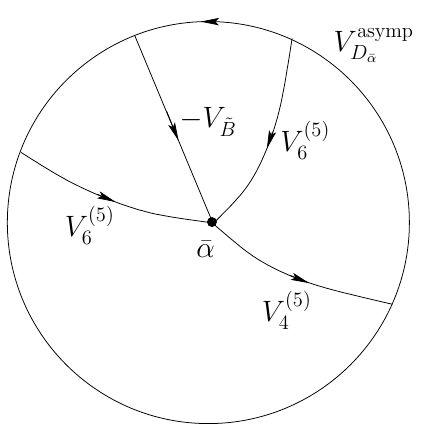}
\hskip 5mm
\includegraphics[width=0.235\textwidth]{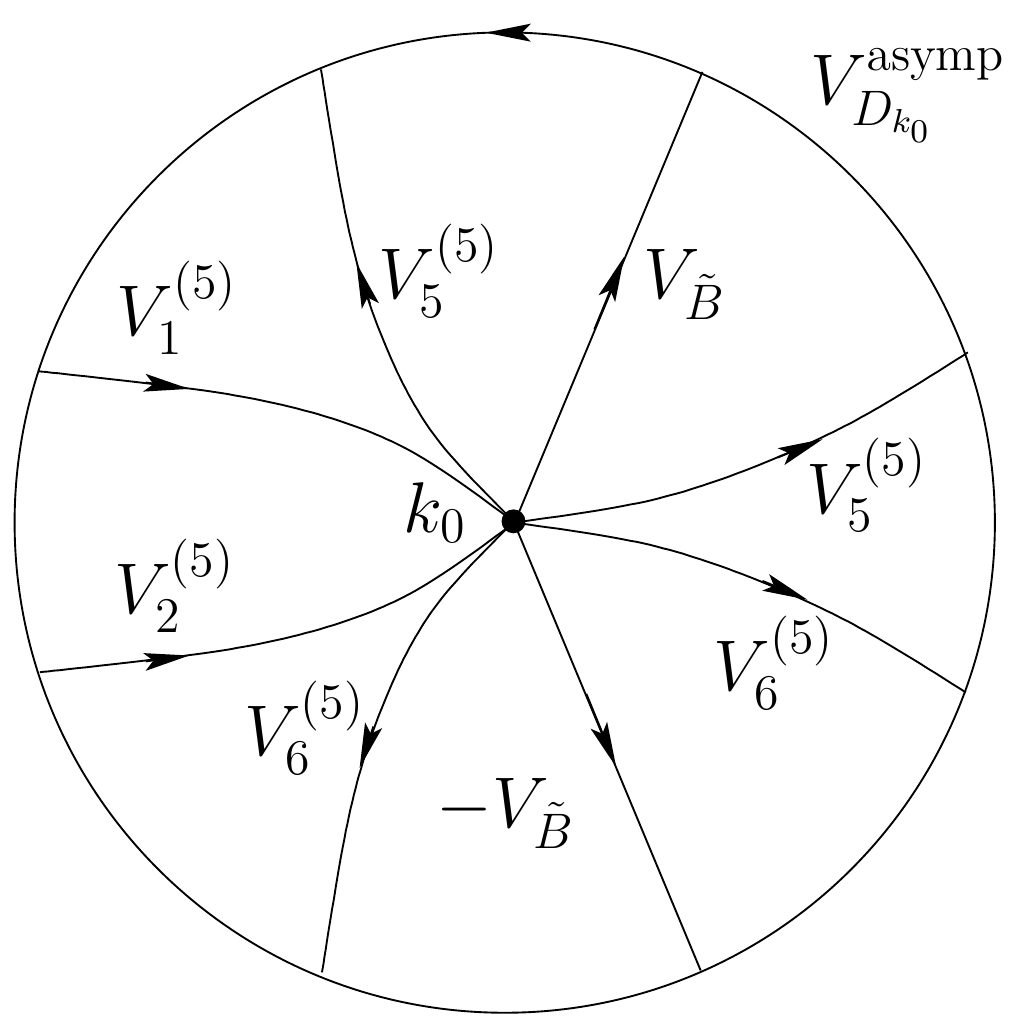}
\caption{The jumps of  $M^D$  inside the disks $\overline{D_\alpha^\varepsilon}$, $\overline{D_{\bar \alpha}^\varepsilon}$ and $\overline{D_{k_{o} }^\varepsilon}$ in the modulated elliptic wave region. }
\label{daefig}
\end{center}
\bigskip
\begin{center}
\includegraphics[scale=1.15]{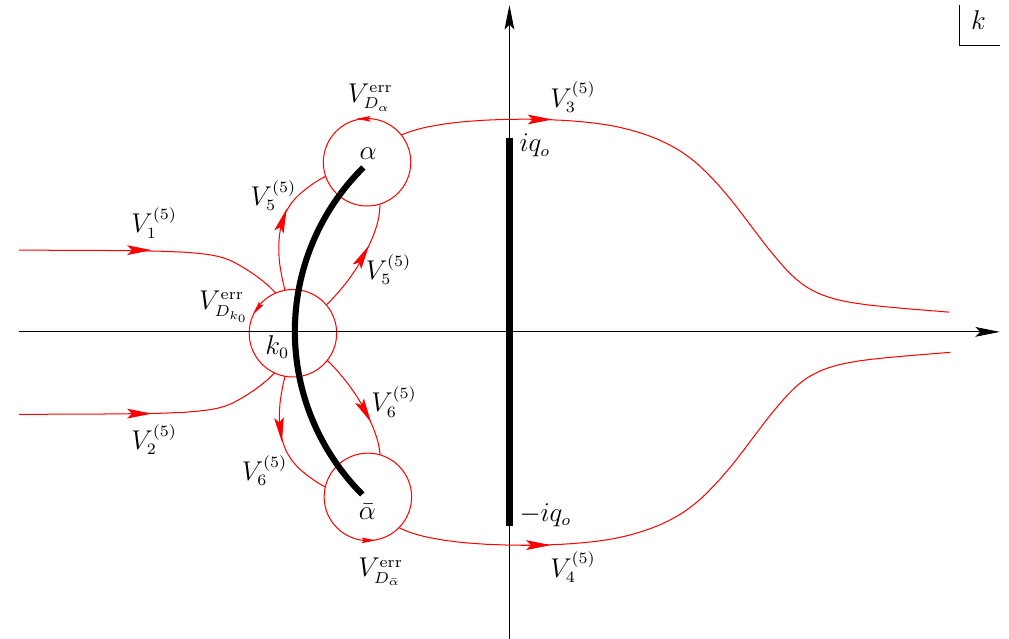}
\caption{The jumps of $M^\err$ in the modulated elliptic wave region.}
\label{mefig}
\end{center}
\end{figure}

\item
the function $M^D$ is analytic in $D_{k_{o} }^{\varepsilon}\cup D_{\alpha}^{\varepsilon}\cup D_{\bar \alpha}^{\varepsilon}\setminus  \cup_{j=1}^{8}  L_j$
with jumps (see Figure \ref{daefig})
\begin{equation}\label{rhpPr} 
M^{D+}(k)
=
M^{D-}(k)
V_j^{(5)},
\quad
k\in \hat L_j= L_j\cap \left(D_{k_{o} }^{\varepsilon}\cup D_{\alpha}^{\varepsilon}\cup D_{\bar \alpha}^{\varepsilon}\right),
\ 
j=1, \ldots, 8,
\end{equation}

\item
the function $M^\err$ is analytinc in $\mathbb C
\setminus (\cup_{j=1}^{6} \check L_j \cup \p D_{k_{o} }^{\varepsilon}\cup \p D_{\alpha}^{\varepsilon}\cup \p D_{\bar \alpha}^{\varepsilon})$ and satisfies the jump conditions 
\begin{subequations}\label{rhpEr}
\begin{equation}
M^{\err+}(k) = M^{\err-}(k) \,
V^\err,\quad  k\in \cup_{j=1}^{6} \check L_j \cup \p D_{k_{o} }^{\varepsilon}\cup \p D_{\alpha}^{\varepsilon}\cup \p D_{\bar \alpha}^{\varepsilon}
\label{rhpEbr}
\end{equation}
(see Figure \ref{mefig}), and the normalization condition
\begin{equation}
M^\err(k) = I +O(1/k),\quad k \to \infty,
\end{equation}
\end{subequations}
with $\check L_j = L_j\setminus ( \overline{D_{k_{o} }^{\varepsilon}}\cup \overline{D_{\alpha}^{\varepsilon}}\cup \overline{D_{\bar \alpha}^{\varepsilon}})$ and
\begin{equation}\label{VEdefr}
V^\err
=
\begin{cases}
M^B V_j^{(5)} (M^B)^{-1}, &k\in \check L_j,
\\
M^{\asymp-}(V_D^\asymp)^{-1}(M^{\asymp-})^{-1}, &k\in \p D_{k_{o} }^{\varepsilon}\cup \p D_{\alpha}^{\varepsilon}\cup \p D_{\bar \alpha}^{\varepsilon},
\end{cases}
\end{equation}
where $V_D^\asymp$ is the  jump of $M^\asymp$ across the circles $\p D_{k_{o} }^\varepsilon$, $\p D_{\alpha}^\varepsilon$ and $\p D_{\bar \alpha}^\varepsilon$, which is yet unknown.
\end{itemize}

\paragraph{Solution of the leading order asymptotic problem.}
We now determine $M^B$, which yields the leading order contribution to the asymptotics of the solution of the NLS equation as $t\to\infty$.
Our approach requires the introduction of appropriate theta functions, similarly to \cite{bv2007}.
In particular, recall that the function $\gamma$, defined by equation \eqref{gammadef}, gives rise to the Riemann surface $\Upsigma$ with sheets $\Upsigma_1$, $\Upsigma_2$ and cycles $\{ \upalpha, \upbeta\}$ depicted in Figure \ref{abcycles}.
Then, consider the Abelian differential
\begin{equation}\label{abelmb}
dw = \frac{C}{\gamma(k)}\, dk, 
\quad
C = 1 \left/ \oint_{\upbeta} \frac{dk}{ \gamma(k)} \right.,
\end{equation}
which is normalized so that
\begin{equation}\label{norm-a}
\oint_{\upbeta} dw = 1
\end{equation}
and has Riemann period  $\tau$ defined by
\begin{equation}\label{tau-a}
\tau = \oint_{\upalpha}  dw.
\end{equation}
Note that $C\in i\Real$. 
In addition,  it can be shown that $\tau\in i\Real^+$ (see, for example, \cite{fk1980}). 
Also, note that the normalization \eqref{norm-a}
and the definition \eqref{tau-a} of $\tau$ imply 
\begin{equation}\label{pap1}
\int_{-iq_{o} }^{iq_{o} } dw = \frac 12,\quad
\int_{\bar \alpha}^{-iq_{o} } dw = \frac \tau 2,\quad \ k\in\Upsigma_1.
\end{equation}
Next, we introduce the genus-1 theta function $\Theta$ as
\begin{equation}\label{thetjdef}
\Theta(k)
=
\sum_{\ell\in\mathbb Z} e^{2i\pi \ell  k+i\pi \ell^2\tau}
=
\theta_3(\pi k, e^{i\pi \tau}),
\end{equation}
where $\theta_3$ denotes the third Jacobi theta function defined by 
%
\begin{equation}\label{t3-def0}
\theta_3(z, \varrho)
=
\sum_{\ell\in\mathbb Z} e^{2i \ell  z} \varrho^{\ell^2}.
\end{equation}
The function $\theta_3$ is analytic for $(z, \varrho)\in \mathbb C\times  \left\{\mathbb C: |\varrho|<1\right\}$.
Hence, since $i\tau<0$, the function 
$\Theta$ is analytic for all $k\in\mathbb C$.
Moreover, $\Theta$ is even  and possesses the translation properties
\begin{equation}\label{thetaid}
\Theta(k+n) = \Theta(k),
\quad
\Theta(k+n \tau) 
=
e^{-2i\pi n k-i\pi n^2 \tau} \Theta(k),
\quad n\in\mathbb Z.
\end{equation}
We then define the vector-valued function $\@M$ as
\begin{equation}\label{mcaldef}
\@M (k, c) 
=
\left(
\frac{
 \Theta\big(-\frac{\Omega   t}{2\pi}+\frac{\omega}{2\pi}+\frac{i\ln \, (\frac{\bar q_-}{iq_{o} })}{\pi}+v(k)+c \big)}
 {  \sqrt{\frac{iq_{o} }{\bar q_-}}\ \Theta\left(v(k)+c \right)},
\frac{\Theta\big(-\frac{\Omega   t}{2\pi}+\frac{\omega}{2\pi}+\frac{i\ln\, (\frac{\bar q_-}{iq_{o} })}{\pi}-v(k)+c \big)}
{ \sqrt{\tfrac{\bar q_-}{iq_{o} }}\ \Theta\left(-v(k)+c \right)}
\right)
\end{equation}
where the constants $\Omega$ and $\omega$ are defined by equations~\eqref{Om-def} and \eqref{omr} as before, the constant $c$ is for now arbitrary  and will be determined in due course, and $v$ is the Abelian map
\begin{equation}\label{vdefr0}
v(k) 
=
\int_{iq_{o} }^k dw.
\end{equation} 
Note that  since $v$ is analytic for all  $k\notin B \cup \tilde B $ and $\Theta$ is analytic for all $k\in\mathbb C$,  the only sources of non-analyticity of the function $\@M$ in the $k$-plane are the branch cuts $B$ and $\tilde B $ and  the possible zeros of the functions $\Theta(v(k)\pm c)$.

The function $\@M$ has been introduced in order to satisfy the jumps and the normalization condition of the Riemann-Hilbert problem \eqref{rhpBr} for $M^B$.
To compute the jumps of $\@M$ note that, using the $\upalpha$- and $\upbeta$-cycles of the Riemann surface~$\Upsigma$, we have
\begin{subequations}\label{BBtilv}
\begin{align}
&v^+(k) 
+v^-(k)=n, \hskip 1.35cm n\in\mathbb Z, \ k\in B,
\\
&v^+(k)+v^-(k)= -\tau +n, \quad n\in\mathbb Z, \  k\in \tilde B.
 \end{align}
 \end{subequations}
Also, using the translation properties \eqref{thetaid} of $\Theta$, we find
\begin{subequations}\label{mcalbbtil}
\begin{align}
&\@M^+(k)
=
\@M^-(k)
\begin{pmatrix}
0 &-q_-/iq_{o} 
\\
\bar q_-/iq_{o}  &0
\end{pmatrix}, \hskip 7mm k\in B,\label{mcalb}
\\
&\@M^+(k)
=
\@M^-(k)
\begin{pmatrix}
0 &  e^{i(\Omega   t-\omega)}  
\\
e^{-i(\Omega   t-\omega)}  & 0
\end{pmatrix}, \quad k\in \tilde B. \label{mcalbtil}
\end{align}
\end{subequations}

The jumps  \eqref{mcalbbtil} differ from those of Riemann-Hilbert problem \eqref{rhpBr} only by a negative sign in the 12-entry. 
Thus, instead of $\@M$ we consider the matrix-valued function $N$ defined by
\begin{equation}\label{ncaldef}
N(k, c)
=
\frac 12
\begin{pmatrix}
\left[p(k)+p^{-1}(k)\right]\@M^{(1)}(k, c)
& 
i\left[p(k)-p^{-1}(k)\right]\@M^{(2)}(k, c)
\\
-i\left[p(k)-p^{-1}(k)\right]\@M^{(1)}(k, -c)
&
\left[p(k)+p^{-1}(k)\right]\@M^{(2)}(k, -c)
\end{pmatrix},
\end{equation}
where  
\begin{equation}\label{pdef}
p(k)
=
\left[
\frac
{\left(k-iq_{o} \right)\left(k-\alpha\right)}
{\left(k+iq_{o} \right)\left(k-\bar \alpha \right)}
\right]^{\frac 14}.
\end{equation}
The function $p$ is analytic away from the branch cuts $B$ and $\tilde B$, is nonzero away from $k=\alpha$ and $iq_{o}$ and has fourth root singularities 
at $k=\bar\alpha$ and $-iq_{o}$..
Moreover, $p$ has the same jump discontinuity across both $B$ and $\tilde B$, namely
\begin{equation}\label{pbbtil}
p^+(k)=ip^-(k), \quad k\in B\cup \tilde B.
\end{equation}
In addition, $p$ admits the large-$k$ expansion
\begin{equation}\label{pasym}
p(k) = 1-\frac{i(q_{o} +\alpha_\im)}{4k}+O\Big(\frac{1}{k^2}\Big),\quad k \to \infty.
\end{equation}

The jumps  \eqref{mcalbbtil} and \eqref{pbbtil} of $\@M$ and $p$ imply that $N$ has the following jumps across $B$ and $\tilde B$:
\begin{subequations}\label{ncalb}
\begin{align}
&N^+(k)
=
N^-(k) V_B,
\quad
k\in B,
\\
&N^+(k)
=
N^-(k) V_{\tilde B},
\quad
k\in \tilde B.
\end{align}
\end{subequations}
Therefore, the function $N$  defined by equation \eqref{ncaldef} has the same jumps as $M^B$ across $B$ and~$\tilde B$. 

Regarding the analyticity of $N$ away from $B$ and $\tilde B$, we observe that the only possible singularities of $N$ other than the usual branch points $\pm i q_{o}$, $\alpha$ and $\bar \alpha$ could arise from the functions 
$\Theta(v(k)\pm c)$ in the denominator of equation~\eref{mcaldef}. 
In this regard, we note that, from the definition \eqref{thetjdef} of $\Theta$  in terms of the Jacobi $\theta_3$ function, 
the zeros of $\Theta$ are simple and located at
$
\frac 12\left(1+\tau\right) +\mathbb Z+\tau\mathbb Z.
$
Moreover,  the function $p - p^{-1}$ has a unique finite simple zero on the cut complex $k$-plane given by
\begin{equation}\label{kstar}
k_*=\frac{q_{o} \alpha_\re}{q_{o} +\alpha_\im}.
\end{equation}
Then, setting the constant $c$, which was introduced as an arbitrary constant in the definition~\eqref{mcaldef} of $\@M$,
equal to
\begin{equation}
c=
v(k_*)+\frac 12\left(1+\tau\right)
\label{c-choice}
\end{equation}
implies that (i) the function $\Theta(v(k)-c)$ has a unique zero on the first sheet $\Upsigma_1$ of the Riemann surface~$\Upsigma$, located at the pre-image of $k_*$, and (ii) the function $\Theta(v(k)+c)$ is nonzero on $\Upsigma_1$
(see, for example, \cite{fk1980}, p. 290-291).
Hence, the choice \eqref{c-choice} of the constant $c$ ensures that the unique singularity of $\@M^{(2)}(k, c)$ and $\@M^{(1)}(k, -c)$ on $\Upsigma_1$ is compensated by the unique zero of $p-p^{-1}$, while $\@M^{(1)}(k, c)$ and $\@M^{(2)}(k, -c)$ are non-singular on $\Upsigma_1$. 
As a consequence, $N$ is analytic as a function in $\Upsigma_1$ away from the branch points.
Therefore,  $N$ is analytic for all $k\in \mathbb C\setminus (B\cup \tilde B)$.

Finally, we discuss the large-$k$ behavior of $N$. 
Combining the definition \eqref{ncaldef} of $N$ and the expansion \eqref{pasym} for $p$ we find
 \begin{equation}\label{ncalasym}
 \lim_{k \to\infty}
N(k, c)
=
 N(\infty, c)
 =
\begin{pmatrix}
\@M^{(1)}(\infty, c) &0
  \\
  0 &  \@M^{(2)}(\infty, -c)
 \end{pmatrix},
 \end{equation} 
where $\@M^{(j)}(\infty,c) = \lim_{k\to\infty}\@M^{(j)}(k,c)$ for $j=1,2$, 
and 
\begin{equation}\label{mcal-lim}
\@M(\infty, c)
=
\left(
\frac{
 \Theta\big(-\frac{\Omega   t}{2\pi}+\frac{\omega}{2\pi}+\frac{i\ln \, (\frac{\bar q_-}{iq_{o} })}{\pi}+v_\infty +c \big)}
 {  \sqrt{\frac{iq_{o} }{\bar q_-}}\ \Theta\left(v_\infty +c \right)},
\frac{\Theta\big(-\frac{\Omega   t}{2\pi}+\frac{\omega}{2\pi}+\frac{i\ln\, (\frac{\bar q_-}{iq_{o} })}{\pi}-v_\infty +c \big)}
{ \sqrt{\tfrac{\bar q_-}{iq_{o} }}\ \Theta\left(-v_\infty +c \right)}
\right)
\end{equation}
with
\begin{equation}\label{vdefr}
v_\infty  
=
\int_{iq_{o} }^\infty dw.
\end{equation}

In summary, both of the matrix-valued functions $M^B$ and $N$ are analytic away from $B\cup \tilde B $ and satisfy identical jumps across $B$ and $\tilde B$. Thus, we deduce that the unique solution of the Riemann-Hilbert problem \eqref{rhpBr} is
\begin{equation}
\label{mbsolr}
M^B(k)
=
e^{i\left(g_\infty -G_\infty t  \right)\sigma_3}
 N^{-1}(\infty, c)N(k, c),
\end{equation}
where the exponential term and the constant matrix $N(\infty, c)$ in equation~\eqref{mbsolr} are necessary so as to normalization of $M^B$ at infinity.

\paragraph{The asymptotic limit $t\to \infty$.}
Starting from formula \eqref{qsol} for the solution $q$ of the  focusing NLS equation \eqref{e:NLS} in terms of $M^{(0)}$ and applying the five successive deformations that lead   to $M^{(5)}$, we find
\begin{equation}\label{qsolm5r}
q(x, t)
=
-2i \big(M_1^{(5)}(x, t)\big)_{12}e^{i\left(g_\infty -G_\infty t  \right)},
\end{equation}
where $M_1^{(5)}$ is the $O(1/k)$ coefficient in the large-$k$ expansion of $M^{(5)}$, i.e.,
\begin{equation}
M^{(5)}(x, t, k) = e^{i\left(g_\infty -G_\infty t  \right)\sigma_3} + \frac{M_1^{(5)}(x, t)}{k}+O \Big(\frac{1}{k^2}\Big),\quad k \to \infty.
\end{equation}
Recalling that for large $k$ the decomposition \eqref{m5ear} of $M^{(5)}$ involves $M^B$ and $M^\err$, we have
\begin{equation}\label{qsolmear}
q(x, t)
=
-2i \left(M_1^{B}(x, t)e^{i\left(g_\infty -G_\infty t  \right)}+M_1^{\err}(x, t)\right)_{12}.
\end{equation}
Combining equation \eqref{mbsolr} 
and the expansion \eqref{pasym} for $p$, we obtain
\begin{equation}\label{qsol-mbr}
\left(M_1^B(x, t)\right)_{12}
=
\frac 12
\left(q_{o} +\alpha_\im\right) 
\frac{\@M^{(2)}(\infty, c)}{\@M^{(1)}(\infty, c)}
\,
e^{i\left(g_\infty -G_\infty t  \right)}.
\end{equation}
Moreover, as in the plane wave region, the term  $M_1^{\err}$ admits the estimate
\begin{equation}\label{m1e-est1r}
\left| M_1^{\err}(x, t)\right|
=
O\big(t^{-\frac 12}\big),
\quad t\to \infty,
\end{equation}
which can be established by constructing appropriate parametrices near the points $\alpha, \bar \alpha, k_{o} $ and then employing the techniques presented in the Appendix. 
Since the construction of these parametrices is similar to the relevant construction presented in \cite{bv2007}, it is omitted here for brevity.

Overall, inserting equation \eqref{qsol-mbr} and estimate \eqref{m1e-est1r} in equation \eqref{qsolmear}, we  conclude that the long-time asymptotic behavior of the solution  of the   focusing NLS equation \eqref{e:NLS} in the modulated elliptic wave region is given by
\begin{multline}
\label{qsol-genus1}
q(x, t)
=
\frac{q_{o} \left(q_{o} +\alpha_\im\right) }{\bar q_-}
\frac{\Theta\big(\!-\frac{\Omega   t}{2\pi}+\frac{\omega}{2\pi}+\frac{i\ln\, (\frac{\bar q_-}{iq_{o} })}{2\pi}-v_\infty +c \big)
\,
\Theta\left(v_\infty +c \right)}
{\Theta\big(\!-\frac{\Omega   t}{2\pi}+\frac{\omega}{2\pi}+\frac{i\ln\, (\frac{\bar q_-}{iq_{o} })}{2\pi}+v_\infty +c \big) 
\,
\Theta\left(-v_\infty +c \right)}
\,
e^{2i\left(g_\infty -G_\infty t  \right)}
\\
+
O\big(t^{-\frac 12}\big),
\quad t\to \infty,
\end{multline}
where the constants $\alpha_\im$, $\Omega$,
 $G_\infty $,    $\omega$, $g_\infty $, $c$ and $v_\infty $   are given by equations \eqref{ak0system},
 \eqref{Om-def}, \eqref{Ginf}, \eqref{omr},  \eqref{ginfr2},  \eqref{c-choice} and \eqref{vdefr} respectively.

The asymptotic solution \eqref{qsol-genus1} can actually be expressed in a simpler form. 
Specifically, similarly to~\cite{jm2013}, we consider
\begin{equation}
f(k)
=
1-\frac{\left(k-\alpha\right)\left(k-iq_{o} \right)}{\gamma(k)}
 = p(k)(p^{-1}(k) - p(k))
\end{equation}
as a function on the Riemann surface $\Upsigma$  such that $\gamma(k)\sim k^2$ as $k\to \infty_1$, where $\infty_1$ denotes the point at infinity on~$\Upsigma_1$. 
The function $f$ has singularities at 
$\bar \alpha$ and $-iq_{o} $, and zeros at $\infty_1$ and at
the finite point $k_*$, which was introduced earlier by equation \eqref{kstar} as   the unique zero    of the function $p-p^{-1}$.
Therefore, $f$ is a meromorphic function on $\Upsigma$ with divisor $(f)$ equal to
\begin{equation}\label{div-def}
(f) = k_*+\infty_1-(-iq_{o} )-\bar \alpha.
\end{equation}
Since the divisor $(f)$ of a meromorphic function is principal, by Abel's theorem
(e.g., see~\cite{bbeim1994}, Theorem 2.14) it follows that 
$v((f))=0$. 
Or, equivalently,
 \begin{equation}
v(k_*) = v(-iq_{o} )+v(\bar \alpha)-v(\infty_1).
 \end{equation}
Recalling the definition of the cycles (see Figure~\ref{abcycles}) and using equations \eqref{pap1}, we find that $v(-iq_{o})=\frac 12+\mathbb Z$ and 
$v(\bar \alpha)=\frac 12-\frac \tau 2+\mathbb Z$
 on $\Upsigma_1$.
Hence, $v(k_*) = -\frac \tau2 -v_\infty$ on $\Upsigma_1$ and then
equation \eqref{c-choice} implies
 \begin{equation}
 c  = \frac 12 -v_\infty +\mathbb Z+\tau \mathbb Z.
 \end{equation}
 Therefore, the leading order asymptotic solution in equation~\eqref{qsol-genus1} simplifies to
\begin{equation}\label{qsol-genus1-simple}
q_\mathrm{asymp}(x, t)
=
\frac{q_{o} \left(q_{o} +\alpha_\im\right) }{\bar q_-}
\
\frac{\Theta\left(\frac 12 \right)
\Theta\left(\frac{1}{2\pi}
\left[\Omega   t-\omega-i\ln\Big(\frac{\bar q_-}{iq_{o} }\Big)\right]+2v_\infty - \frac 12 \right)}
{
\Theta\left(2v_\infty -\frac 12\right)
\Theta\left(\frac{1}{2\pi}
\left[\Omega   t-\omega-i\ln\Big(\frac{\bar q_-}{iq_{o} }\Big)\right] - \frac 12 \right)}
\,
e^{2i\left(g_\infty -G_\infty t  \right)},
\end{equation}
where the real constants $\alpha_\im$, $G_\infty $, 
$\omega$ and $g_\infty$ are given by equations \eqref{ak0system}, \eqref{Ginf}, \eqref{omr} and \eqref{ginfr2} respectively, 
the complex constant $v_\infty$ is defined by equation \eqref{vdefr}, 
and the real constant $\Omega$ is defined in equation~\eqref{Om-def}.
Finally, note that the constant $\Omega$ can be calculated explicitly in terms of elliptic functions. 
Indeed, using standard results from the theory of Abelian differentials of the second kind, we obtain the expression
\begin{equation}\label{Om-fin}
\Omega
=
\frac{\pi \left|\alpha+iq_{o} \right|}{K(m)}
\left(\xi-2\alpha_\re\right),
\end{equation}
where $K(m)$ 
is the complete elliptic integral of the first kind, defined by
\begin{equation}\label{km-def}
K( m)
=
\int_0^{\frac \pi 2} \frac{dz}{\sqrt{1-m^2\sin^2 z}},
\end{equation}
with the elliptic modulus $m$ equal to
\begin{equation}\label{ell-mod}
m =  \sqrt{1-\left|\frac{\alpha-i q_{o} }{\alpha+iq_{o} }\right|^2}
=
\frac{2\sqrt{q_{o} \alpha_\im}}{\left|\alpha+i q_{o} \right|}\,.
\end{equation}
The proof of Theorem \ref{mew-t} is thus complete.


\section{Representation via elliptic functions: proof of theorem \ref{mew-ell-t}}  
\label{ell-sec}
  
In this section  we prove Theorem \ref{mew-ell-t}, i.e., we express the modulus of the asymptotic solution 
\eqref{qsol-genus1-simple} in the modulated elliptic wave region in terms of  elliptic functions.

Recall that the function $\Theta$ appearing in equation~\eqref{qsol-genus1-simple} was defined in formula~\eqref{thetjdef} in terms of 
the third Jacobi theta function $\theta_3$, 
with the period $\tau$ defined by equation \eqref{tau-a} and the nome $\varrho$  of $\theta_3$ set to 
\begin{equation}\label{nome-def0}
\varrho
=
e^{i \pi \tau}.
\end{equation}
Furthermore,  
using  the  $\upalpha$- and $\upbeta$-cycles depicted in Figure \ref{abcycles}, we express $\tau$ in the form
\begin{equation}\label{tau-m}
\tau
=
{iK\big(\sqrt{1-m^2}\big)}/{K\left(  m\right)},
\end{equation}
where $K(m)$ and $m$ defined by equations~\eqref{km-def} and~\eqref{ell-mod} as before.
Hence, equation \eqref{nome-def0} becomes 
\begin{equation}\label{nome-def}
\varrho
=
e^{-\pi{K\left(\sqrt{1-m^2}\right)}/{K\left(m\right)}}.
\end{equation}
It will also be useful to introduce the three other Jacobi theta functions 
\begin{subequations}
\label{thetas-def}
\begin{gather}
\theta_1(z, \varrho)
=
\sum_{\ell \in\mathbb Z} (-1)^{\ell-\frac 12}
e^{(2\ell+1)i z} \varrho^{(\ell+\frac 12)^2},
\label{t1-def}
\\
\theta_2(z, \varrho)
=
\sum_{\ell \in\mathbb Z} e^{(2\ell+1)i z} \varrho^{(\ell+\frac 12)^2},
\qquad
\theta_4(z, \varrho)
=
\sum_{\ell \in\mathbb Z} (-1)^n e^{2i\ell z} \varrho^{\ell^2}.
\label{t4-def}
\end{gather}
\end{subequations}
\label{t2-def}
With $\varrho$ given by equation~\eqref{nome-def}, 
the above theta functions are associated with the Jacobi elliptic function $\textrm{sn}$ and the elliptic modulus $m$ via  the relations
\begin{equation}\label{jef-def}
\textrm{sn}(z, m) = \frac{\theta_3(0)}{\theta_2(0)}\, \frac{\theta_1(z\theta_3^{-2}(0))}{\theta_4(z\theta_3^{-2}(0))},
\quad
m 
 =
 \frac{\theta_2^2(0)}{\theta_3^2(0)}.
\end{equation}
Hereafter, the nome $\varrho$ will be suppressed from the arguments of the theta functions for brevity.
We are now ready to express the asymptotic solution in terms of elliptic functions.

We begin with the constant~$v_\infty$, defined in equation~\eqref{vdefr}.
Note that by the definition of the $\upbeta$-cycle we have
\begin{equation}
2v_\infty -\frac 12
=
\int_{-iq_{o} }^\infty dw
+
\int_{iq_{o} }^\infty dw
+\mathbb Z,
\end{equation}
hence $2v_\infty -\frac 12$  is imaginary modulo $\mathbb Z$. 
Moreover, we introduce the real quantities
\begin{equation}\label{fipsi-def} 
\phi = \tfrac{1}{2}\left[\Omega   t-\omega-i\ln\left(\tfrac{\bar q_-}{iq_{o} }\right)\right], 
\quad 
\psi = -i\pi \left(2v_\infty -\tfrac 12 \right),
\end{equation}
with $\Omega$ given by equation~\eqref{Om-fin}.
Then, we note that the identity 
$\theta_3(k, \varrho)
=
\theta_4\left(k+\frac{\pi}{2}, \varrho\right)$,
together with the evenness of $\theta_3$, implies 
$
\theta_3(\frac \pi2) 
=
 \theta_3(-\frac \pi2) 
 =
 \theta_4(0)
 $
 and
 $
 \theta_3(\varphi - \frac \pi2)
 =
 \theta_4(\varphi).
 $
Using these relations as well as the relation \eqref{thetjdef} between $\Theta$ and  $\theta_3$, we can write the leading order asymptotic solution \eqref{qsol-genus1-simple} in the form
\begin{equation}\label{qred1}
q_\mathrm{asymp}(x, t)
=
\frac{q_{o} \left(q_{o} +\alpha_\im\right) }{\bar q_-}
\
\frac{\theta_4 (0)
\theta_3 (\phi+i\psi )}
{
\theta_3(i\psi)
\theta_4(\phi)}
\,
e^{2i\left(g_\infty -G_\infty t  \right)}.
\end{equation}
Moreover, recalling that the constants $g_\infty , G_\infty $ are real, and employing the addition formula
\begin{equation}
\theta_3 (\phi+i\psi )\theta_3 (\phi-i\psi )\theta_4^2(0)
=
\theta_4^2(\phi)\theta_3^2(i\psi)-\theta_1^2(\phi)\theta_2^2(i\psi),
\end{equation}
we have
\begin{equation}
\left| q_\mathrm{asymp}(x, t) \right|^2
=
\left(q_{o} +\alpha_\im\right)^2 
\left[
1
-
\frac{\theta_1^2(\phi)\theta_2^2(i\psi)}
{\theta_4^2 (\phi)\theta_3^2(i\psi )}
\right].
\label{prin}
\end{equation}
Thus, according to definition \eqref{jef-def} we have
\begin{equation}
\left| q_\mathrm{asymp}(x, t) \right|^2
=
\left(q_{o} +\alpha_\im\right)^2 
\left[
1
-
m\,
 \frac{ \theta_2^2(i\psi)}
{ \theta_3^2(i\psi )}
 \,
 \textrm{sn}^2(\theta_3^{2}(0)\phi, m) 
\right].
\label{ded0}
\end{equation}
It now remains to compute the ratio of the theta functions in equation~\eqref{ded0} as well as the constant $\theta_3^2(0)$.
For the latter, we simply recall the standard formulae~\cite{byrd1971}
\begin{equation}\label{zeros-id3}
\theta_2^2(0)
=
\frac{2m K(m)}{\pi},
\quad
\theta_3^2(0)
=
\frac{2K(m)}{\pi},
\quad
\theta_4^2(0)
=
\frac{2\sqrt{1-m^2} K(m)}{\pi},
\end{equation}
where $\theta_2^2(0)$ and $\theta_4^2(0)$ were also included for later use.
The calculation of the ratio ${ \theta_2^2(i\psi)}/{ \theta_3^2(i\psi )}$, on the other hand, is more involved.
We first note that, using the identities
\begin{subequations}\label{theta-ids}
\begin{align}
\theta_1(k, \varrho)
&=
-i e^{ik+\frac{i\pi\tau}{4}} \theta_4\left(k+\frac{\pi \tau}{2}, \varrho\right),
\\
\theta_2(k, \varrho)
&=
 \theta_1\left(k+\frac{\pi}{2}, \varrho\right),
\\
\theta_3(k, \varrho)
&=
\theta_4\left(k+\frac{\pi}{2}, \varrho\right),
\end{align}
\end{subequations}
we have
\begin{equation}
\frac{\theta_2^2(i\psi)}{\theta_3^2(i\psi )}
=
\frac{\theta_2^2(2y-\frac \pi 2)}{\theta_3^2(2y-\frac \pi 2)}
=
\frac{\theta_1^2(2y)}{\theta_4^2(2y)},
\quad y = \pi v_\infty .
\end{equation}
Subsequently, using the duplication formulae
\begin{gather}
\label{dup}
\theta_1(2y)\theta_2(0)\theta_3(0)\theta_4(0)
=
2\theta_1(y)\theta_2(y)\theta_3(y)\theta_4(y),
\quad
\theta_4(2y)
\theta_4^3(0)
=
\theta_3^4(y)-\theta_2^4(y),
\end{gather}
we find
\begin{equation}
\label{ratios0}
\frac{\theta_1^2(2y)}{\theta_4^2(2y)}
=
\frac{4\theta_4^4(0)}{\theta_2^2(0)\theta_3^2(0)}
\frac{\theta_1^2(y)}{\theta_2^2(y)}\frac{\theta_3^2(y)}{\theta_2^2(y)}
\frac{\theta_4^2(y)}{\theta_2^2(y)}
\left(\left[\frac{\theta_3(y)}{\theta_2(y)}\right]^4-1\right)^{-2}.
\end{equation}
The three ratios
\begin{equation}\label{ratios}
\frac{\theta_1^2(y)}{\theta_2^2(y)},
\quad
\frac{\theta_3^2(y)}{\theta_2^2(y)},
\quad
\frac{\theta_4^2(y)}{\theta_2^2(y)}
\end{equation}
appearing in equation~\eref{ratios0}
can be computed in a similar way. 
Here we present a detailed derivation of the formula for the third ratio 
and we provide the relevant formulae for the first two ratios for brevity.

The idea is to express the desired ratio in terms of a meromorphic function, similarly to \cite{j2014}.
Specifically, to compute the last ration in equation~\eref{ratios}, we consider the function
\begin{equation}\label{f42-def}
f_{42}(k)
=
e^{i\pi \int_{iq_{o} }^k dw-\frac{i\pi \tau}{4}}
\frac{\Theta\left(\int_{ \alpha}^k dw -\frac 12 - \frac \tau 2\right)}
{\Theta\left(\int_{-iq_{o} }^k dw -\frac 12 - \frac \tau 2\right)}
\end{equation}
 as a function on the Riemann surface $\Upsigma$. 
Noting that on the first sheet $\Upsigma_1$ we have 
\begin{equation}\label{comps2}
\int_{\alpha}^k dw
=
\frac \tau 2+ \int_{iq_{o} }^k dw + \mathbb Z,
\qquad
\int_{-iq_{o} }^k dw
=
 \frac 12
+
\int_{iq_{o} }^k dw+ \mathbb Z,
\end{equation}
and recalling also the periodicity properties of $\Theta$, we find
\begin{equation}\label{f42b}
f_{42}(k)
= 
e^{i\pi \int_{iq_{o} }^k dw-\frac{i\pi \tau}{4}}
\,
\frac{\theta_3\left(\pi \int_{iq_{o} }^k dw - \frac \pi 2 \right)}
{\theta_3\left(\pi \int_{iq_{o} }^k dw  - \frac{\pi \tau}{2}\right)}.
\end{equation}
Employing the second and the third of the identities \eqref{theta-ids}, we then obtain
\begin{equation}\label{f42-1}
f_{42}(k)
=
\frac{\theta_4\left(\pi \int_{iq_{o} }^k dw \right)}
{\theta_2\left(\pi \int_{iq_{o} }^k dw  \right)}.
\end{equation}
Observe that $f_{42}$ evaluated at $k=\infty_1$ is equal to the ratio $\theta_4(y)/\theta_2(y)$ that we wish to compute.

Furthermore, note that the function $f_{42}$ defined by equation \eqref{f42-def} is meromorphic on  $\Upsigma_1$, with a simple pole at $k=-iq_{o} $ and a simple zero at $k=\alpha$. Hence,  $f$ may be expressed in the form
\begin{equation}\label{f42-2}
f_{42}(k)
=
A_{42}(k) \frac{\left(k- \alpha\right)^{\frac 12}}{\left(k+iq_{o}\right)^{\frac 12}},
\end{equation}
where $A_{42}$ is a holomorphic function  bounded at $\infty_1$.
Now note that $f_{42}^2(k)$ is also a meromorphic function on the complex $k$-plane.  Hence $A_{42}^2(k)$ must be a constant by Liouville's theorem. 
Therefore, to determine the function $f_{42}$ [and hence the ratio $\theta_4(y)/\theta_2(y)$] it remains to determine the constant $A_{42}$. 

The computation of $A_{42}$ is done by evaluating the residue of $f_{42}$ at $k=-iq_{o} $ in two different ways. 
First, using the representation \eqref{f42-2} we find
\begin{equation}\label{res42-1}
\textrm{Res}\left[ f_{42}(k), -iq_{o} \right]
=
A_{42} \cdot i\left(iq_{o}  + \alpha\right)^{\frac 12}.
\end{equation}
Alternatively, employing the form \eqref{f42-1} and noting that 
$
\theta_2\left(\pi \int_{iq_{o} }^k dw  \right) = R(k) \left(k+iq_{o} \right)^{\frac 12}
$
for some function $R$ such that $R(-iq_o)\neq 0$, we obtain
\begin{equation}\label{res-temp}
\textrm{Res}\left[ f_{42}(k), -iq_{o} \right]
=
\frac{\theta_4\left(\pi \int_{iq_{o} }^{iq_{o} } dw \right)}
{R(-iq_{o} )}
=
\frac{\theta_4\!\left(\frac \pi 2\right)}
{R(-iq_{o} )}.
\end{equation}
Actually, we have
\begin{equation}
\left[R(-iq_{o} )\right]^2
=
\frac{\p}{\p k}\theta_2^2\left(\pi \int_{iq_{o} }^k dw  \right)\bigg|_{k=-iq_{o} }
=
\frac{\left[2\pi C\, \theta_2'\!\left(\frac \pi 2 \right)\right]^2}{-2iq_{o} \left(iq_{o} + \alpha\right)\left(iq_{o} +\bar \alpha\right)},
\label{resR}
\end{equation}
where the imaginary constant $C$ is defined by equation \eqref{abelmb}.
Thus, equation \eqref{res-temp} becomes
\begin{equation}\label{res42-2}
\textrm{Res}\left[ f_{42}(k), -iq_{o} \right]
=
\theta_4\!\left(\tfrac \pi 2\right)
\left[
-
\frac{2iq_{o} \left(iq_{o} +\alpha\right)\left(iq_{o} +\bar \alpha\right)}
{\left[2\pi C\, \theta_2'\left(\frac \pi 2 \right)\right]^2}
\right]^{\frac 12}.
\end{equation}
Matching the expressions \eqref{res42-1} and \eqref{res42-2}, we deduce
\begin{equation}
A_{42}^2
=
\frac{2iq_{o} \left(iq_{o} +\bar \alpha\right)\theta_4^2\!\left(\frac \pi 2 \right)}
{\left[2\pi C\, \theta_2'\!\left(\frac \pi 2 \right)\right]^2}.
\end{equation}
Finally, evaluating equations \eqref{f42-1} and \eqref{f42-2} at $k=\infty_1$ and using the identities \eqref{theta-ids}, we obtain
\begin{equation}\label{f42-50}
\frac{\theta_4^2(y)}{\theta_2^2(y)}
=
\frac{ 2iq_{o} \left(iq_{o} +\bar \alpha\right)\theta_3^2(0)}
{\left[2\pi C\, \theta_1'(0)\right]^2}.
\end{equation}
The constant $C$ defined by equation \eqref{abelmb} can be expressed in terms of the complete elliptic integral of the first kind $K(m)$ via the formula
 \begin{equation}\label{C-id}
C
 =
\frac{i\left|\alpha+iq_{o} \right|}{4K(m)}.
 \end{equation}
Thus, using also the identity
$
\theta_1'(0)
=
\theta_2(0)
\theta_3(0)
\theta_4(0)
$
we can write expression \eqref{f42-50} in the form
\begin{equation}\label{f42-5}
\frac{\theta_4^2(y)}{\theta_2^2(y)}
=
-\frac{ 2iq_{o} \left(iq_{o} +\bar \alpha\right)}
{m  \sqrt{1-m^2} \left|\alpha+iq_{o} \right|^2}.
\end{equation}

Computations identical to the above yield  the following expressions for the other two ratios in \eqref{ratios}:
\begin{equation}\label{f12-5}
 \frac{\theta_1^2(y)}{\theta_2^2(y)}
=
-\frac{\left(iq_{o} +\alpha\right)\left(iq_{o} +\bar \alpha\right)}
{ \sqrt{1-m^2}\left|\alpha+iq_{o} \right|^2 },
\qquad
\frac{\theta_3^2(y)}{\theta_2^2(y)}
=
-\frac{2iq_{o}  \left(iq_{o} +\alpha\right)}
{m\left|\alpha+iq_{o} \right|^2}.
\end{equation}
Inserting expressions \eqref{f42-5} and \eqref{f12-5} in equation
\eqref{ratios0} we  obtain
\begin{equation}\label{t14sq}
\frac{\theta_1^2(2y)}{\theta_4^2(2y)}
=
\frac{1}{m} \frac{4q_{o} \alpha_\im }{\left(q_{o} +\alpha_\im\right)^2}.
\end{equation}
Inserting in turn this expression into equation \eqref{ded0} yields expression \eqref{sqmod-finn-t} for the modulus of the leading order solution of the focusing NLS equation \eqref{e:NLS} in the modulated elliptic wave region.

The proof of Theorem \ref{mew-ell-t} is complete.


\section{Discussion and concluding remarks}
\label{conc-sec}

We have shown that,
for all initial conditions that satisfy the hypotheses of Theorem~\ref{pw-t}, the long-time asymptotics
decomposes the $xt$-plane into two plane wave regions (in each of which the solution is approximately equal to the background value up to a phase)
separated by a central region in which the leading-order behavior is given by a slow modulation of the traveling wave solutions of the focusing NLS equation.
Note that the spatial structure of the asymptotic solution (both in the plane wave regions and in the modulated elliptic wave region) is
independent of the initial conditions of the problem,
and that the initial conditions only determine the slowly varying offset $X$ of the elliptic solution (via the reflection coefficient),
whereas the envelope of the modulated elliptic wave is independent of it.
Thus,
the long-time asymptotics of generic localized perturbations of the constant background in modulationally unstable media on the infinite line displays universal behavior.
In this sense, \textit{the asymptotic stage of modulational instability is universal.}

Importantly, the results of the present work also show that, 
even though the jumps along the branch cut $B$ of the original Riemann-Hilbert problem \eqref{rhp0}  grow exponentially  with $t$, the solution of this Riemann-Hilbert problem --- and hence the solution of the focusing NLS equation --- remains bounded  in the limit $t\rightarrow \infty$. 
We note that this is a fairly common result of the analysis of Riemann-Hilbert problems via the Deift-Zhou method, 
as the factorizations of the jump matrices and the deformations of the jump contours allow one to eliminate all the terms that exhibit the exponential growth
in the jump conditions.

We also reiterate that a special case of the modulated elliptic wave
(without the slowly varying offset parameter $X$) had been previously obtained using Whitham theory in \cite{Kam-book}.
(The motion of the Riemann invariants had also been previously obtained in \cite{el1993}.)
Compared to those works, however, the results of the present results represent a significant step forward in that:
(i)~they establish rigorously the validity of the modulated elliptic solution as the long-time asymptotic state of the problem;
(ii)~they establish rigorously that the solution of the focusing NLS equation with nonzero boundary conditions at infinity
remains bounded for all times;
(iii)~they establish the universality of the modulated elliptic solution as the asymptotic state of a large class of perturbations
of the constant background.
At the same time, it is remarkable that the particular solution of the genus-1 Whitham system obtained in \cite{el1993} 
is relevant for the long-time asymptotics  
even though the Whitham equations for the focusing NLS equation are elliptic
(which implies that the corresponding initial value problem is ill-posed for generic initial data).

The results of this work open up a number of interesting problems, both from a mathematical and a physical point of view.
One of them is the generalization of the long-time asymptotics to initial data for which the reflection coefficient
is not analytic in a neighborhood of the continuous spectrum.
This situation corresponds to the case of initial conditions that tend to the contant background algebraically as $x\to\pm\infty$.
We expect that such a generalization can be achieved by employing rational approximations, along similar lines to what is done in the case of zero boundary conditions
\cite{dz1995}.

Another interesting problem
is the generalization of the long-time asymptotics to initial conditions which lead to the presence of a discrete spectrum.
Apart from the intrinsic interest of this problem from a mathematical point of view,
the analysis of such kinds of problems is also important from a physical point of view,
since it will allow one for the first time a study of the interactions between solitons and radiation
in a modulationally unstable medium.
The same framework will also allow one to address another related important open question,
concerning the stability of solitons on nonzero background in modulationally unstable media.


\section*{Appendix: estimation of the error}
\setcounter{equation}0
\def\theequation{A.\arabic{equation}}

%
We now establish estimate \eqref{m1e-est1} for the term $M^\err$ that appears in formula \eqref{qsolmea} for the leading order asymptotics in the plane wave region. 
Note that the proof also carries over to the estimate \eqref{m1e-est1r} for $M^\err$ in the modulated elliptic wave region 
once the appropriate parametrices are employed.
(See~\cite{bv2007} for the construction of such parametrices.)

Since $M^\err$ solves the Riemann-Hilbert problem \eqref{rhpE}, using Plemelj's formulae we may express the $O(1/k)$ coefficient
$M_1^{\err}$ in the large-$k$ expansion of $M^{\err}$ as
$$
M_1^{\err}(x, t)
=
-\frac{1}{2i\pi}\int_{\cup_j \check L_j\cup \p D_{k_1}^\varepsilon}
M^{\err-}(\nu)\left[V^\err(\nu)-I\right] d\nu.
$$
By the Cauchy-Schwarz inequality we then find
\begin{equation}\label{m1e-est0}
\left| M_1^{\err}(x, t)\right|
\leqslant
\frac{1}{2\pi}
\Big(
\!
\left\| M^{\err-}\! -I  \right\|_{L_k^2(\cup_j \check L_j\cup \p D_{k_1}^\varepsilon)}
\left\| V^\err\! -I \right\|_{L_k^2(\cup_j \check L_j\cup \p D_{k_1}^\varepsilon)}
+
\left\|V^\err\! -I  \right\|_{L_k^1(\cup_j \check L_j\cup \p D_{k_1}^\varepsilon)}
\!
\Big).
\end{equation}

We will estimate each of the terms involved in the right-hand side of the above inequality separately.
We begin with $\left\| V^\err-I \right\|_{L_k^p(\check L_j)}$.
Along  the contours $\check L_j$, we have
$
V^\err-I
=
M^B \big( V_j^{(4)}-I\big) (M^B)^{-1}.
$
Thus, we find
\begin{equation}
\left\| V^\err-I \right\|_{L_k^p(\check L_j)}
\leqslant
\left\| M^B \right\|_{L_k^\infty(\check L_j)}
\big\| V_j^{(4)}-I \big\|_{L_k^p(\check L_j)}
\left\| (M^B)^{-1} \right\|_{L_k^\infty(\check L_j)}.\label{veiest}
\end{equation}
Since $M^B$ is analytic away from the branch cut $B$ 
and $M^B=O(1)$ as $k \to \infty$, we have
$$
\left\| M^B \right\|_{L_k^\infty(\check L_j)}
\left\| (M^B)^{-1} \right\|_{L_k^\infty(\check L_j)}<  c<\infty, \quad c>0.
$$
Hence, inequality \eqref{veiest} becomes
\begin{equation}\label{veiest2}
\left\| V^\err-I \right\|_{L_k^p(\check L_j)}
\leqslant
c\,
\big\| V_j^{(4)}-I \big\|_{L_k^p(\check L_j)}, \quad c>0.
\end{equation}

The right-hand side of inequality \eqref{veiest2} can be estimated
by exploiting the exponential decay along $\check L_j$.
For example, for $j=2$ inequality \eqref{veiest2} reads
\begin{equation}\label{veiest3}
\left\| V^\err-I \right\|_{L_k^p(\check L_2)}
\leqslant
c
\begin{pmatrix}
0
&
0
\\
\left\| \frac{r e^{-2i(\theta t -g)} }{1+r\bar r}\,\delta^{-2}\right\|_{L_k^p(\check L_2)}
&
0
\end{pmatrix}.
\end{equation}
Recall that the function $\delta$ defined by equation \eqref{deldef} is nonzero, analytic away from $(-\infty, k_1]$ and $O(1)$ as $k \to \infty$. Similarly, the function $g$ defined by equation \eqref{gdef} is analytic away from $B$ and  $O(1)$ as $k \to \infty$. Thus, 
\begin{equation}
\Big\| \frac{r e^{-2i(\theta t -g)} }{1+r\bar r}\,\delta^{-2}\Big\|_{L_k^p(\check L_2)}
\leqslant
c\,
\Big\| \frac{r}{1+r\bar r}\, e^{-2i\theta t } \Big\|_{L_k^p(\check L_2)},
\quad c>0.
\end{equation}
Furthermore, since $\check L_2$ lies outside  the disk $D_{k_1}^\varepsilon$, we have
$\| \Re(i\theta)\|_{L^\infty(\check L_2)}>C>0.$
Hence, exploiting also the analyticity of the reflection coefficient $r$ provided by Lemma \ref{an-lem},
we find
\begin{equation}\label{veiest3b}
\left\| \dfrac{r}{1+r\bar r}\, e^{-2i\theta t }  \right\|_{L_k^p(\check L_2)}
<\tilde C e^{-Ct} \quad \forall t>T>0, \quad  C, \tilde C>0.
\end{equation}
Combining inequalities \eqref{veiest3}-\eqref{veiest3b} we then obtain the estimate 
$$
\left\| V^\err-I \right\|_{L_k^p(\check L_2)}
\leqslant
\tilde C
e^{-Ct} \quad \forall t>T>0,\quad  C, \tilde C>0.
$$
Similar estimates hold along the remaining contours. Overall, we have 
\begin{equation}\label{veiestj}
\left\| V^\err-I \right\|_{L_k^p(\check L_j)}
\leqslant
\tilde C
e^{-Ct}
 \quad \forall t>T>0, \quad  C, \tilde C>0, \quad j=1, 2, 3,4.
\end{equation}

Next, we  wish to estimate $  V^\err-I $ along the circle $\p D_{k_1}^\varepsilon$.
Recall, however, that now
\begin{equation}\label{jak}
V^\err=M^{\asymp-}(V_D^\asymp)^{-1}(M^{\asymp-})^{-1},
\end{equation}
where $V_D^\asymp$ is the \textit{unknown} jump  of $M^\asymp$ across  $\p D_{k_1}^\varepsilon$. Thus, we first need to estimate this unknown jump.

\begin{lemma*}
The jump $V_D^\asymp$ of $M^\asymp$ across the circle $\p D_{k_1}^\varepsilon$ admits the estimate
\begin{equation}\label{vdat}
V_D^\asymp
=
I
+
O\big(t^{-\frac 12}\big),
\quad t\to\infty.
\end{equation}
\end{lemma*}

\begin{proof}
Similarly to \cite{dz1993}, for $|k-k_1|\leqslant \varepsilon$ with $\varepsilon$ sufficiently small we take the Taylor series of the function $ \theta$ about the point $k_1$ to write
\begin{equation}\label{zdef2}
 \theta(k) t =  \theta(k_1) t+z^2,
\end{equation}
where 
\begin{equation}\label{zdef}
z 
=
\sqrt t \left(k-k_1\right) \Big(\sum_{n=0}^\infty \theta_{n+2} \left(k-k_1\right)^n\Big)^{\frac 12},
\quad \theta_n = \frac{\theta^{(n)}(k_1)}{n!}.
\end{equation}
Furthermore, as $k\in \overline{D_{k_1}^\varepsilon}$ we   can write
$
\left(
\sum_{n=0}^\infty \theta_{n+2} \left(k-k_1\right)^n
\right)^{\frac 12}
=
\sum_{n=0}^\infty \alpha_n \left(k-k_1\right),
$
$
\alpha_n\in\mathbb C, 
$
hence, equation \eqref{zdef} implies
\begin{equation}\label{zapp}
\frac{z}{\sqrt t}
=
\sum_{n=1}^\infty \alpha_{n-1}\left(k-k_1\right)^n,\quad k\in \overline{D_{k_1}^\varepsilon}.
\end{equation}
Moreover, for $\varepsilon$ sufficiently small equation
\eqref{zapp} can be inverted via recursive approximations to yield  
$$
k=k_1+\sum_{n=1}^\infty \beta_n \big(\tfrac{z}{\sqrt t}\big)^n,\quad k\in \overline{D_{k_1}^\varepsilon}.
$$

Next, we let
\begin{equation}\label{mpdefb}
M^D( k)
=
M^B( k) \ m^D\Big(\sqrt t \sum_{n=1}^\infty \alpha_{n-1} \left(k-k_1\right)^n\Big),\quad k\in D_{k_1}^\varepsilon,
\end{equation}
which equivalently reads
\begin{equation}\label{mpdef}
m^D(z)
=
\left(M^B\right)^{-1}
\!
\Big(
k_1+\sum_{n=1}^\infty \beta_n \big(\tfrac{z}{\sqrt t}\big)^n
\Big)
\
M^D
\Big(k_1+\sum_{n=1}^\infty \beta_n \big(\tfrac{z}{\sqrt t}\big)^n\Big),\quad z\in D_0^\varepsilon.
\end{equation}
Note that the presence of $M^B$
does not affect the jump conditions of $M^D$ inside the disk $D_{k_1}^\varepsilon$ since  $M^B$ is analytic there.
Furthermore, the mapping 
\begin{equation}\label{map-conf}
\overline{D_{k_1}^\varepsilon}\ni k\mapsto z = \sqrt t \sum_{n=1}^\infty \alpha_{n-1} \left(k-k_1\right)^n
\in \overline{D_0^\varepsilon}
\end{equation}
is conformal and, as such, it preserves angles locally. Hence, there exists an appropriate choice of the  jump contours $\hat L_j$ that  lie inside the disk $D_{k_1}^\varepsilon$, which is mapped under \eqref{map-conf} to certain contours $\hat \Sigma_j$ (see Figure \ref{k_to_z}) that lie inside the disk $D_0^\varepsilon$ and  form a cross centered at $z=0$ such that its extension does not intersect with the branch cut $B$.
\begin{figure}
\begin{center}
\includegraphics[scale=.9]{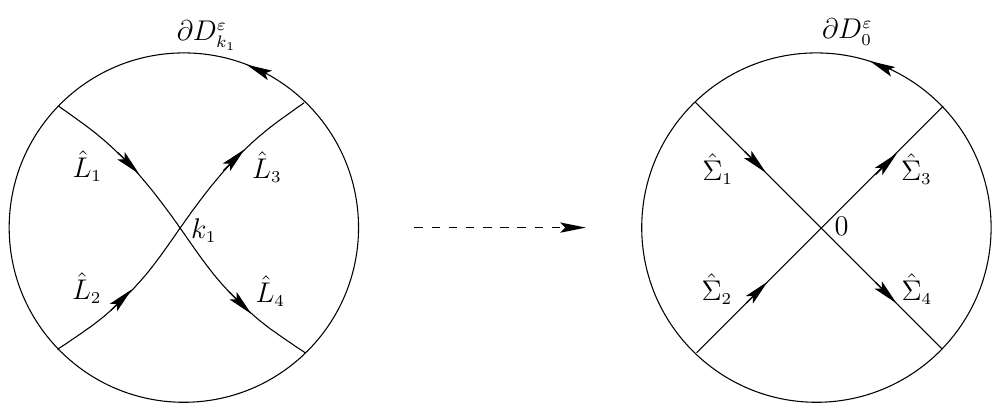}
\caption{The transformation from the contours $\hat L_j$ to the contours $\hat \Sigma_j$, $j=1,2,3,4$.}
\label{k_to_z}
\end{center}
\end{figure}
Thus, for $k\in D_{k_1}^\varepsilon$ we infer from equations \eqref{rhpP} and \eqref{mpdef} that the function $m^D(z)=m^D(x, t, z)$ is analytic in $D_{0}^{\varepsilon}\setminus \cup_{j=1}^4 \hat \Sigma_j  $ and satisfies the jump conditions
\begin{equation}
m^{D+}(z)
=
m^{D-}(z)
\
V_j^{(4)}
\Big(k_1+\sum_{n=1}^\infty \beta_n \big(\tfrac{z}{\sqrt t}\big)^n\Big),\quad z\in \hat \Sigma_j,\ j=1, 2,3,4.\label{tsipf}
\end{equation}
Extending the jump conditions \eqref{tsipf} in the whole complex $k$-plane implies that $m^D$ is analytic in $\mathbb C\setminus  \cup_{j=1}^4 \Sigma_j$ and satisfies the jump conditions
\begin{equation}
m^{D+}(z) = m^{D-}(z) \,
V_j^{(4)} \Big(k_1+\sum_{n=1}^\infty \beta_n \big(\tfrac{z}{\sqrt t}\big)^n\Big), \quad z\in \Sigma_j,\, j=1,2,3,4,
\label{rhpPb}\\
\end{equation}
and the normalization condition
$
m^D(z) = I +O \left(1/z\right),$ $z\to \infty,
$
where the contours $\Sigma_j$ are the extensions of the contours $\hat \Sigma_j$ outside the disk $D_0^\varepsilon$.

Since $r(k)=O(1/k)$ as $k\to \infty$ (see \cite{bk2014}), we may now proceed similarly to   \cite{dz1993}.
Eventually, we find 
$$
m^D(x, t, z)
=
 \delta_0^{\hat \sigma_3}\, \tilde m_\infty^D(x, z)+ O\big(t^{-\frac 12}\ln t\big), \quad t\to \infty,
$$
where for
$
\upnu = -\frac{1}{2\pi}\, \ln\left(1+|r(k_1)|^2\right)
$
and
$
\upchi(k)=-\frac{1}{2i\pi}\int_{-\infty}^{k_1} \ln(k-\nu) \ d\left[\ln\left(1+|r(\nu )|^2\right)\right]
$
we define
$
\delta_0
=
\beta_1^{i\upnu} t^{-\frac{i\upnu}{2}}e^{\upchi(k_1)}e^{i \theta(k_1) t-ig(k_1)}.
$
Moreover, the $t$-independent function $\tilde m_\infty^D(z)=\tilde m_\infty^D(x,  z)$ is analytic in $\mathbb C\setminus \cup_{j=1}^4 \Sigma_j $ and satisfies the jump conditions
$$
\tilde m_\infty^{D+}(z) = \tilde m_\infty^{D-}(z)\,
V_j^{\infty}(z),\quad z\in \Sigma_j,\, j=1,2,3,4,
$$
and the normalization condition
$
\tilde m_\infty^D(z) = I +O \left(1/z\right),$ $z\to \infty,
$
with
\begin{align*}
&V_1^{\infty}
=
\begin{pmatrix}
1
&
\frac{\bar r(k_1) }{1+\left|r(k_1)\right|^2}\, z^{2i\nu}e^{2iz^2}
\\
0
&
1
\end{pmatrix},
\quad
V_2^{\infty}
=
\begin{pmatrix}
1
&
0
\\
\frac{r(k_1)}{1+\left|r(k_1)\right|^2}\, z^{-2i\nu}e^{-2iz^2}
&
1
\end{pmatrix},
\nonumber\\
&
V_3^{\infty}
=
\begin{pmatrix}
1
&
0
\\
r(k_1) \, z^{-2i\nu}e^{-2iz^2}
&
1
\end{pmatrix},
\hskip .55cm
V_4^{\infty}
=
\begin{pmatrix}
1
&
\bar r(k_1)\, z^{2i\nu}e^{2iz^2}
\\
0
&
1
\end{pmatrix}.
\end{align*}
Thus, since $M^{\asymp +}=M^D$ and $M^{\asymp -}=M^B$ for $k\in \p D_{k_1}^\varepsilon$, using equation \eqref{mpdef} we find
$$
V_D^\asymp = m^D\Big(\sqrt t \sum_{n=1}^\infty \alpha_{n-1} \left(k-k_1\right)^n\Big),\quad k\in \p D_{k_1}^\varepsilon.
$$
Then, since   $z\to\infty$ in the limit $t\to \infty$
for $k\in \p D_{k_1}^\varepsilon$, we  use the large-$z$ expansion of $m^D(z)$ evaluated at $\sqrt t \sum_{n=1}^\infty \alpha_{n-1} \left(k-k_1\right)^n$ to obtain
$$
V_D^\asymp
=
I
+
\frac{m_1^D}{\sqrt t \sum_{n=1}^\infty \alpha_{n-1} \left(k-k_1\right)^n}
+
O\bigg(\Big[\sqrt t \sum_{n=1}^\infty \alpha_{n-1} \left(k-k_1\right)^n\Big]^{-2}\bigg),
\quad t\to\infty,
$$
from which we infer estimate \eqref{vdat}
\end{proof}

Returning to equation \eqref{jak} we have 
\begin{equation} 
\left\| V^\err-I \right\|_{L_k^p(\p D_{k_1}^\varepsilon)}
\leqslant
\left\| V^\err-I \right\|_{L_k^\infty(\p D_{k_1}^\varepsilon)}
\bigg(\int_{k\in D_{k_1}^\varepsilon} \left|dk\right|\bigg)^{\frac 1p}
\leqslant
\left(2\pi \varepsilon\right)^{\frac 1p}
O\big(t^{-\frac 12}\big).\label{veiestd}
\end{equation}
Combining inequalities \eqref{veiestj} and \eqref{veiestd}  we thereby obtain the estimate
\begin{equation}\label{veiest-t}
\left\| V^\err-I \right\|_{L_k^p(\cup_j\check L_j\cup \p D_{k_1}^\varepsilon)}
\leqslant
O\big(t^{-\frac 12}\big), \quad t\to \infty.
\end{equation}

Returning to inequality \eqref{m1e-est0}, it remains to estimate $M^{\err-}-I$. This is done by using standard estimates that involve  the Cauchy transform (see \cite{bv2007}). 
Eventually, we find
\begin{equation}
\left\| M^{\err-}-I\right\|_{L_k^2(\cup_j\check L_j\cup \p D_{k_1}^\varepsilon)}
\leqslant
C
\left\|
 V^\err-I 
\right\|_{L_k^2(\cup_j\check L_j\cup \p D_{k_1}^\varepsilon)},
\quad C>0.\label{me-est}
\end{equation}
Inequalities \eqref{m1e-est0}, \eqref{veiest-t} and \eqref{me-est} 
 complete the proof of estimate \eqref{m1e-est1}.

%
%
%
%
%
%
%
%
%
\paragraph{Acknowledgements.} 
The authors thank 
M.~J. Ablowitz,
P. Deift, 
G. El, 
R. Jenkins, 
B.~Prinari
and M. Ramachandran
for useful discussions and comments.  
This work was partially supported by the National Science Foundation under grant number DMS-1311847.
D.M. also acknowledges an AMS-Simons Travel Grant.
%


\vskip 10mm
\noindent{\bf Gino Biondini}\\
Department of Mathematics\\
State University of New York at Buffalo\\
Buffalo, NY 14260\\
e-mail: biondini@buffalo.edu

\flushright{
\vskip-1.02in
\noindent{\bf Dionyssios Mantzavinos}\\
Department of Mathematics\\
State University of New York at Buffalo\\
Buffalo, NY 14260\\
e-mail: dionyssi@buffalo.edu}

\end{document}